\documentclass[10pt]{article}
\usepackage[T1]{fontenc}
\usepackage{amsmath,amssymb,amsfonts,amsthm,textcomp,graphicx,nicefrac,mathtools,mathrsfs, dsfont} 
\usepackage{tikz}
\usepackage[left = 0.9in, top = 1in, bottom = 1in, right = 1in]{geometry}
\usepackage{enumitem}
\usepackage{bbm}
\usepackage[utf8]{inputenc}
\usepackage{enumitem}
\usepackage{nicefrac}
\usepackage{algorithm}
\usepackage[noend]{algpseudocode}
\usepackage{wrapfig}

\def\<{{\langle}}
\def\>{{\rangle}}

\def\reals{\mathbb{R}} %
\renewcommand{\exp}[1]{\operatorname{exp}\left(#1\right)} %
\providecommand{\argmax}{\mathop\mathrm{arg\, max}} %
\providecommand{\argmin}{\mathop\mathrm{arg\, min}}

\usepackage{setspace}

\newtheoremstyle{dotless}{}{}{\itshape}{}{\bfseries}{}{ }{}
\theoremstyle{dotless}

\newtheorem*{defi}{Definition}
\theoremstyle{plain}

\newtheorem{example}{Example}

\newtheorem{myth}{Theorem}

\newtheorem{myprop}[myth]{Proposition}
\newtheorem{mylem}[myth]{Lemma}
\newtheorem{corr}[myth]{Corollary}

\newtheorem*{myth*}{Theorem}
\newtheorem*{corr*}{Corollary}

\newtheoremstyle{named}{}{}{\itshape}{}{\bfseries}{.}{.5em}{#1 #3}
\theoremstyle{named}
\newtheorem*{namthm*}{Theorem}

\usepackage{chngcntr}

\usepackage[colorlinks=true, citecolor = blue]{hyperref}
\usepackage{cleveref}
\crefname{myth}{Theorem}{Theorems} 

\newcounter{parentnumber}

\usepackage[
backend=biber,
style=alphabetic,
sorting=anyt,
maxbibnames=99,
backref,
maxalphanames=4
]{biblatex}
\newcommand{\fhull}{\operatorname{f-conv}}
\newcommand{\cfhull}{\overline{\fhull}}
\newcommand{\lc}{\mathcal{L}}
\newcommand{\lciidseq}{\lc^\infty}
\newcommand{\seqG}{\mathtt{\Gamma}}

\newcommand{\lcproj}{\mathfrak{L}}

\usepackage{enumitem}
\usepackage{graphicx}

\renewcommand{\epsilon}{\varepsilon}
\newcommand{\lac}{\ll_{\mathrm{loc.}}}
\newcommand{\denp}[1]{Z^{\mathtt{#1}}}

\newcommand{\fccomb}[4]{ \left[ #1 \overset{#3, #4}{\longrightarrow} #2 \right] }

\newcommand{\indi}{\mathds{1}}
\newcommand{\leb}{\mathrm{Leb}}
\newcommand{\boundparam}{B}
\newcommand{\batchinterval}{I}
\usepackage{csquotes}

\begin{document}
\title{A Sequential Test for Log-Concavity}

\author{%
  Aditya Gangrade$^1$, Alessandro Rinaldo$^{1}$ and Aaditya Ramdas$^{12}$ \\
  \texttt{agangra2@andrew.cmu.edu, arinaldo@cmu.edu, aramdas@cmu.edu} \vspace{0.1in}\\
   $^1$Department of Statistics and Data Science, Carnegie Mellon University\\
   $^2$Machine Learning Department, Carnegie Mellon University\\
}
\date{}

\maketitle

\begin{abstract}
   On observing a sequence of i.i.d.\ data with distribution $P$ on $\reals^d$, we ask the question of how one can test the null hypothesis that $P$ has a log-concave density.
   This paper proves one interesting negative and positive result: the non-existence of test (super)martingales, and the consistency of universal inference. To elaborate, the set of log-concave distributions $\lc$ is a nonparametric class, which contains the set $\mathcal G$ of all possible Gaussians with any mean and covariance. Developing further the recent geometric concept of fork-convexity, we first prove that there do no exist any nontrivial test martingales or test supermartingales for $\mathcal G$ (a process that is simultaneously a nonnegative supermartingale for every distribution in $\mathcal G$), and hence also for its superset $\lc$. Due to this negative result, we turn our attention to constructing an e-process --- a process whose expectation at any stopping time  is at most one, under any distribution in $\lc$ --- which yields a  level-$\alpha$ test by simply thresholding at $1/\alpha$. We take the approach of universal inference, which avoids intractable likelihood asymptotics by taking the ratio of a nonanticipating likelihood over alternatives against the maximum likelihood under the null. Despite its conservatism, we show that the resulting test is consistent (power one), and derive its power against Hellinger alternatives. To the best of our knowledge, there is no other e-process or sequential test for $\lc$.
\end{abstract}

\section{Introduction}

Log-concavity is an important and prevalent modelling assumption in the modern study of shape-constrained nonparametrics \cite{samworth2018recent}. Log-concave distributions include many common families of densities, including normal, exponential, extreme-value, and logistic distributions, and further are frequently justified in diverse application domains including economics, reliability theory and filtering in engineering, and survival analysis in medicine \cite{bagnoli2006log}. At the same time, the family is technically amenable, and admits a unique maximum likelihood estimate with a well developed minimax theory and computationally efficient estimators \cite{cule2010theoretical, carpenter2018near, kur2019optimality, axelrod2019polynomial, dumbgen2011logcondens, rathke2019fast,cule2010maximum}. As a result, log-concave densities offer practitioners a broadly applicable and usable structure.

Given the attractive properties of estimation within the log-concave family, tests for membership in the same are an important and necessary line of investigation. We note that along with the applications mentioned above, such tests also have theoretical interest; for instance, in much of computational learning theory, efficient learning algorithms are only known when covariates are sampled according to a log-concave distribution \cite[e.g.][]{kalai2008agnostically}. While the estimation of log-concave densities has seen significant advances over the past decade or two (see, e.g., the citations above, and the survey by Samworth \cite{samworth2018recent}), testing for log-concavity has been relatively poorly developed. Indeed, prior to 2021, there were no valid and powerful tests for the same---both theoretically and practically---outside of certain restricted one-dimensional settings. In a significant development, recent work of Dunn et al. \cite{dunn2021-logConcaveTesting} has developed such a test, based on the Universal Inference strategy of Wasserman et al. \cite{wasserman2020universal}. 

Our work is concerned with testing log-concavity in a sequential setting. Concretely, we assume that we are given streaming access to a sequence $\{X_t\}$ that are drawn independently and identically from some $d$-dimensional density $p$, and we wish to test the membership of $p$ within the family of log-concave densities. Such a sequential test can be identified with a stopping time $\tau$, where stoppage indicates rejection of the null hypothesis, and the test is $\alpha$-valid if under the null, the probability that $\tau < \infty$ is bounded by $\alpha.$ The principal attractiveness of such sequential tests arises from their adaptivity: rather than fixing a number of samples a priori, the test may adapt to the difficulty of the underlying instance, rejecting earlier in easier settings, and allowing for a greater number of samples to detect subtle deviations from the null hypothesis. 

Below, we first set up some notation, and then proceed to contextualise our study, and give a brief overview of the contributions of our paper. 

\subsection{Problem setup and background}

We begin by describing the notation needed for our discussion, the testing problem under consideration, and the fundamental notions of test martingales and e-processes. We shall give further definitions and details in \S\ref{sec:def}, as well as later in the text as the context arises.

\noindent \emph{Spaces and measures.} Let $\{X_t\} = (X_1, X_2, \dots)$ denote a sequence of $d$-dimensional random vectors with entries indexed by $t$, which are measurable maps from $\Omega:= (\mathbb{R}^d)^{\mathbb{N}}$ to $\mathbb{R}^d,$ endowed with the cylindrical Borel sigma-algebra $\mathscr{B}(\mathbb{R}^d)^{\mathbb{N}}$. We use typewriter style fonts, e.g.\ $\mathtt{P},$ to denote laws of random processes (i.e. probability measures on $(\Omega, \mathscr{B}(\mathbb{R}^d)^{\mathbb{N}})$), and standard fonts, e.g.\ $P$ to denote laws on $(\mathbb{R}^d, \mathscr{B}(\mathbb{R}^d))$. We use $\mathscr{F} = \{ \mathscr{F}_t \}$ to denote the natural filtration of the process $\{X_t\}$, where $\mathscr{F}_t := \sigma(X_1,\dots,X_t)$, for each $t$. For a Borel probability measure $P$ on $\mathbb{R}^d$, we use $P^{\infty}$ to denote the law of an i.i.d.~process drawn according to $P$. We use $\mathcal{D}$ to denote the set of probability measures on $\mathbb{R}^d$ with Lebesgue densities, and $\mathcal{D}^\infty = \{ P^\infty: P \in \mathcal{D}\}$. For $P \in \mathcal{D},$ we use $p$ to denote its Lebesgue density. For technical convenience we define $\mathcal{D}_1 := \{ P \in \mathcal{D}: \mathbb{E}[ \max(0,\log p(X))] < \infty, \mathbb{E}[\|X\|] < \infty\}$. A set of laws $\mathcal{P}$ is said to be mutually absolutely continuous (m.a.c.) if for all $P,Q\in \mathcal{P}, P \ll Q \ll P$. Finally, we frequently use $X_1^t := (X_1, \dots, X_t)$ to denote finite prefixes of $\{X_t\}$. 

\noindent \emph{Log-concave measures.} A function $f : \mathbb{R}^d \to \mathbb{R}$ is said to be log-concave if there exists a concave function $g$ such that $f = e^g$. If $f$ is further a density with respect to the Lebesgue measure, then it is said to be a log-concave density. We denote the set of measures with log-concave densties as $\lc,$ and use $\lciidseq$ to denote the set of i.i.d. log-concave measures on euclidean sequences, i.e.\ $\lciidseq:= \{ P^\infty : P \in \lc\}.$ 

\paragraph{Sequential test for log-concavity.} The testing problem of interest is formulated as follows: Let $\smash{X_t \overset{\mathrm{i.i.d.}}\sim P}$ for some unknown $P \in \mathcal{D}.$ We wish to test the null hypothesis $H_0 : P \in \lc$. 

A sequential test corresponds to $\{\mathscr{F}_t\}$-adapted stopping time, representing the (possibly infinite) time at which the test stops and rejects the null hypothesis. We shall refer to this stopping time as the rejection time of the sequential test. A test is said to be $\alpha$-valid if its rejection time, $\tau$ satisfies that \[ \sup_{P \in \lc} P^{\infty}(\tau < \infty) \le \alpha, \]
meaning that, under the null, the probability of ever rejecting, i.e. of incurring a Type I error, is at most $\alpha$.
Similarly, a test is said to be asymptotically $(1-\beta)$-powerful against $\mathcal{Q} \subset \mathcal{D} \setminus \lc$ if the probability of  failing to reject the null under any distribution in the alternative $\mathcal{Q}$ (also know as a type II error) is uniformly bounded by $\beta$: 
\[ \inf_{Q \in \mathcal{Q}} Q^{\infty}(\tau = \infty) \le \beta.\] 
A test is said to be consistent against $\mathcal{Q}$ if it is asymptotically $1$-powerful against the same. Note that, when consistent, these tests are typically called `power-one tests' (following Robbins) to differentiate them from the traditional Waldian sequential testing paradigm for which stopping does not imply rejection of the null.

\paragraph{Test martingales, test supermartingales and e-processes.} We briefly survey key notions underlying our discussion, namely test martingales, and e-processes, leaving  details to \S\ref{sec:triviality} and \S\ref{sec:power} respectively. 

\begin{defi}
A process $\{M_t\}$ is a \emph{nonnegative supermartingale (NSM)} with respect to a filtration $\{\mathscr{F}_t\}$ and a law $\mathtt{P}$ if it is adapted, nonnegative, and $\mathbb{E}_{\mathtt{P}}[M_t | \mathscr{F}_{t-1}] \le M_{t-1}$ for each $t$. If the inequality is further an equality at each $t$, then $\{M_t\}$ is a \emph{nonnegative martingale (NM)}. We shall succinctly say that such a process is a $\mathtt{P}$-NSM or $\mathtt{P}$-NM respectively.
\end{defi}

Obviously, every $\mathtt{P}$-NM is also a $\mathtt{P}$-NSM. 
An important basic inequality of Ville~\cite{ville1939etude} controls the tail behaviour of NSMs: if $\{M_t\}$ is a $\mathtt{P}$-NSM such that $M_0 =1,$ then for every $\alpha \in (0,1],$ \[ \mathtt{P}( \exists t \ge 1: M_t \ge 1/\alpha) \le \alpha. \] The result above is a sequential (time-uniform) analogue of Markov's inequality. Equivalently, one can make claims at \emph{arbitrary} stopping times: for all stopping times $\tau$, $\mathtt{P}(M_\tau \ge 1/\alpha) \le \alpha$. This can be seen by applying the optional stopping theorem for NSMs \cite[Ch.~V, Thm.~28]{meyer1966probability} and Markov's inequality.

We now extend the above notions to \emph{composite families} of sequential laws. Throughout this paper we shall take the filtration to be the natural filtration of the data, and will leave it implicit in our definitions below.
\begin{defi}
For a set of sequential laws $\mathfrak{P},$ we say that a process $\{M_t\}$ is a \emph{$\mathfrak{P}$-NSM} if $\{M_t\}$ is a $\mathtt{P}$-NSM for every $\mathtt{P} \in \mathfrak{P}$. Similarly, $\{M_t\}$ is a $\mathfrak{P}$-NM if it is a $\mathtt{P}$-NM for every $\mathtt{P} \in \mathfrak{P}.$ A $\mathfrak{P}$-NSM such that $M_0 = 1$ is called a \emph{test supermartingale} for $\mathfrak{P}$, and a $\mathfrak{P}$-NM such that $M_0= 1$ is called a \emph{test martingale} for $\mathfrak{P}$.
\end{defi}

Observe that test supermartingales satisfy Ville's inequality for each $\mathtt{P} \in \mathfrak{P}$, i.e., if $\{M_t\}$ is a test supermartingale for $\mathfrak{P}$, then for every $\alpha \in (0,1],$ \begin{equation} \label{eq:villes_ineq} \forall \mathtt{P} \in \mathfrak{P}, ~\mathtt{P}( \exists t \ge 1 : M_t \ge 1/\alpha) \le \alpha. \end{equation}
Test supermartingales are so named because they form the canonical path to sequentially testing composite hypotheses, which is encapsulated entirely by the above relation, in that valid tests can be derived by rejecting only when a test supermartingale crosses a threshold. 
They are particularly interesting in nonparametric settings; for example one can use them to sequentially test the mean of a bounded random variable~\cite{waudby2020estimating}, for testing symmetry~\cite{ramdas2020admissible}, for two-sample testing~\cite{shekhar2021game}, independence testing~\cite{podkopaev2022sequential}, and testing calibration~\cite{arnold2021sequentially}, to mention only a few interesting sequential nonparametric problems.
We shall discuss test supermartingales extensively in this paper.

E-Processes \cite{ramdas2022game,ramdas2020admissible, grunwald2019safe, howard2020time} are a recently defined class of processes that will also play a central role in this paper. 
\begin{defi}
    A process $\{ E_t \}$  is called an e-process with respect to a sequential law $\mathtt{P}$ if it is non-negative, and for every stopping time $\tau,$ we have \( \mathbb{E}_{\mathtt{P}}[E_\tau] \le 1.\) Similarly, $\{ E_t \}$ is an e-process for a class of sequential laws $\mathfrak{P}$ if it is an e-process with respect to every $\mathtt{P} \in \mathfrak{P}$. 
\end{defi} E-Processes have a variety of equivalent definitions \cite[Lem.~6]{ramdas2020admissible}. In particular it is sufficient for the process to satisfy $\mathbb{E}_{\mathtt{P}}[E_\tau] \le 1$ for only bounded stopping times. 

By the optional stopping theorem (which holds without restriction on stopping times for nonnegative supermartingales), notice that every test supermartingale for a class $\mathfrak{P}$ is also an e-process for this class. Thus, e-processes generalise the notion of 
 test supermartingales. We observe that a Ville-type relation also holds for e-processes, simply due to Markov's inequality: if $E_t$ is an e-process for $\mathfrak{P},$ then for every $\alpha \in (0,1],$ \begin{equation} \label{eq:e-process_inequality}  \forall \mathtt{P} \in \mathfrak{P},  \textrm{ stopping times } \tau, 
 ~ \mathtt{P}(E_\tau \ge 1/\alpha) \le \alpha \mathbb{E}[E_\tau] \le \alpha. \end{equation} 

Much as Ville's inequality over the class (\ref{eq:villes_ineq}) captures the relevance of test supermartingales to sequential testing, the above inequality captures the relevance of e-processes to the same. The notion of e-processes, along with the non-sequential analogue of e-values, is gaining vogue in recent work in statistics due to this key property, along with the fact that e-processes exist for many composite and nonparametric testing problems for which test supermartingales do not exist (see, e.g., the recent survey by Ramdas et al. \cite{ramdas2022game}). We will also encounter this situation in the current paper.

It is important to note that test supermartingales or e-processes can directly be interpreted as evidence against the null hypothesis: since we expect them to be less than one under the null, the larger their realized value, the more evidence we have that the null hypothesis is wrong. Thus, there is no explicit need to threshold them at $1/\alpha$ for some prespecified $\alpha$; one can alternatively simply report the final value at the final stopping time of the experiment (which can itself be arbitrarily chosen). Nevertheless, we present this paper in the language of level-$\alpha$ tests because that is far more popular, and we refer the interested reader to the aforementioned references for further discussion on e-processes.

\subsection{Inadequacy of Test (Super)Martingales, and the Power of E-Processes} 

One dominant (but sometimes hidden) principle behind sequential testing of composite hypotheses is the use of nonnegative martingales (NMs), or nonnegative supermartingales (NSMs). Concretely, to test a composite hypothesis $\mathtt{P} \in \mathfrak{P},$ 
one attempts to construct a $\mathfrak{P}$-test supermartingale $\{M_t\}$, which was defined earlier. By Ville's inequality (\ref{eq:villes_ineq}), the chance that $M_t$ ever exceeds $1/\alpha$ under \emph{any} null law is bounded by $\alpha$. Thus, these test supermartingales immediately yield a valid test: reject the null when $M_t \ge 1/\alpha.$ The associated rejection time, of course, is the $M_t$-hitting time of $1/\alpha$. Such tests have game-theoretic interpretations, through the fact that nonnegative (super)martingales represent wealth processes in betting games \cite{ramdas2022game}. For example, a $\mathfrak{P}$-test martingale is the wealth process of a gambler who bets against the hypothesis that the sequence $\{X_t\}$ is drawn according some law in $\mathfrak{P}$. The game is designed so that the gambler cannot  hope to reliably (in expectation) make money if the null hypothesis is true; this is imposed by a restriction that under any law in $\mathfrak{P}$, the expected wealth multiplier in each round should be at most unity.

However, for sufficiently rich classes $\mathfrak{P},$ such a game leaves the gambler powerless; the gambler is so constrained by the aforementioned restriction that the only option is to not bet at all (or throw away money). This phenomenon was first observed in work on testing exchangability in discrete time binary processes by Ramdas et al. \cite{ramdas2022testingExchangability}, who demonstrated that any process $\{M_t\}$ that is an NSM for all exchangable binary laws is, almost surely, a strictly decreasing process (the wealth starts at one and can only possibly go down).  As a result, any test based on thresholding such processes must be powerless against any alternative.  Our first technical contribution demonstrates an anlogous phenomenon in the setting of log-concave distributions. Specifically, we show that the smaller class of i.i.d.\ Gaussian processes is not testable using NMs (or NSMs), since all such processes are trivial in the sense of being almost surely constant (or decreasing). The claim is summarised below, where $\mathcal{G}^\infty$ denotes the set of all i.i.d.\ Gaussian laws (of any mean and variance).

\begin{myth*}
     \emph{(Informal)} There are no nontrivial $\mathcal{G}^\infty$-NSMs or $\mathcal{G}^\infty$-NMs. A fortiori, there are also no nontrivial $\lc^\infty$-NSMs or $\lc^\infty$-NMs.
\end{myth*}

Thus, log-concave densities represent a natural class of distributions that cannot be tested via martingales.

\paragraph{Testing via E-Processes.} Given that one cannot test for log-concavity (or indeed, Gaussianity) using nonegative (super)martingales, we are left in a situation where the prevalent design paradigm for sequential testing is neutralised. There are two contrasting lines of attack that can be employed instead. 

The first of these involves designing a restricted filtration $\mathscr{G}_t$, distinct from the natural filtration, under which there might exist nontrivial test supermartingales. Ramdas et al.~\cite{ramdas2022testingExchangability,ramdas2022game} highlight the remarkable 
 fact that shrinking a filtration could introduce new nontrivial (composite) test martingales when none existed in the original filtration. Such a strategy was notably used by Vovk et al.~\cite{vovk2003testing, fedorova2012plug} to develop a sequential test for exchangeability, where as mentioned above, no nontrivial test supermartingales exist in the data filtration. 
There are two main disadvantages to such an approach. First, such test martingales only yield an e-process for a restricted set of stopping times (those under the restricted filtration).
From an applied point of view, the use of such an e-process demands discipline from a practitioner---they cannot look at the raw data to decide when to adaptively stop (a predefined stopping rule, like the hitting time of $1/\alpha$ is okay, but it may never be reached, in which case we may still wish to present the obtained evidence at the stopping time). 
Second, from a design point of view, the construction of appropriate filtrations is itself a subtle task that is heavily problem-dependent, and thus designing such tests is more of an art than a science. In particular, no such construction is known or obvious for sequential log-concavity testing.

In contrast, we follow the alternative strategy of testing via an e-process. Recall that a process $\{E_t\}$ is an e-process for a set of sequential laws $\mathfrak{P}$ if, for every stopping time $\tau$ and every $\mathtt{P} \in \mathfrak{P}, \mathbb{E}_{\mathtt{P}}[E_\tau] \le 1.$ Such processes bear a deep relationship to the aforementioned test martingales. Indeed, it has been argued that (admissible) e-processes must take the form $\inf_{\mathtt{P} \in \mathfrak{P}} M_t^{\mathtt{P}},$ where each $\{M_t^{\mathtt{P}}\}$ is a $\mathtt{P}$-NM \cite{ramdas2020admissible}. The same observation lends e-processes a gambling interpretation as the wealth process of a gambler against a `family of games', wherein the gambler simultaneously plays a game against each $\mathtt{P} \in \mathfrak{P},$ and their wealth is taken as the smallest wealth amongst these games. The gambler can then make money only if each of these games makes money, i.e., if $\forall \mathtt{P} \in \mathfrak{P}, M_t^{\mathtt{P}}$ grows without bound, which would then indicate that every $\mathtt{P} \in \mathfrak{P}$ can be rejected.

E-Processes offer a similar testing approach as the previously discussed test supermartingales, as elucidated by the inequality (\ref{eq:e-process_inequality}). Indeed, given an e-process $\{E_t\}$ for $\mathfrak{P},$ we can construct an $\alpha$-valid test of membership in $\mathfrak{P}$ by rejecting only if $E_t \ge 1/\alpha$. Indeed, in this case, the rejection time is \[ \tau_\alpha := \inf\{t \ge 1: E_t \ge 1/\alpha\}, \] and using the inequality (\ref{eq:e-process_inequality}), we may conclude that \[  \forall \mathtt{P} \in \mathfrak{P}, \mathtt{P}(\tau_\alpha < \infty) \le \alpha, \] i.e. this test is valid for the composite null $\mathfrak{P}$. Note further that the validity extends beyond this: let $\sigma$ be any other stopping time with respect to the natural filtration of the data. We further have that $\mathtt{P}(E_{\sigma} \ge 1/\alpha) \le \alpha,$ and thus no extraneous stopping criterion can affect the validity of the test, as long as rejection occurs only if $E_\sigma \ge 1/\alpha$. 

The theory and applications of e-processes have seen considerable development in the recent literature on sequential analysis (along with the more basic notion of e-variables in batched settings). 
The concept is attractive thanks to its flexibility and simplicity (despite generalizing nonnegative martingales), but constructing powerful e-processes is partly science and partly art \cite{ramdas2022game}. In  composite testing,  e-processes are of central importance since they do not encounter the same pitfalls as NSMs and NMs, and there do indeed exist nontrivial e-processes even on classes where no such NSMs exist. Indeed, in some sense, e-processes can be shown to lie at the very core of sequential composite testing~\cite{ruf2022composite}.

\subsection{Test Using Universal Likelihood Ratios: A simple E-Process} The universal inference strategy  \cite{wasserman2020universal} gives a simple and generic construction of e-processes when a maximum likelihood estimate can be easily computed.

To contextualise this approach, we first consider the case of a point null and alternative $P^{\infty}$ and $Q^{\infty}$. In this case, classical sequential testing theory posits that the sequential likelihood ratio \[ L_t = \prod_{s = 1}^t \frac{q(X_s)}{p(X_s)}\] yields a valid and powerful test upon thresholding at $1/\alpha$. Indeed, under the null, $\{L_t\}$ is an e-process, since it is an NM. 

Against simple nulls but composite alternatives, likelihood ratios such as the above are typically adjusted to account for the variety of possible alternatives. One  way to do this is to replace the above numerator with an estimate $\hat{q}_s(X_s)$. Importantly, as long as this $\hat{q}_s$ is nonanticipating, i.e., is $\mathscr{F}_{t-1}$-measurable (depending only on the first $t-1$ datapoints), the martingale property continues to hold. To highlight this nonanticipation, we shall denote these estimators as $\hat{q}_{s-1}$.
A second option is to mix over alternatives, perhaps using some non-informative ``prior'', but we will go with the first option in this paper because we are dealing with a highly nonparametric alternative (essentially the complement of all log-concave laws, or the unspecified subset of those against which one may hope to have power) --- it is easy to use kernel density estimates for $\hat q_s$, but not so easy to mix over such a loosely specified nonparametric alternative.

The sequential universal likelihood ratio statistic (ULR) extends the above to composite nulls when a maximum likelihood estimator (MLE) is computable. Concretely, the statistic is as follows: let $\hat{q}_{t-1}$ be any predictable probability density, that is $\hat{q}_{t-1}$ may be expressed as a function of only $\{X_1, \dots, X_{t-1}\}$ and additional independent randomness. As before, we should think of $\hat{q}$ as trying to estimate the underlying law $p$. Let $\hat{p}_t$ be the MLE over the null class $\lc$ with the data $X_1^t$, i.e., \[ \hat{p}_t = \argmax_{\hat{p} \in \lc} \sum_{s \le t} \log\hat{p}(X_s).\] Notice that, unlike $\hat q_{t-1}$, the MLE $\hat{p}_t$ makes use of $X_t$. The sequential ULR statistic is the process  %
\[ R_t := \prod_{s \le t} \frac{ \hat{q}_{s-1}(X_{s})}{\hat{p}_t(X_s)}.
\]
(Of course, if the numerator was simply $\prod_{s \le t} \hat q_t(X_s)$, where $\hat q_t$ is an MLE over a larger class calculated using $\{X_1,\dots,X_t\}$, then we would get the usual generalized likelihood ratio process. However, we will handle very rich nonparametric alternatives over which computing the MLE is for all practical purposes impossible, and further, for irregular models like log-concave distributions, such generalized likelihood ratios are very ill-behaved and not well understood.)

The principal factor underlying the utility of $R_t$ is that it is an e-process. Indeed, for any $P \in \lc,$ and $t$,  $R_t$ is dominated by $F_t(P) = \prod_{s\le t} \hat{q}_{s-1}(X_s)/p(X_s)$. Further, $\{F_t(P)\}$ is a $P^\infty$-martingale started at $1$, due to the predictability of $\hat{q},$ and thus for any stopping time $\tau$,  \[ \mathbb{E}_{P^\infty}[R_\tau] \le \mathbb{E}_{P^\infty}[F_\tau(P)] \le 1.\]  
Notice in the argument above that while the e-process is dominated by a $P^\infty$-martingale, it is not itself a martingale. Indeed, this property is crucial to the existence of nontrivial e-processes even when there are no such test martingales. We note that this property of domination by a $P^\infty$-NM for every $P \in \lc$ (or in general $\mathtt{P}$-NM for $\mathtt{P} \in \mathfrak{P}$) is equivalent to the e-process property itself, and can be taken as an alternate definition of the same \cite[Lem.~6]{ramdas2020admissible}.

Due to the above observation, the ULR e-process yields a valid test upon thresholding at $1/\alpha.$ The power of any such test relies on the two aspects of how well $\hat{p}_t$ and $\{\hat{q}_s\}_{s \leq t}$ estimate the underlying law $p$. Indeed, we argue in \S\ref{sec:power} that if $p \not\in \lc,$ then $\prod_{s \leq t} p(X_s)/{\hat{p}_t (X_s)}$ must grow exponentially with $t$. Thus, as long as the sequential estimates $\hat{q}_{t}$ approximate $p$ well in a cumulative regret sense, the procedure above must be consistent. Concretely, define the \emph{regret of prediction using $\{\hat{q}_t\}$} as \[ \rho_t(\hat{q};P) := \sum_{s \le t} (-\log \hat{q}_{s-1}(X_s))  - \sum_{s \le t} (-\log p(X_s)),\] 
so that better estimation results in lower regret,
and define the `well-estimable' class \[ \mathcal{Q}(\hat{q}) :=\{ P : \rho_t(\hat{q};P)/t \to 0 \textrm{ $P$-a.s. as } t \nearrow \infty \}. \] 
In  Section~\ref{sec:power} we show the following: 
\begin{myth*} \emph{(Informal)}
    Let $R_t$ denote the ULR e-process with the sequential estimator $\mathscr{E}.$ Then the test that rejects when $R_t \ge 1/\alpha$ is $\alpha$-valid, and consistent against $\mathcal{Q}(\hat{q})$.  
\end{myth*}
In fact, in \S\ref{sec:power}, we demonstrate a more refined version of the above statement, which allows $\rho_t$ to grow linearly, but at a rate bounded by the distance of $p$ from log-concavity. In any case, we comment that the class $\mathcal{Q}(\cdot)$ above is quite rich. For instance, using sieve estimators yields low-regret estimation in the above log-loss sense for nonparametric classes such as laws on compact intervals with smooth and bounded densities. The ULR e-process thus gives a powerful test for log-concavity against a rich set of alternates, even though no test martingale can deliver such properties. Our work thus offers further insight into the sequential testing of rich composite nulls, and the primacy that e-processes must take in the modern study of the same.

Along with the above asymptotic consistency result, we further derive  finite rejection rate bounds by controlling the typical rejection time of the ULR e-process in terms of the Hellinger distance of the alternaive law from log-concavity. In particular, we show explicit bounds on typical rejection times against Lipschitz and bounded laws on the unit box.
The above theoretical exploration is augmented with simulation studies on a simple parametric family comprising a mixture of two Gaussians to empirically evaluate the validity and power of the test. We find that in small dimensions $d \le 3,$ the tests show excellent validity, as well as reasonable power. We further use this simulation study to highlight the role of the quality of the estimators $\hat{q}_t$ in the power of the test.

\paragraph{Summary of Contributions}

To summarise, this paper is concerned with the theoretical and methodological aspects of sequential testing for log-concavity. We first show a negative result that demonstrates that the  approach of constructing test (super)martingales is powerless for testing this class of laws, and along the way also offering simple characterisations of the fork-convex hull of i.i.d.~sequential laws. In the positive direction, we propose using the Universal Inference based e-process as a way to test log-concavity in the absence of test martingales. We theoretically demonstrate both the consistency of the resulting sequential test, along with concrete adaptive bounds on typical rejection time under a wide class of alternatives, and illustrate the same via simulation studies.  

\section{Definitions, and Background on Log-Concave Distributions} \label{sec:def}

We begin with basic background on log-concave distributions, and necessary notation. We refer the reader to the survey of Samuard and Wellner for further details \cite{logconcavity_survey}.

\paragraph{Log-Concave Laws.} A distribution $P$ on $(\mathbb{R}^d, \mathscr{B}(\mathbb{R}^d))$ is called logarithmically concave (henceforth log-concave) if for every pair of compact sets $A,B$ and $\lambda \in (0,1),$ \[ P(\lambda A + (1-\lambda)B) \ge P(A)^{\lambda} P(B)^{1-\lambda}, \] $\lambda A + (1-\lambda) B$ is the Minkowski sum $\{ \lambda x + (1-\lambda)y : x \in A, y \in B\}.$ It is well known that a distribution that admits a density with respect to the Lebesgue measure is log-concave if and only if $P(\mathrm{d}x) = e^{g(x)} \mathrm{d}x$ for a concave function $g$. Recall that $\lc$ denotes the class of log-concave distributions with density on $\mathbb{R}^d$, while $\lc^\infty$ denote the set of i.i.d.~sequential laws $P^\infty$ for $P \in \lc$. 

\paragraph{Log-Concave M-projection.} Recall that $\mathcal{D}$ denotes the set of laws on $(\mathbb{R}^d, \mathscr{B}(\mathbb{R}^d))$ that admit densities with respect to the Lebesgue measure, and that \[ \mathcal{D}_1:= \{ P \in \mathcal{D}: \mathbb{E}[\max(0, \log p(X))] < \infty, \mathbb{E}[\|X\|] < \infty\},\] where $p(\cdot )$ is the density of $P.$ For every $P \in \mathcal{D}_1,$ there exists a unique law \[ \lcproj_P := \argmin_{L \in \lc} \mathrm{KL}(P\|L),\] where $\mathrm{KL}(\cdot\|\cdot)$ is the KL-divergence, called its \emph{log-concave M-projection}. We shall abuse notation and use $\lcproj_p$ to denote the Lebesgue density of $\lcproj_P$ (one is admitted as long as $P \in \mathcal{D}_1$). For a set of points $\{x_1^t\}$, $t \ge d+1,$ the \emph{log-concave maximum likelihood estimator (MLE)} is the log-concave M-projection of the empirical law $P_t = \sum_{s \le t} \delta_{x_s}/t,$ denoted $\hat{P}_t$. Most commonly, we shall refer to its Lebesgue density, $\hat{p}_t,$ which may equivalently be defined as \[ \hat{p}_t := \argmax_{ \substack{\log f \textrm{ is a concave function}\\ f \ge 0, \int f = 1}} \sum_{s \le t} \log f(x_s). \] The log-concave MLE has extremely favourable theoretical properties when $X_1^t \overset{\mathrm{i.i.d.}}{\sim} P$ for some $P \in \mathcal{D}_1$. For instance $\hat{p}_t \to \lcproj_p$ in the strong sense that $\exists a > 0 : \int e^{a\|x\|} |\hat{p}_t(x) - \lcproj_p(x)| \to 0$ almost surely \cite{cule2010theoretical}. 

\paragraph{Locally Absolutely Continuous Sequential Measures.} Let $\Gamma$ denote the standard Gaussian law on $\mathbb{R}^d$, $\gamma$ its density, and let $\mathtt{\Gamma} = \Gamma^\infty.$ Notice that $\mathtt{\Gamma}$ is the law of a white noise. A sequential law $\mathtt{P}$ is said to be \emph{locally absolutely continuous} (l.a.c.) with respect to $\mathtt{\Gamma},$ denoted $\mathtt{P} \lac \mathtt{\Gamma},$ if for all $t$, the law of the finite prefix $\mathtt{P}|_t(\cdot) := \mathtt{P}(X_1^t \in \cdot)$ is absolutely continuous with respect to $\mathtt{\Gamma}|_t$. Such l.a.c.~laws admit a \emph{density process}, denoted \[ \denp{P}_t := \frac{\mathrm{d}\mathtt{P}|_t}{\mathrm{d}\mathtt{\Gamma}|_t}. \]
As an example, if $\mathtt{P} = P^\infty$ for some law $P$ with Lebesgue density $p$, then $\mathtt{P} \lac \mathtt{\Gamma},$ and $\denp{P}_t = \prod_{s \le t} p(X_s)/\gamma(X_s)$. Of course, we may specify sequential laws (that are $\lac \mathtt{\Gamma}$) by specifying their density processes. Note that $Z_t^\mathtt{P}$ is a likelihood ratio process with respect to $\mathtt{\Gamma}$, and so is a $\mathtt{\Gamma}$-martingale. We shall henceforth use $\mathtt{\Gamma}$ as a reference measure for sequential laws, and almost entirely work under laws that are $\lac \mathtt{\Gamma}$.

Notice that if $\denp{P}_{t-1} > 0$ then for $\{X_t\} \sim \mathtt{P},$ $\denp{P}_t/\denp{P}_{t-1}$ is the $\mathtt{\Gamma}$-conditional density of $X_t$ given $X_1^{t-1}.$ Further, $\denp{P}_{t-1} = 0 \implies \denp{P}_t = 0.$ As a result, we may write for any adapted process $\{M_t\}$ that \[ \denp{P}_{t-1} \mathbb{E}_{\mathtt{P}}[ M_t(X_1^t) |\mathscr{F}_{t-1}] = \denp{P}_{t-1} \mathbb{E}_{\mathtt{P}}[ M_t(X_1^t)\indi\{\denp{P}_t > 0\} |\mathscr{F}_{t-1}] = \mathbb{E}_{\mathtt{\Gamma}}[ M_t \denp{P}_t \indi\{\denp{P}_t > 0\}|\mathscr{F}_{t-1}] = \mathbb{E}_{\mathtt{\Gamma}}[M_t \denp{P}_t|\mathscr{F}_{t-1}]. \]

 From this, we observe that a process $\{M_t\}$ is a $\mathtt{P}$-NSM if and only if $\{Z_t^{\mathtt{P}} M_t\}$ is a $\mathtt{\Gamma}$-NSM. Indeed, if the former is true, then we conclude from the above that $\mathbb{E}_{\mathtt{\Gamma}}[M_t\denp{P}_t ] \le \denp{P}_{t-1} M_{t-1},$ while if the latter is true, then we can conclude that $\denp{P}_{t-1}\mathbb{E}_{\mathtt{P}}[M_t |\mathscr{F}_{t-1}] \le \denp{P}_{t-1}M_{t-1} \iff \indi\{\denp{P}_{t-1} > 0\}\mathbb{E}_{\mathtt{P}}[M_t |\mathscr{F}_{t-1}] \le \indi\{\denp{P}_{t-1} > 0\} M_{t-1}.$ Since $\indi\{\denp{P}_{t-1} > 0\}$ holds $\mathtt{P}$-a.s., it follows that $\mathbb{E}_{\mathtt{P}}[M_t|\mathscr{F}_{t-1}] \le M_{t-1}$ $\mathtt{P}$-a.s., and so $M_t$ is a $\mathtt{P}$-NSM. By maintaining equalities in the above analysis, the analogous statement also holds for NMs. These facts are quite useful in our later study of fork-convex hulls.

\section{There Are No Nontrivial Test Supermartingales for Log-concavity}\label{sec:triviality}

We begin with defining a natural notion of triviality. 

\begin{defi}
An NSM $\{M_t\}$ is said to be \emph{trivial} if \( \seqG(\exists t: M_t > M_{t-1}) =  0.\) An NM $\{M_t\}$ is said to be trivial if \( \seqG(\exists t: M_t \neq M_{t-1}) = 0. \)
\end{defi}
In words, a  NSM (NM) is trivial if, almost surely, it is a non-increasing (constant) process. For the remainder of this section, we will set $\{\mathscr{F}_t\}$ to be the natural filtration. We recall the notion of test supermartingales for a class of laws $\mathfrak{P}$, which we shall refer to as just nonnegative supermartingales.

\begin{defi}
For a set sequential laws $\mathfrak{P},$ we say that a process $\{M_t\}$ is a \emph{$\mathfrak{P}$-NSM} if $\{M_t\}$ is a $\mathtt{P}$-NSM for every $\mathtt{P} \in \mathfrak{P}$. Similarly, $\{M_t\}$ is a $\mathfrak{P}$-NM is it is a $\mathtt{P}$-NM for every $\mathtt{P} \in \mathfrak{P}.$ 
\end{defi}
With these definitions in hand, we state the main result of this section, the proof of which is left to \S\ref{sec:no_martingales_proof}.
\begin{myth}\label{thm:triviality}
There are no nontrivial $\mathcal{G}^\infty$-NSMs or $\mathcal{G}^\infty$-NMs under the natural filtration, and a fortiori, there are no nontrivial $\lciidseq$-NSMs or $\lciidseq$-NMs under the natural filtration.
\end{myth}

As discussed by Ramdas et al. \cite{ramdas2022testingExchangability}, the above result implies that any valid level-$\alpha$ sequential test for log-concavity based on thresholding $\lciidseq$-NSMs or $\lciidseq$-NMs must be powerless. Indeed, in the former case, such a test against any law that is locally absolutely continuous with respect to $\seqG$ will almost surely never exceed its starting value,  and thus will almost surely never reject.  

\noindent \emph{Intuition behind the proof.} The result arises from a contradiction. To illustrate this, suppose $\{M_t\}$ is a $\mathcal{G}^\infty$-NSM, and that for some time $t$, given $\mathscr{F}_{t-1},$ it increases on the event $\{X_t \in O\}$ for some open \emph{ball} $O,$ i.e., conditionally on $X_1^{t-1} = x_1^{t-1},$ $\{X_t \in O\} \subset \{M_t > M_{t-1}\}$. Notice that due to nonnegativity, at worst it could be zero outside the ball. Now, consider a Gaussian $G_O$ of such a small variance that $G_O(O) \approx 1.$ By tuning this variance, we can ensure that $M_t > M_{t-1}$ with probability arbitrarily close to $1$ given the history, and since the drop in the $M_t$ remains bounded outside of the ball, this ensures that the conditional expectation of $M_t$ strictly increases. Since this violates the supermartingale property against $G_B^\infty \in \mathcal{G}^\infty,$ we must conclude that no such ball $O$ exists.

Of course, the set on which $\{M_t\}$ increases need not contain any ball, but still be of nontrivial mass, not to mention that this set may vary with the history in a complex way. We address such gaps by exploiting the notion of fork-convexity \cite{ramdas2022testingExchangability} which serves as a sequential analogue of convexity especially germane to (super)martingale properties, and is treated in the following section. In particular, it holds that any process $\{M_t\}$ that is a $\mathcal{G}^\infty$-NSM (or NM) is also a NSM (or NM) with respect to any sequential law in the `fork-convex hull' of $\mathcal{G}^\infty$. The main argument then demonstrates that the fork-convex hull of $\mathcal{G}^\infty$ is incredibly rich, and contains the laws of arbitrary independent processes with density (i.e., processes of jointly independent $\{X_t\}$ such that $X_t \sim p_t \in \mathcal{D}$). This large set of laws entirely obstructs the NSM (or NM) property from holding in any nontrivial manner, essentially using a robust version of the previous intuitive example. Schematically, we take the following route to establish this result, where the forward direction of each implication exploits fork-convex combinations (and the reverse is trivial).

\begin{figure}[hb]
    \centering
    \includegraphics[width = \linewidth]{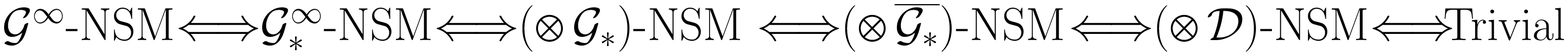}
    \caption{Schematic view of the argument. $\mathcal{G}_*$ is the set of all finite mixtures of Gaussians and $\overline{\mathcal{G}_*}$ denotes its $L_1$ closure. For any set $\mathcal{P}$, the class $\bigotimes \mathcal{P}$ consists of independent sequential laws with marginals in $\mathcal{P}$. See \S\ref{sec:cfhull_of_independent_laws} for definitions.}
    \label{fig:my_label}
\end{figure}

\subsection{Fork-convex Combinations}

In an algorithmic sense, for two laws $P,Q$, an $\alpha$-convex combination $R = \alpha P + (1-\alpha)Q$ is the law of the output of the following procedure: independently sample $U \sim P$ and $V\sim Q$, and output $X = U$ or $V$ according to the outcome of an independent $\alpha$-coin. Fork-convex combinations are the natural sequential extension of such a procedure. Concretely, we sample two trajectories $\{U_t\}\sim \mathtt{P}$ and $\{V_t\} \sim \mathtt{Q}$, release $X_t = U_t$ for $t \le s$ for some time $s$, and then flip a $h$-coin (where $h$ can depend on the history) to decide whether the subsequent tail is $X_t = U_t$ or $X_t = V_t$ for $t > s$. Notice that this is a much richer notion than convex combinations: firstly, the decision to release $U_t$ or $V_t$ only needs to be made for a tail of the output sequence, and secondly, the mixture proportion can depend on the history. Formally, this is defined as follows.
\begin{defi}[\cite{ramdas2022testingExchangability}]
    Let $\mathtt{P}, \mathtt{Q} \lac \mathtt{\Gamma}$ be sequential laws. Let $s \in \mathbb{N},$ and let $h \in [0,1]$ be an $\mathscr{F}_{s}$-measurable random variable such that $\mathtt{\Gamma}(h < 1, Z_s^{\mathtt{Q}} = 0) = 0.$ Then the $(s,h)$-fork-convex combination of $\mathtt{P}$ with $\mathtt{Q}$ is the sequential law $\mathtt{R}$ with density process \[ \denp{R}_t := \denp{P}_t \indi\{t \le s \} + \left( h \denp{P}_t + (1-h) \denp{Q}_t \frac{\denp{P}_s}{\denp{Q}_s}\right) \indi\{t > s\}.\] We shall denote this succinctly as \( \mathtt{R} = \fccomb{\mathtt{P}}{\mathtt{Q}}{s}{h}.\)
\end{defi}
Notice that fork-convex combinations probabilistically allow single data-dependent change-points, or `switches', from a law $\mathtt{P}$ to $\mathtt{Q}$. The ratio $\denp{P}_s/\denp{Q}_s$ accounts for the fact that the prefix up to time $s$ was drawn according to $\mathtt{P}$ in the case of a switch, and the condition on $h$ ensures that $\denp{Q}_s \neq 0$ when we switch to $\mathtt{Q}$ (informally meaning that the initial segment of data was not impossible under $\mathtt Q$). 

The importance of the above definition lies in the fact that fork-convex combinations preserve (super)martingale properties. Recall from \S\ref{sec:def} that $\{M_t\}$ is a $\mathtt{P}$-NSM if and only if $\{\denp{P}_tM_t\}$ is a $\mathtt{\Gamma}$-NSM. Now suppose $\{M_t\}$ is both a $\mathtt{P}$-NSM and $\mathtt{Q}$-NSM, let $\mathtt{R}$ be a $(s,h)$-fork-convex combination of $\mathtt{P}$ and $\mathtt{Q}$. For $t \ge s+1,$ and we have \begin{align*} \mathbb{E}_{\mathtt{\Gamma}}[ \denp{R}_t M_t |\mathscr{F}_{t-1}] &= h \mathbb{E}_{\mathtt{\Gamma}}[ \denp{P}_t M_t | \mathscr{F}_{t-1}] + (1-h) \frac{\denp{P}_s}{\denp{Q}_s} \mathbb{E}_{\mathtt{\Gamma}}[ \denp{Q}_t M_t] \le h\denp{P}_{t-1} M_{t-1} + (1-h) \frac{\denp{P}_s}{\denp{Q}_s} \denp{Q}_{t-1} M_{t-1} = \denp{R}_{t-1} M_{t-1},\end{align*} where we have utilized the fact that $h$ and $\denp{\cdot}_s$ are $\mathscr{F}_{t-1}$-measurable. The same calculation is trivial for $t \le s,$ and follows similarly for the martingale property. The same property extends considerably beyond finite combinations to closed fork-convex hulls, which generalise the standard notion of closed convex hull of sets.

\begin{defi}[\cite{ramdas2022testingExchangability}]
     A set is said to be fork-convex if it contains all fork-convex combinations of its elements. Let $\mathfrak{P}$ be a set of sequential laws that are locally absolutely continuous with respect to $\mathtt{\Gamma}$. The \emph{fork-convex hull} of $\mathfrak{P},$ denoted $\fhull(\mathfrak{P})$ is the intersection of all fork-convex sets containing $\mathfrak{P}$. The \emph{closed fork-convex hull} of $\mathfrak{P},$ denoted  $\cfhull(\mathfrak{P})$ is the closure of its fork-convex hull with respect to $L_1(\mathtt{\Gamma})$ convergence of the likelihood ratio processes at every fixed time $t$. 
\end{defi}
Explicitly, the closure in the definition includes all processes $\mathtt{Q}$ such that there exists a sequence $\mathtt{Q}_n$ with density process $\{Z^{\mathtt{Q}_n}_t\}$ such that $\forall t, Z^{\mathtt{Q}_n}_t \to \denp{Q}_t$ in $L_1(\mathtt{\Gamma})$. We shall refer to this as the local $L_1(\mathtt{\Gamma})$ closure. This closure induces considerable flexibility into closed fork-convex hulls, making the notion a powerful concept in light of the following phenomenon, observed by Ramdas et al. \cite[Thm.~13]{ramdas2022testingExchangability} whose argument we reproduce below.  
\begin{myprop}\label{prop:cfhull_nsm_closure}
   For a set of sequential laws $\mathfrak{P},$ a process is a $\mathfrak{P}$-NSM if and only if it is a  $\cfhull(\mathfrak{P})$-NSM. 
\end{myprop}
   \begin{proof}
       The result is evident for the fork-convex hull as an extension of the previous two-point calculation. This extends to closures as follows. Let $\{M_t\}$ be the process in question, and suppose $\mathtt{P}_n \to \mathtt{P}$ in the sense above for $\mathtt{P}_n \in \fhull(\mathfrak{P})$. Let $Z_t^n := Z_t^{\mathtt{P}_n}$ and $Z_t := Z_t^{\mathtt{P}}.$ We know that for each $t$, $Z_t^n \to Z_t$ in $L_1(\mathtt\Gamma).$ We need to show that $Z_t M_t$ is a $\mathtt{\Gamma}$-NSM. To this end, fix a $t$, and, by passing to a subspace, assume that $Z_t^n \to Z_t$ and $Z_{t-1}^n \to Z_{t-1}$ pointwise a.s. Now, since $Z_t^n M_t$ is a $\mathtt{\Gamma}$-martingale, using Fatou's lemma yields\[ \mathbb{E}_{\mathtt{\Gamma}}[Z_t M_t|\mathscr{F}_{t-1}] = \mathbb{E}_{\mathtt{\Gamma}}[\liminf Z_t^n M_t|\mathscr{F}_{t-1}] \le \liminf \mathbb{E}_{\mathtt{\Gamma}}[Z_t^n M_t|\mathscr{F}_{t-1}] \le \liminf Z_{t-1}^n M_{t-1} = Z_{t-1}M_{t-1}. \qedhere\] 
   \end{proof}

It is worth noting that while the NSM property is preserved under closures above, the same is not necessarily true of the martingale property due to the use of Fatou's Lemma when handling closures in the above proof. Nevertheless, the NM (and indeed the martingale property without appeal to non-negativity) persists under fork-convex hulls, without the closure, giving us the following characterisation.

\begin{myprop}\label{prop:fhull_nm_closure}
    For a set of sequential laws $\mathfrak{P},$ a process is a $\mathfrak{P}$-NM if and only if it is a  $\fhull(\mathfrak{P})$-NM.
\end{myprop}

\subsection{The Fork-Convex Hull of Independent Sequential Laws}\label{sec:cfhull_of_independent_laws}

Proposition~\ref{prop:cfhull_nsm_closure} gives us a concrete attack to showing the triviality of $\mathcal{G}^\infty$-NSMs: we shall show that the fork-convex hull of this set is far too rich to allow the existence of nontrivial NSMs. The bulk of our argument develops simple structural characterisations of fork-convex hulls of independent sequential laws. This section describes this characterisation using three properties, whose proof we leave to \S\ref{appx:triviality_proof_lemmas}. We begin with a key definition that sets notation for `independent sequential laws' from a set.
\begin{defi}\label{def:indep_seq_from_set}
        Let $\mathcal{P}$ be a set of distribution on $\mathbb{R}^d$. For a sequence of distributions $\{P_t\}_{t \in \mathbb{N}},$ we define $\bigotimes \{P_t\}$ as the sequential distribution of a stochastic process $\{X_t\}_{t \in \mathbb{N}}$ such that all $X_t$ are jointly independent, and for each $t \in \mathbb{N},$ $X_t \sim P_t.$ We further define $\bigotimes \mathcal{P} := \{ \bigotimes \{P_t\} : P_t \in \mathcal{P} ~  \forall t\},$ i.e.\ the set of laws of independent stochastic processes with laws at each time lying in $\mathcal{P}.$
    \end{defi}
Note that $\bigotimes \mathcal{P}$ is a much richer set than the i.i.d.\ sequential laws, which we denote $\mathcal{P}^\infty:=\{P^\infty: P \in \mathcal P\}.$ In light of this, the following result demonstrates the richness of fork-convex hulls. Recall that a set of laws is mutually absolutely continuous (m.a.c.) if every pair of laws contained in it is mutually absolutely continuous.
\begin{mylem}\label{lem:cfhull_products}
    Let $\mathcal{P}\subset \mathcal{D}$ be a m.a.c.~set of laws with density on $\mathbb{R}^d$. Then, $\cfhull(\mathcal{P}^\infty) \supset \bigotimes \mathcal{P} \supset \mathcal{P}^\infty$.
\end{mylem}
To sketch the argument underlying the above, fix any $\mathtt{P} = \bigotimes \{P_t\}.$ It suffices to demonstrate a sequence of laws $\{\mathtt{R}^T\}_{T \in \mathbb{N}},$ each generated by finite fork-convex combinations of $\mathcal{P}^\infty$-laws (and their fork-convex combinations) such that for $t \le T,$ the density process of $\mathtt{R}^T$ and $\mathtt{P}$ agree. The conclusion then follows under closure, since $\mathtt{R}^T \to \mathtt{P}$ in the appropriate sense. The concrete witness for the above Lemma is the following sequence \[ \mathtt{R}^1 := P_1^\infty, \mathtt{R}^T = \fccomb{\mathtt{R}^{T-1}}{P_T^\infty}{T-1}{0},\] where each fork-convex combination is valid since $\mathcal{P}$ is m.a.c. In essence, this exploits the fact that fork-convex combinations let one switch between laws after a time of our choosing. See \ref{appx:triviality_proof_lemmas} for details.

Next, we exploit the convex combination properties of fork-convex combinations to demonstrate that fork-convex combinations of i.i.d.~laws includes i.i.d.~products over mixtures as well. To this end, let us define the mixture classes as below.

\begin{defi}
    Let $\mathcal{P}$ be a set of distributions on $\mathbb{R}^d$. For $k \in \mathbb{N},$ we let $\mathcal{P}_k$ be the class of laws formed by $k$-fold mixtures of laws in $\mathcal{P},$ and denote $\mathcal{P}_* = \bigcup_{k \in \mathbb{N}} \mathcal{P}_k$ as the class of laws formed by finite mixtures of laws in $\mathcal{P}$.
\end{defi}
Note that $\mathcal{P}_*$ is well defined since $\mathcal{P}_k$ form an increasing set. The second key result shown in \S\ref{appx:triviality_proof_lemmas} is
\begin{mylem}\label{lem:cfhull_mixtures}
    Let $\mathcal{P} \subset \mathcal{D}$ be a m.a.c.~set of laws on $\mathbb{R}^d$. Then $\cfhull(\mathcal{P}^\infty) \supset \mathcal{P}_*^\infty$. 
\end{mylem}
The key observation underlying the above is already demonstrated in showing that $\cfhull(\mathcal{P}^\infty) \supset \mathcal{P}_2^\infty.$ To see this, fix any $P, Q\in \mathcal{P},$ and $\alpha \in [0,1]$. We need to demonstrate a sequence of laws $\mathtt{R}^T$ constructed via repeated fork-convex combinations that match the density process of $\mathtt{R}:= (\alpha P + (1-\alpha) Q)^\infty$ for times up to $T$. This is realised as follows: \[\mathtt{R}^0 := P^\infty, \mathtt{S}^T := \fccomb{\mathtt{R}^{T-1}}{Q^\infty}{T-1}{\alpha}, \mathtt{R}^T := \fccomb{\mathtt{S}^T}{P^\infty}{T}{0}.\] In the above, $\mathtt{S}^T$ matches the density process of $\mathtt{R}$ up to time $T$ by mixing between $\mathtt{R}^{T-1}$ (whose tail behaves as $P^\infty$) and $Q^\infty$ appropriately. $\mathtt{R}^T$ then switches the tail of $\mathtt{S}^T$ to behave as $P^\infty$ to enable the recursion. This argument extends to $\mathcal{P}_k^\infty$ for any arbitrary $k$ by inducting over $k$ (which is possible since a member of $\mathcal{P}_k$ is a mixture of a $\mathcal{P}_{k-1}$ law and a $\mathcal{P}$ law). Since $k$ is arbitrary, this immediately extends to $\mathcal{P}_*.$

Finally, we exploit the closure properties of fork-convex hulls under $L_1(\mathtt{\Gamma})$ to extend fork-convex hulls from product measures over a set to product measures over closures of that set.
\begin{mylem}\label{lem:cfhull_closures}
    Let $\mathcal{P}$ be a set of distributions on $\mathbb{R}^d$ that have densities. Then $\cfhull(\bigotimes \mathcal{P}) \supset \bigotimes \overline{\mathcal{P}}$, where $\overline{\mathcal{P}}$ is the $L_1(\Gamma)$-closure of $\mathcal{P}$.
\end{mylem}
The above lemma is a straightforward consequence of the closure properties as detailed in \S\ref{appx:triviality_proof_lemmas}.

\subsection{Proof of the Absence of Nontrivial Test Martingales}\label{sec:no_martingales_proof}

The previous section demonstrates that taking closed fork-convex hulls can yield significant expansion to product laws over sequences. This section exploits these properties to demonstrate the triviality of $\mathcal{G}^\infty$-NSMs. The key observation underlying this is the following standard fact about the richness of Gaussian mixtures. Recall that $\overline{\mathcal{P}}$ denotes the $L_1(\Gamma)$-closure of $\mathcal{P}$. 

\begin{mylem}\label{lem:gauss_closures}
    $\mathcal{G}_*$ is $L_1(\Gamma)$-dense in the set of all distributions with densities, i.e., $\overline{\mathcal{G}_*} = \mathcal{D}$. 
\end{mylem}

The $L_1(\mathrm{Leb})$-denseness of mixtures of Gaussians in $\mathcal{D}$ is a classical fact; for instance see the work of Alspach and Sorenson \cite{alspach1972nonlinear} or Lo \cite{lo1972finite}. More recently, a considerably more robust result was presented by Bacharoglu \cite{bacharoglou2010approximation}, who shows that Gaussian mixtures are dense in nonnegative simple functions in both an $L_1$ and an $L_\infty$ sense. This also suffices for our purposes since nonnegative simple functions are themselves $L_1$-dense in nonnegative integrable functions.  The $L_1(\Gamma)$-denseness follows since $\Gamma$ admits a uniformly bounded density with respect to the Lebesgue measure.

We note that our argument extends to any such set, i.e., to any $\mathcal{P}$ such that $\overline{\mathcal{P}} = \mathcal{D}$. The Gaussians serve as a convenient witness within $\lc$ for which this property holds. With this in hand, we proceed as below.

\begin{proof}[Proof of Theorem~\ref{thm:triviality}]

Let $\{M_t\}$ be a $\mathcal{G}^\infty$-NSM. First observe by Lemma~\ref{lem:cfhull_mixtures} and Proposition~\ref{prop:cfhull_nsm_closure} that as a consequence, $\{M_t\}$ is also a $\mathcal{G}_*^\infty$-NSM. Next, by Lemma~\ref{lem:cfhull_products} and Proposition~\ref{prop:cfhull_nsm_closure}, it is further a $(\bigotimes \mathcal{G}_*)$-NSM. Similarly, by Lemma~\ref{lem:cfhull_closures} and Proposition~\ref{prop:cfhull_nsm_closure}, we conclude that $\{M_t\}$ is also a $(\bigotimes \overline{\mathcal{G}_*})$-NSM. Finally, by Lemma~\ref{lem:gauss_closures}, we conclude that $\{M_t\}$ is a $(\bigotimes \mathcal{D})$-NSM.\footnote{We can also argue this more directly: observe that taking closed fork-convex hull is an idempotent operation, i.e. $\cfhull(\cfhull(\mathfrak{P})) = \cfhull(\mathfrak{P})$ (which follows from the facts that closed fork-convex hulls are fork-convex, and that closures of closed sets are invariant). Therefore, using the chain of Lemmata of \S\ref{sec:cfhull_of_independent_laws}, $\cfhull(\mathcal{G}^\infty) \supset \bigotimes \overline{\mathcal{G}}_*$, and so $\{M_t\}$ is a $(\bigotimes \mathcal{D})$-NSM.}

We now argue that $\bigotimes \mathcal{D}$ is too rich to admit nontrivial NSMs. The argument is by contradiction---we assume that $M_t > M_{t-1}$ for some $t$ with nontrivial probability, and use this to construct a law in $\bigotimes \mathcal{D}$ that violates the NSM property. The argument repeatedly exploits the topological equivalence of $(\mathbb{R}^d)^t$ and $\mathbb{R}^{dt}$ under the product and metric topologies respectively. We shall denote the Lebesgue measure in $m$ dimensions as $\mathrm{Leb}_m,$ and we note that the product Lebesgue measure on $(\mathbb{R}^d)^t$ is identical to $\mathrm{Leb}_{dt}$, and use the latter to denote the former.

Let us proceed with the argument. For a natural number $t$, define the event $\mathsf{A}_t := \{M_t > M_{t-1}, M_{t-1} < \infty\},$ i.e. that $\{M_t\}$ increases at time $t$. It suffices to argue that no matter the $t$, the mass of $\mathsf{A}_t$ is zero, since $\mathtt{\Gamma}(M_{t-1} = \infty)$ must be zero due to integrability of $M_{t-1}$. For the sake of contractiction, assume $\mathtt{\Gamma}(\mathsf{A}_t) > 0.$ For $n \in \mathbb{N},$ define the approximations $\mathsf{A}_t^n := \{ M_t \ge M_{t-1} + 1/n, M_{t-1} \le n\}.$ The $\mathsf{A}_t^n$ form an increasing sequence of sets, and converge to $\{M_t > M_{t-1}, M_{t-1} < \infty\} = \mathsf{A}_t.$ 

Now, since $\mathtt{\Gamma}(\mathsf{A}_t)> 0$ and $\mathsf{A}_t \in \mathscr{F}_{t},$ we conclude that $\leb_{dt}(\mathsf{A}_t) > 0$ due to the mutual absolute continutity of Gaussians and Lebesgue measures on Euclidean spaces. Without loss of generality, we may assume $\leb_{dt}(\mathsf{A}_t) < \infty$ (since otherwise we may pass to a subset of $\mathsf{A}_t$ such that of positive and finite mass, using sigma-finiteness of the Lebesgue measure, and run the argument on this subset). Since $\mathsf{A}_t^n \nearrow \mathsf{A}_t,$ we have by regularity of measure that $\leb_{dt}(\mathsf{A}_t^n) \to \leb_{dt}(\mathsf{A}_t),$ and in particular there exists an $n$ such that $\leb_{dt}(\mathsf{A}_t^n) \in (0,\infty).$ Fix such an $n$ for the remainder of the argument.

Recall that an open rectangle in $\mathbb{R}^m$ is a Cartesian product of open intervals, i.e. a set of the form $\bigtimes_{i = 1}^m (a_i, b_i)$ for $a_i < b_i.$ Similarly, we say that $R$ is an open rectangle in $(\mathbb{R}^d)^t$ if there exist open $\mathbb{R}^d$-rectangles $S_1 \dots S_t$ such that $R = \bigtimes_{s = 1}^t S_s$. The following statement is a consequence of basic topological and measure theoretic properties of Euclidean spaces, which we prove in \S\ref{appx:nice_rectangle}.

\begin{mylem}{\label{lem:find_nice_rectangle}}
    Let $E \subset (\mathbb{R}^d)^t$ be such that $\leb_{dt}(E) > 0.$ For every natural $m \in \mathbb{N},$ there exists an open rectangle $R$ in $(\mathbb{R}^d)^t$ such that \[\leb_{dt}(R) > 0 \quad \textit{and} \quad \leb_{dt}(R \cap E) \ge \frac{m}{m+1} \leb_{dt}(R).\]
\end{mylem}

Exploiting the above result, we may construct a sequence of rectangles in $(\mathbb{R}^d)^t, \{R_m\}_{m \in \mathbb{N}}$ each of positive mass such that \[ \frac{\leb_{dt}(\mathsf{A}_t^n \cap R_m)}{\leb_{dt}(R_m)} \ge \frac{m}{m+1}.\] 

Now, since each $R_m$ is a rectangle, there exists a law $\mathtt{D}_m \in \bigotimes \mathcal{D}$ such that the prefix restriction $\mathtt{D}_m|_t = \mathrm{Unif}(R_m)$.  Indeed, if $R_m = \bigtimes_{s = 1}^t S_s^m,$ then $\mathtt{D}_m = \bigotimes \{D_s^m\},$ where $D_s^m = \mathrm{Unif}(S_s^m)$ for $s \le t,$ and $D_s^m = \Gamma$ for $s > t$. We claim that for large $m$, $\mathtt{D}_m$ witness a violation of the NSM property for $\{M_t\}$. We demonstrate this using the process $\{N_t\} := \{\min(M_t, n+1)\}$. 

Notice that if $\{M_t\}$ is a $\mathtt{P}$-NSM, then so is $\{N_t\}$, since \[ \mathbb{E}[N_t|\mathscr{F}_{t-1}] \le \min(\mathbb{E}[M_t|\mathscr{F}_{t-1}], \mathbb{E}[n+1|\mathscr{F}_{t-1}]) = \min(M_{t-1}, n+1) = N_{t-1},\] and the nonnegativity follows since both $M_t$ and $n+1$ are nonnegative. Further, since $M_{t-1} \le n$ on $\mathsf{A}_t^n,$ it follows that $N_t \ge N_{t-1} + 1/n$ on $\mathsf{A}_t^n$ as well, since $n + 1/n \le n+1.$

Consequently, we have \begin{align*} \mathbb{E}_{\mathtt{D}_m}[N_t] &\ge \mathbb{E}_{\mathtt{D}_m}\left[(N_{t-1} + 1/n)\indi\{X_1^{t} \in \mathsf{A}_t^n\}\right]  + 0\\  &= \mathbb{E}_{\mathtt{D}_m}\left[N_{t-1} \indi\{X_1^{t} \in \mathsf{A}_t^n\}\right] + \frac{\mathtt{D}_m(\mathsf{A}_t^n)}{n} \ge \mathbb{E}_{\mathtt{D}_m}\left[N_{t-1} \indi\{X_1^{t} \in \mathsf{A}_t^n\}\right] + \frac{m}{n(m+1)},\end{align*} where the final inequality exploits the fact that at least a $m/(m+1)$ fraction of the mass of $R_m$ lies in $\mathsf{A}_t^n$, and we have used the nonnegativity of $N_{t}.$

However, since $N_{t-1}$ is upper bounded by $n+1$, we observe that \[ 0 \le \mathbb{E}_{\mathtt{D}_m}[N_{t-1} \indi\{X_1^t \in (\mathsf{A}_t^n)^c\}] \le (n + 1) \mathtt{D}_m( (\mathsf{A}_t^n)^c ) \le \frac{n+1}{m+1},\] and so \[\mathbb{E}_{\mathtt{D}_m}[N_{t-1} \indi\{X_1^t \in \mathsf{A}_t^n\}] = \mathbb{E}_{\mathtt{D}_m}[N_{t-1}] - \mathbb{E}_{\mathtt{D}^m}[N_{t-1} \indi\{X_1^t \in (\mathsf{A}_t^n)^c ] \ge \mathbb{E}_{\mathtt{D}_m}[N_{t-1}] - \frac{n + 1}{(m+1)}. \] 

But now, we conclude that \[\mathbb{E}_{\mathtt{D}_m}[N_t] \ge \mathbb{E}_{\mathtt{D}_m}[N_{t-1}] + \frac{(m/n) - (n+1)}{m+1}.\] Choosing $m > 3n^2,$ and exploiting $n \ge 1,$ this implies that $\mathbb{E}_{\mathtt{D}_m}[N_t] > \mathbb{E}_{\mathtt{D}_m}[N_{t-1}],$ thus contradicting the supermartingale property of $\{N_t\}$ under $\mathtt{D}_m$ (since supermartingales must have non-increasing mean sequences). We conclude that it cannot hold that $\mathtt{\Gamma}(\mathsf{A}_t) > 0,$ i.e., $M_t \le M_{t-1}$ $\mathtt{\Gamma}$-almost surely. 

But, since $t$ is arbitrary, we immediately conclude that \[\mathtt{\Gamma}( \exists t \ge 2 : M_t > M_{t-1}) \le \sum_{t \ge 2} \mathtt{\Gamma}(M_t > M_{t-1})  = 0.\]

 The argument for NMs follows from this as well. If $\{M_t\}$ is a $\bigotimes \mathcal{D}$-NM, then it is also an NSM, and thus almost surely does not increase. But this means that $M_1 - M_t \ge 0$ is also a nonnegative supermartingale, and therefore does not increase, which implies that $M_t$ also does not decrease almost surely.  \qedhere
    
\end{proof}

\noindent \emph{Remark.} It may be possible to develop a different argument that does not explicitly need to pass through the notion of fork-convex hulls. Perhaps one could directly work with the $A_s^n$ above, and replace $\mathtt{D}_m$ a by sufficiently skinny Gaussian $\mathtt{G}_m$ such that $\mathtt{G}_m(A_s^n \cap R) \approx \mathtt{G}_m(R) \approx 1$. However, there would still be sufficiently many technical details to iron out, so such an approach is not necessarily shorter or cleaner. More importantly however, our chosen path of development above leads to a richer characterisation of fork-convex hulls of i.i.d.~processes with densities, and further directly illustrates the utility of such a characterisation. It thus deepens our understanding of the important geometric concept of fork-convexity.

\section{The Sequential Universal Likelihood Ratio E-Process}\label{sec:power}

We begin by recalling the definition of e-processes from the introduction.

\begin{defi}
An $\{\mathscr{F}_t\}$-adapted process $\{E_t\}$ is said to be an e-process for a set of sequential laws $\mathfrak{P}$ if \[ \sup_{\mathtt{P} \in \mathfrak{P}} \sup_{\tau} \mathbb{E}[E_\tau] \le 1,\] where the second supremum is over all stopping times. Further, if for some  $n \geq 1$ it holds a.s. with respect to all $\mathtt{P} \in \mathfrak{P}$ that $E_1 = E_2 = \dots = E_{n-1} = 1,$ then we say that $E_t$ is an e-process for $\mathfrak{P}$ started at time $n$. 
\end{defi}

Next, we define the universal likelihood ratio (ULR) process~\cite{wasserman2020universal}, which forms the main object of interest for this section.
\begin{defi}
    Let $\mathscr{E}$ denote a sequence of estimators $\{\mathscr{E}_t\}_{t \ge 0}$ such that each $\mathscr{E}_t : (\mathbb{R}^d)^{t} \to \mathcal{D}$. At any $t$, denote $\hat{q}_{t} = \mathscr{E}_t(X_1, \dots, X_{t}).$ Finally, let $\hat{p}_t$ denote the log-concave maximum likelihood estimate over the data $X_1, \dots, X_t$ (which exists if $t > d$). The ULR process is the statistic \[ R_t(X_1^t;  \mathscr{E}) := \indi\{t \le d\} + \indi\{t > d\}   \prod_{d+1 \le s \le t} \frac{ \hat{q}_{s-1}(X_s)}{ \hat{p}_t(X_s)} . \] 
\end{defi}

We shall often suppress the dependence of $R_t$ on $X_1^t$ and $\mathscr{E}$. The initial setting of $R_t = 1$ for $t \le d$ is to account for the fact that log-concave MLEs are known to exist only if at least $d+1$ samples are available. 

As discussed in the introduction, $\{R_t\}$ constitutes an e-process due to the predictability of $\hat{q}_{t-1}$ and the fact that they are probability densities. 
We formally state the validity of $R_t$ as a proposition. 

\begin{myprop}
    For any $\mathscr{E},$ the process $\{R_t\}$ is an e-process for $\lciidseq$ started at time $d+1$. Consequently, rejecting the null hypothesis when $R_t \ge 1/\alpha$ results in an $\alpha$-valid test for log-concavity.
\end{myprop}

\noindent \emph{On exact MLEs.} The measures $\hat{p}_t$ need not exactly maximise the likelihood ratio in the above. Indeed, if instead of the exact log-concave MLE $\hat{p}_t$ we instead an estimate $\tilde{p}_t$ such that \[ \sum_{s \le t} \log \tilde{p}_t(X_s) \ge - \log(1/\varepsilon) +  \sum_{s \le t} \log \hat{p}_t(X_s), \] then $\varepsilon R_t \cdot \prod \frac{\hat{p}_t(X_s)}{\tilde{p}_t(X_s)}$ is an e-process, and this can be thresholded at $1/\alpha$ as before. This observation is pertinent since practical procedures for computing the log-concave MLE of a dataset are inexact, and only approximate the solution up to a (user-specified) additive gap in the log-likelihood objective, and require computation that scales polynomially with the inverse of this additive gap.  

For the remainder of this section, we shall equate laws $P \in \mathcal{D}_1$ with their density, denoted $p$.

\subsection{Consistency of the ULR E-Process for Testing Log-Concavity}

Consistency of the ULR e-process depends strongly on the underlying estimator $\mathscr{E}.$ Indeed, as an extreme example, consider the case of $\hat{q}_{t}(X_t) = \indi\{ X_t = X_1\}$, for which the resulting $R_t$ is a.s. $0$ for any time $t \ge d+1$ so long as the law $P$ is continuous, and the test is thus powerless against such laws. 

It thus follows that the ULR e-process can only yield power against a set of laws determined by the estimator $\mathscr{E}$. Concretely, we shall argue the same against the following set of `well estimable' laws. Below, $d_H$ below denotes the Hellinger distance.

\begin{defi}
    For a sequential estimator $\mathscr{E},$ and a density $p \in \mathcal{D}_1$, define the prediction regret for a sequence $\{X_t\}$ as \[ \rho_t(\mathscr{E}; p) := \sum_{s \le t} \log p(X_s) - \log \hat{q}_{s-1}(X_s).\] Further, let $\lcproj_p$ denote the log-concave M-projection of $p$. We define the class of distributions that are well estimable by $\mathscr{E}$ with respect to log-concavity as 
    \[ \mathcal{Q}(\mathscr{E}; c) := \left\{p \in \mathcal{D}_1 : P^\infty\left( \limsup_{t \to \infty} \frac{\rho_t(\mathscr{E}; p)}{t d_H^2(p, \lcproj_p)} \le c\right) = 1 \right\}.\]
\end{defi}

The main result of this section is that the ULR-based test is powerful against the above well-estimable laws, which is shown later in this section. 

\begin{myth}\label{thm:consistency}
    There exists a constant $c > 1/25$ such that if $p \in \mathcal{Q}(\mathscr{E}; c)\setminus \lc$, then $P^\infty(R_t \to \infty) = 1.$ Consequently, the ULR e-process yields a consistent test against i.i.d. draws from any distribution in $\mathcal{Q}(\mathscr{E}; c)$. 
\end{myth}

The well-estimability condition above essentially requires that the distribution can be estimated well in a log-loss sense. For i.i.d.~distributions, one expects that for reasonable $\mathscr{E},$ the estimates $\hat{q}_t$ converge to some $\hat{q},$ and thus the regret grows for large $t$ as $\rho_t \approx t \mathrm{KL}(p \|\hat{q})$ (which could grow sublinearly in $t$ if $\mathrm{KL}(p \|\hat{q}_t)\to 0$, but the latter convergence is not required). The class $\mathcal{Q}$ thus roughly consists of distributions can be estimated well in KL divergence. Such estimation can be a challenging task in complete generality, since the KL divergence is quite sensitive to mismatch in the tails of distributions. However, under mild restrictions such as compactness of support and smoothness, such estimability is quite forthcoming. Indeed, we give the following statement to illustrate this point. This is proved in \S\ref{appx:bounded_unit_box_regret}.

\begin{corr}\label{prop:consistency_box}
    Let $\mathcal{D}_{\mathrm{Box, Lip,\boundparam}}$ denote the set of $1$-Lipschitz densities supported on the unit box $[-1,1]^d$ and bounded between $[1/\boundparam,\boundparam].$ There exists a sequence of sieve maximum likelihood estimators $\mathscr{E}$ such that for every $c > 0,$  $\mathcal{D}_{\mathrm{Box, Lip,\boundparam}} \subset \mathcal{Q}(\mathscr{E}; c)$, i.e., the ULR e-process yields a consistent test against i.i.d.~draws from such distributions.
\end{corr}

It is further interesting that the consistency of the test does not require that the regret $\rho_t/t \to 0,$ and only that it gets small enough relative to the Hellinger distance between $p$ and its log-concave M-projection $\lcproj_p$. This signals that deviations from log-concavity may be detected far before the underlying law can be estimated, which is quite favourable theoretically, although its practical effects depend significantly on how large a $c$ can be taken in Theorem~\ref{thm:consistency}. 

\begin{proof}[Proof of Theorem~\ref{thm:consistency}]
    We begin by defining $\sigma_t(p) = \sum_{s \le t} \log p(X_s) - \log \hat{p}_t(X_s).$ Observe that \[ \log R_t = \sigma_t(p) - \rho_t(\mathscr{E};P).\] Further, by assumption, we have that $p \in \mathcal{Q}(\mathscr{E}, c)$ for some $c,$ and thus for any $\zeta > 0,$ we have that \[ \rho_t \le (1 + \zeta) c t d^2_H(p, \lcproj_p),\] for large enough $t$.  Consequently, to show that $R_t \to \infty,$ it suffices to show that $P^\infty$ almost surely,
\begin{equation}\label{eq:claim} \liminf_{t \to \infty} \frac{\sigma_t}{t d_H^2(p, \lcproj_p)} \overset{}{\ge} (1 + 2\zeta) c. 
    \end{equation}
It is at this point that the following lemma is useful, the proof of which is left to \S\ref{appx:consistency_proof}. \begin{mylem}\label{lem:growth_of_sigma_t}
    For any $p \in \mathcal{D}_1,$ it holds that \[P^\infty\left( \liminf_{t \to \infty} \frac{\sigma_t(P)}{t d_H^2(P,\lcproj_P)} \ge \frac{1}{25} \right) =1.\]
\end{mylem}

The claim~\eqref{eq:claim} thus follows so long as $(1 + 2\zeta) c \le 1/25,$ and since $\zeta > 0$ can be taken arbitrarily small, this allows us to take any $c < 1/25.$ We note that the constants in this argument are loose, and informal calculations suggest that it may be possible to improve $c$ up to about $1/6$. 

The proof of Lemma~\ref{lem:growth_of_sigma_t} 
relies on strong convergence properties of the log-concave MLE $\hat{p}_t$ to the log-concave M-projection $\lcproj_p$. Recall that for a pair of functions $u \le v,$ a bracket $[u,v]$ is the set of all functions that lie between $u$ and $v$ everywhere. By exploiting a characterisation of the convergence properties of log-concave MLEs due to Cule and Samworth \cite{cule2010theoretical}, Dunn et al.~\cite[Lem.~1]{dunn2021-logConcaveTesting} show that there is a small bracket that is well separated from $P$ such that $\hat{p}_t$ eventually lies in this bracket. Conditioning on this event, we then exploit a classical result of Wong and Shen \cite{wong1995probability} which show linear growth of $\sigma_t(p)$ with condiitonal probability at least $1 - \exp{-\Omega(t)},$ at which point the lemma follows by Borel-Cantelli. As mentioned before, see \S\ref{appx:consistency_proof} for the full proof of Lemma~\ref{lem:growth_of_sigma_t}.
\end{proof}

\subsection{Power of the ULR E-Process for Testing Log-Concavity}

The argument underlying Theorem~\ref{thm:consistency} is also amenable to deriving rates, under further restrictions on the underlying law $P$. As in the previous section, we argue this using the decomposition $\log R_t = \sigma_t(p) - \rho_t(\mathscr{E};p).$ 

\subsubsection{Challenges, and Context from the Theory of Log-Concave MLEs}

With the above approach, the argument breaks into two parts. Frstly, we assume that we use a good enough estimator $\mathscr{E}$ so that $\rho_t$ is not too large with high probability. Such an assumption is necessary for the approach we take, although in principle the test can be analysed using a different decomposition, in which case this assumption may perhaps be weakened. In any case, we observe that for concrete alternate hypotheses such as laws with Lipschitz densities supported on the unit hypercube, $\rho_t$ can indeed be appropriately controlled. It is worth noting, however, that the resulting rate bounds are strongly driven by the behaviour of $\rho_t,$ and thus the estimator being considered, which limits the power of the results to follow.

The second part of the argument requires us to show that $\sigma_t$ is large, i.e., to argue that the log-concave MLE cannot represent the underlying law very well when it is not log-concave. While a natural statement, arguing this is challenging because this requires us to understand the behaviour of log-concave MLEs `off-the-model,' i.e., when the data is not drawn from a log-concave distribution itself. With the notable exception of Barber and Samworth \cite{barber2021local}, this task has not been undertaken in the literature, with most works focusing on on-the-model minimax rate bounds \cite{kim2016global, kur2019optimality, han2021set,carpenter2018near}. Let us consider this in some detail.

Tight analysis of the on-the-model log-concave estimation problem fundamentally relies on a subtle  reduction of the rates of log-concave MLEs to the problem of controlling deviations of empirical processes over convex sets, i.e., to that of controlling $\sup_C |P(C) - P_t(C)|$ under data drawn from $P$, where $P_t$ is the empirical law, and the supremum is over convex sets in a bounded domain \cite{carpenter2018near}. Using this observation and a refined study of these deviations, Kur et al.~\cite{kur2019optimality} recently showed tight on-the-model estimation rates of the form $d_H(\hat{p}_t, p) = O(n^{-1/(d+1)})$ when $p \in \lc$ and $d \ge 3$ (Han showed similar results, along with extensions to $s$-concave densities \cite{han2021set}). 
While significant elements of this study can be extended to analysing off-the-model behavior, the analysis ultimately cannot be applied to our situation. The gap arises because their argument only upper bounds the  quantity \[ \theta_t := \mathbb{E}_{X \sim P}[\log \lcproj_p(X)] - \mathbb{E}_{X \sim P}[ \log \tilde{p}_t(X)],\] where for a small constant $c$, $\tilde{p}_t \propto \max(c,p_t)$ is a slight modification the log-concave MLE. When $p = \lcproj_p,$ this object is a KL-divergence, and so is lower bounded. However, when $p \neq \lcproj_p$ (that is, $P$ is not log-concave), this quantity is may well be negative. Notice that this is a problem for us precisely because $\mathbb{E}[\sigma_t]/t \approx \theta_t + \mathbb{E}_{X \sim p}[\log p(X) - \log \lcproj_p(X)].$ When $p \not\in \lc,$ the second term can indeed be shown to be large, but the lack of a lower bound on the first term limits the applicability of such results. We also note that other aspects of the argument, which are relatively simple in on-the-model analysis (for instance, arguing that the mass $p$ places on sets of the form $\{\lcproj_p(x) < \gamma\}$ is small), are also rendered inoperative in off-the-model analysis.

Of course, we can in principle exploit the results of Barber and Samworth instead. However, these results give quite poor rates. Roughly speaking, Theorem 5 of their paper \cite{barber2021local} shows that off-the-model, $d_H(\hat{p}_t, \lcproj_p) \lesssim t^{-1/4d},$ and thus any analysis that exploits this result cannot hope to show that $\sigma_t$ is large for $t \ll d_H(p,\lc)^{-4d}$. This power of $4d$ arises since the analysis of \cite{barber2021local} passes through a reduction to convergence of empirical laws in Wasserstein distance (which gives the relatively benign factor of $d$), and further suffers a $1/4$th power slowdown relative to this convergence (which is both unavoidable, and leads to a $4d$ exponent). 

Our analysis sidesteps these issues by controlling the growth of $\sigma_t$ on the basis of bracketing entropy (see \S\ref{appx:consistency_proof}) bounds for the class of bounded log-concave laws on compact supports. Our bound below holds for all $d$ but appears to be new for $d \ge 4$. Indeed, we show the following statement.

\begin{mylem}\label{lem:lc_entropy_bound}
    Let $\lc_{d,\boundparam}$ denote the set of laws with log-concave densities that supported on $[-1,1]^d$ and uniformly upper bounded by a constant $\boundparam$. There exists a constant $C_d$ dependending only on the dimension such that %
    \[ \mathcal{H}_{[]}( \lc_{d,\boundparam}, \zeta) = C_d \widetilde{\Theta}\left( (B/\zeta)^{\max(d/2, d-1)} \right),\] where the $\widetilde{\Theta}$ hides terms that scale polylogarithmically with $\zeta$ or $\boundparam$. 
\end{mylem}

We note that the lower bounds on the bracketing entropy implicit in the statement of Lemma~\ref{lem:lc_entropy_bound} were already shown by Kim \& Samworth \cite[Thm.~8]{kim2016global}, who further also showed the corresponding upper bounds for $d \le 3.$ 
While not the central point of the paper, we develop upper bounds for the same when $d \ge 4$. 
There are two salient technical points regarding the bound above. Firstly, observe that for $d \ge 4,$ the entropic bounds lie in the non-Donsker regime, i.e., when the Dudley integral $\int_0^{\varepsilon} \sqrt{\mathcal{H}_{[]}(\lc_{d,\boundparam}, \zeta} \mathrm{d}\zeta$ does not converge due to a blow-up near $\zeta = 0$, which typically (but not always) represents a slowdown in the convergence rates that can be shown via entropy integrals. Secondly, for $d \ge 3,$ the bound grows as $\zeta^{-(d-1)}$ rather than as $\zeta^{-d/2}.$ The latter quantity is pertinent it is close to the growth rate of the bracketing entropy of convex sets which is $\zeta^{-(d-1)/2}$; see \S\ref{appx:lc_entropy_bound_proof}. This fact underlies the power of the previously discussed reduction of the analysis of log-concave MLE rates to control on the deviations of empirical processes over convex sets, which admit a slower entropy growth. %

Lemma~\ref{lem:lc_entropy_bound} is proved in \S\ref{appx:lc_entropy_bound_proof}. The growth bounds of this result ensure that for $t \gtrsim d_H(p,\lc)^{2(d-1)},$ $\sigma_t$ is linearly large in $t,$ even when the underlying law is not log-concave. This $2(d-1)$ exponent should be compared to the aforementioned $4d$th power scaling that one expects to emerge from using the Wasserstein continuity based approach discussed above.  Of course, the dependence on $d$ could potentially be improved even further. For instance, if the on-the-model analysis can indeed be extended to off-the-model, it is plausible to expect dependence of the form $d+1$ instead of $2d - 2.$ However, this remains a challenging problem for future work.

It is worth noting that while the techniques for the bounds in Lemma~\ref{lem:lc_entropy_bound} exist in the  literature, the bounds themselves appear to not have explicitly been observed. We believe that this might be because it was previously observed that due to existing lower bounds on this entropy, the resulting growth rate bounds that emerged from entropic considerations could not be optimal for the rate analysis of the log-concave MLE, at least in on-the-model settings. It should also be noted that the bounds above are explicitly for compactly supported log-concave laws  (which is a restriction, but a relatively mild one, due to the exponentially decaying tail enjoyed by all log-concave densities). Further note that that the brackets we construct for this setting are `improper', i.e. the bracketing functions are themselves not log-concave, which may limit utility in direct analysis of the difference between $\hat{p}_t$ and $\lcproj_p,$ but is good enough when studying the behaviour of $\sigma_t$.

\subsubsection{Bounding Typical Rejection Times for the ULR Test}

As discussed previously, our analysis of $\sigma_t$ passes through a bracketing entropy bound for bounded, compactly supported log-concave laws. For such bounds to be effective, we need to ensure that the log-concave MLE $\hat{p}_t$ itself is bounded. This is enabled by the quantity $\Delta_P$, defined for a law $P$ as \[ \Delta_P := \min_{v:\|v\| = 1} \mathbb{E}_P[ |\langle v, X - \mathbb{E}_P[X]\rangle| ], \] 
which was identified by Barber and Samworth \cite{barber2021local} as a means to lower bound the covariance of the log-concave projection of $P$, which in turn can be exploited to upper bound the supremum of $\lcproj_p$ and (indirectly) $\hat{p}_t$. Observe that $\Delta_P$ roughly corresponds to the minimum eigenvalue of the covariance matrix of $P$ --- indeed, it is best seen as a robust version of the same.  

With this in hand, we are ready to state our main result, the proof of which is the subject of \S\ref{appx:rate_bound_proof}. Recall that $d_H(p, \lc) = \inf_{q \in \lc} d_H(p,q)$ and $\tau_\alpha := \inf\{t: R_t \ge 1/\alpha\}$ is the rejection time.

\begin{myth}\label{thm:rate}

    Suppose $p$ is supported on $[-1,1]^d,$ and let $\pi_t \to 0$ be a sequence such that for every $t$, \[ P^\infty \left(  \frac{\rho_t(\mathscr{E}, p)}{t d_H^2(p, \lc )} \ge \frac{1}{25} \right) \le \pi_t. \]
    Then there exists a constant $c \in [1/600,1]$ and a natural number $T_0$ such that for any $t \ge T_0 + \frac{\log(1/\alpha)}{c d_H^2(p, \lc)},$ 
    \[ P^\infty\left( \tau_\alpha > t  \right) \le \pi_t + \frac{1}{c}\exp{-ct  d_H^2(p,\lc)} + \frac{1}{c} \exp{- ct\Delta_P^2/d^2}, \] and  \[ T_0 = C_d \cdot \widetilde{O}\left(\Delta_P^{-\max(d^2/2, d^2-d)} d_H(p, \lc)^{-\max((d+4)/2,2(d-1))} + d^2 \Delta_P^{-2}\right), \] where $C_d$ depends only on $d$, and the $\widetilde{O}$ hides terms depending polylogarithmically on $d_H(p,\lc), \Delta_P$.
\end{myth}

Observe that from the statement above we may conclude that the average rejection time is bounded as \[ \mathbb{E}[\tau_\alpha] = \sum_t P^\infty(\tau_\alpha > t) \le  \sum_t \pi_t + O(T_0).\] Here, the first term is driven by the predictability of $p$ using the estimators $\mathscr{E},$ while the second term is driven by our analysis of the noise scale of log-concave density estimation in off-the-model scenarios. 

In typical situations, the former of these terms will dominate the resulting bounds, since typical alternate classes will be much larger than the class of log-concave distributions. For instance, using results on the estimation of uniformly lower-bounded Lipschitz densities \cite[e.g.]{wong1995probability}, we show the following result  about the set 
$\mathcal{D}_{\mathrm{Box,Lip,\boundparam}}$ introduced in Corollary~\ref{prop:consistency_box}.

\begin{corr}\label{thm:rate_for_box}
For any constant $\boundparam > 0$, there exists a sequence of sieve maximum likelihood estimators $\mathscr{E}$ such that if $p \in \mathcal{D}_{\mathrm{Box,Lip,\boundparam}},$ the ULRT rejection time $\tau_\alpha$ is bounded in expectation as $\mathbb{E}[\tau_\alpha] = \widetilde{O}( d_H(p, \lc)^{-2(d+3)}).$
\end{corr}
The proof is in \S\ref{appx:bounded_unit_box_regret}. 
We remark that the above rates adapt to the extent to which the underlying law $p$ violates log-concavity in the sense that the time-scales of rejection are driven by $d_H(p, \lc).$ Indeed, this represents an important advantage of sequential tests as opposed to batched tests, in that validity is retained, and detection is guaranteed at an adapted time-scale.

\noindent \emph{On tightness.} We note that the exponents of Theorem~\ref{thm:rate}, and in particular, Corollary~\ref{thm:rate_for_box} are likely loose for the problem of testing log-concavity. This is an artefact of the analysis; for instance, the slow rate in Corollary~\ref{thm:rate_for_box} is largely determined by the rate requirements for estimating Lipschitz laws on the unit box, which arises due to the $\pi_t$ terms present in Theorem~\ref{thm:rate}. It is possible that this aspect can be improved, since nothing neccessitates that we use an estimator that captures the underlying density $p$ well.

Indeed, instead of analysing the prediction regret with respect to $p$ itself, we could decompose $R = \sigma_t(q) - \rho_t(\mathscr{E};q)$ for some other law $q$, perhaps lying in a smaller class of densities $\mathcal{Q}$ than those possible for $p$. As long as (i) $\mathscr{E}$ does as good a job at prediction under the log-loss as any law in $\mathcal{Q},$ and (ii) no matter what $p \not\in \lc$ is, there is a law in $\mathcal{Q}$ that is `closer to' $p$ than any law in $\lc,$ a similar analysis should be possible, although this requires possibly subtle off-model control on the behaviour of $\mathscr{E}$, as well as a careful choice of $\mathcal{Q}$ itself to control the relative values of distances such as $d_H(p, \mathcal{Q})$ and $d_H(p, \lc)$. One such approach which appears promising for log-concave laws is to exploit $s$-concave densities to play the role of $\mathcal{Q},$ which are particularly attractive since they form a rich extension of the class of log-concave laws, but nevertheless enjoy identical minimax MLE convergence rates as them \cite{han2016approximation, han2021set}.

\section{Algorithmic Proposal, and Simulation Study}\label{sec:sim}

We now proceed to algorithmically describe the ULR e-process based test for log-concavity, and investigate the behaviour of a concrete implementation of the same on a simple parametric family. 

\subsection{Computational Aspects, Batching, and a Concrete Testing Algorithm}

Under specification of the sequential estimators $\mathscr{E}$, and a method for fitting the log-concave MLE, the statistic $R_t$ is explicitly computable, and thus naturally leads to implementations. While the e-process is powerful against wide classes of alternatives, its implementation suffers from a fundamental computational issue, that arises due to the recomputation of $\hat{q}_t$ and $\hat{p}_t$ in each round. This cost grows superlinearly with $t$ since since the entire denominator $\prod \hat{p}_t(X_s)$ must be evaluated on the entirety of the stream, and the cost of estimating this $\hat{p}_t$ is itself superlinear in the number of samples $t$. A second issue arises upon increasing the data dimensions $d,$ since computational costs of estimating $\hat{p}_t$ grow quite fast with this. Even though polynomial in $d$ algorithms exist for computing the log-concave MLE \cite{axelrod2019polynomial}, the fastest available method for this is typically hundreds of times slower when processing $\sim 100$ points in even the modest $d = 5$ when compared to the time needed to process the same sized dataset for $d = 1$ \cite{rathke2019fast}.  We address this issue by exploiting batching to reduce the computational load, which makes computations viable in the moderate $d \le 4$. The idea is to wait to accumulate $\batchinterval >1$ fresh samples before recomputing $R_t,$ rather than updating it at every round.\footnote{Note of course, that for our simulation study, the repetition of simulations required to study power and size mean that we only implement our fully nonparametric test for up to $d = 4$. Nevertheless, even this is reasonable to run for $d = 6,$ wherein a single run over a horizon of $100$ steps takes about $20$s.}

Let us point out that such batched updates still retain the e-process property, and thus validity, as long as the $\hat{q}_{t-1}$ remain nonanticipating over the entire batch. Concretely, we may set a schedule, captured by an increasing sequence of times $\mathscr{T} = \{t_k\},$ and evaluate the statistic \[ R_t(\mathscr{T}) = R_{t-1}(\mathscr{T}) \indi\{t \not\in \mathscr{T} \} + \sum \indi\{t \in \mathscr{T}\} \prod_{j \le k(t)} \prod_{s = t_{j-1}}^{t_j} \frac{\hat{q}_{t_{j-1}}(X_s)}{ \hat{p}_{t_k}(X_s)},\] where $k(t) = \max\{ k : t_k \le t\}.$ In words, the schedule divides streams into a sequence of batches of size $t_k - t_{k-1},$ and each time a new batch is accumulated, we evaluate a new estimate $\hat{q}$ on the previous batches, and re-evaluate the log-concave MLE on the entirety of the data seen. This process continues to be dominated by a batched version of $F_t(P),$ which retains the martingale property under $P^\infty,$ thus yielding validity. The simplest viable schedule is to set $t_k = k \batchinterval$ for a constant `batching interval' $\batchinterval$. This effectively boils down to testing the log-concavity of $p^{\otimes \batchinterval},$ which is valid since tensor products of log-concave laws remain log-concave. 

Notice that such batching may result in a reduction in power. For instance, rejection can only occur at the time $t_k,$ and further the statistic may be deflated because data points with a large signal may be `washed-out' due to milder behaviour across the remainder of the batch. Nevertheless, we find in simulation studies that this drop in power is nominal, and comes at the cost of a significant improvement in runtime. 

With this in hand, we can provide an explicit algorithmic description of our test below.

\begin{algorithm}[H]                %
    \caption{Log-Concave Universal Likelihood Ratio Test }\label{alg:log_concave_e_process_test}
    \begin{algorithmic}[1]
        \State \textbf{Input:} Batching schedule $\{t_k\}_{k = 1}^\infty$ with $t_1 \ge d+1,$ estimator $\mathscr{E}$, level $\alpha$.
        \State \textbf{Initialise:} $R_t \gets 1$ for $t \le t_1$, $K \gets 1$, $N_1 = 1$, $t \gets 1.$
        \While{$R_t < 1/\alpha$}
            \If{$t = t_K$}
                \State $\hat{q} = \mathscr{E}(X_1^{t- t_{K-1}}).$
                \State $N_t \gets N_{t-1} \cdot \prod_{s = t_{K-1} + 1}^{t_K} \hat{q}(X_s).$
                \State $\hat{p} \gets \mathscr{L}(X_1^{t_K}).$
                \State $R_t \gets N_t \cdot \left( \prod_{s = 1}^{t_K} \hat{p}(X_s)\right)^{-1}.$
                \State $K \gets K + 1.$
            \Else
                \State $R_t \gets R_{t-1}$.
                \State $N_t \gets N_{t-1}$.
            \EndIf
            \State $t \gets t+1$.
        \EndWhile    %
    \end{algorithmic}%
\end{algorithm} 

\subsection{Evaluating the ULR E-Process Test}

We investigate the behaviour of the test of Algorithm~\ref{alg:log_concave_e_process_test} on the following simple test-bed family of laws, where $\mathbf{e}_d = \mathbf{1}_d/\sqrt{d}$ is the unit vector along the all-ones direction in $\mathbb{R}^d,$. \[ p(x;\mu,d) := \frac{1}{2 (2\pi)^{d/2}}\left( \exp{ - \|x - \tfrac{\mu}{2}\mathbf{e}_d\|^2/2} + \exp{ - \|x + \tfrac{\mu}{2} \mathbf{e}_d\|^2/2}\right), \] i.e., balanced two component Gaussian mixture laws with means $\pm \tfrac{\mu}{2} \mathbf{e}_d$ and identity covariance. The norm of the mean-difference is precisely $\mu$, which we assume without loss of generality to be nonnegative. A small modification of this family of laws was also used as a test-bed for the non-sequential test proposed by Dunn et al.~\cite{dunn2021-logConcaveTesting}. 

These laws are extremely convenient for proof-of-concept investigation of tests of log-concavity. Indeed, observe that up to a rotation, the $d$-dimensional law is a tensorisation of the one-dimensional law with a log-concave law (specifically a standard Gaussian in $d-1$-dimensions). Since log-concavity properties are invariant under rotations, and since the log-concave M-projection of product laws is a product of the marginal log-concave M-projections \cite{logconcavity_survey}, this gives a very simple characterisation of the log-concavity properties of this law. Concretely, the distance from log-concavity is purely a function of the norm of the mean-difference $\mu,$ and $p(x;\mu,d)$ is log-concave if and only if $\mu > 2.$ These laws thus give us a simple way to check both the size and power of the test statistics, as well as study the effect of increase in dimension on the power.

Finally, we give details of the simulations. All data reported is a mean over 100 runs of each experiment. All simulations are run up to $100$ time steps, which is mainly for computational practicality. Note thus that our size estimates are systematically lower than the true size (with infinte horizon). We run the case of $d = 1,$ which is computationally the cheapest, over the longer horizon of $500$ time steps to illustrate that not much changes in this case, at least as regards the empirical validity of our test. For $d =1,2,3$, the tests are batched at an interval of $\batchinterval = 20,$ while for computational practicality the test is batched at $\batchinterval = 25$ for $d=4$. These are significant fractions of the time horizon studied, but do not significantly lower power, at least for $d= 1,$ as demonstrated by explicit simulation. 

All code is implemented in $\texttt{R}$. The nonparametric estimator used for $\mathscr{E}$ is the kernel density estimate as implemented in the \texttt{ks} package \cite{chacon2018multivariate}, and the log-concave MLE used is either from the \texttt{logConDens} package \cite{dumbgen2011logcondens} for $d = 1$ or the \texttt{fmlcd} package ($d = 2,3,4$) \cite{rathke2019fast}. We note that the latter is not guaranteed to return the log-concave MLE since it optimises a non-convex approximation to the program defining the same. However, we find that compared to alternatives like the \texttt{logConcDEAD} package \cite{cule2010maximum}, the \texttt{fmlcd} implementation retains similar validity and power, but runs significantly faster.  We also investigate using parametric Gaussian mixture model fits to illustrate the effect of inefficiency in $\mathscr{E}$ on power, for which we use the EM algorithm as implemented in the \texttt{mclust} package \cite{scrucca2016mclust}. All simulations were executed on an AMD Ryzen 5650U processor, a medium range CPU for a laptop computer.

\subsubsection{Fully Nonparametric tests}

Figure~\ref{fig:fully_nonparametric_tests} shows the behaviour of our instantiation of the algorithm with the fully nonparametric approach of using kernel density estimators as $\mathscr{E}$ over $p(\cdot;\mu,d)$ for $d = 1, 2$ as $\mu$ is varied, run at the size $\alpha = 0.1$, with $\batchinterval =  20$ for $d \in \{1,2,3\}$ and $\batchinterval =  25$ for $d = 4.$ We plot five traces which record the fraction of runs out of $100$ independent runs that the test rejected the null hypothesis at times smaller than $20, 40, 60, 80,$ and $100,$ where $100$ was the horizon over which the test was run. 

There are three major observations. Firstly, we observe that the test shows excellent validity. Indeed, the null hypothesis holds true for $\mu \le 2,$ and the test does not reject more than a $0.02$ fraction of the time in either case in this scenario. Secondly, we observe that at least for sufficiently large $\mu$, all of the tests do reject within $100$ steps. Finally, we notice that the power sharply drops as $d$ increases. To concretely discuss this, let $\mu_*(d)$ be defined as the smallest value of $\mu$ for which $\mathbb{P}_{X \sim p(\cdot;\mu,d)}(\tau_{0.1} \le 100) = 0.9$. The plots in figure~\ref{fig:fully_nonparametric_tests} give us estimates of $\mu_*(d)$, which increase sharply with $d$---from about $6$ in $d= 1$ to over $1000$ in $d = 4$.\footnote{With pilot simulations in $d=5,$ we observe that $\mu_*(5) \approx 1500.$ We note that these simulations were already too costly, in terms of time, to implement completely for $d = 5,$ due both to the increased costs of fitting MLEs in higher dimensions, and due to the fact that as rejection rates decrease with dimension, more runs need to be executed over the whole horizon, which extends the total cost of the experiment. We hope to implement the method on larger computational resources for such moderate $d$s.}

This reduction in power is perhaps expected, given the considerable deterioration in the nonparametric estimation rates with $d$. Nevertheless, we may question how much of the above decay in power is driven by the inefficiencies in fitting log-concave MLEs, and how much accrues due to the inefficiency of kernel density estimation. We investigate this effect in \S\ref{sec:oracle_tests} by studying Oracle tests.

\paragraph{Longer Run for $d = 1$.} To show that the validity persists over longer time horizons, we implement the fully nonparametric method over $500$ time steps for $d = 1,$ using $\batchinterval =  50.$ Observe in Figure~\ref{fig:longer_runs} that rejection under the null $\mu\le 2$ is well controlled even at this increased timescale, while rejection rates steadily improve as the horizon grows, although the improvement is somewhat marginal over the horizon of $500$ versus $200$.

\begin{figure}[t]
    \centering
    \includegraphics[width = 0.45\linewidth]{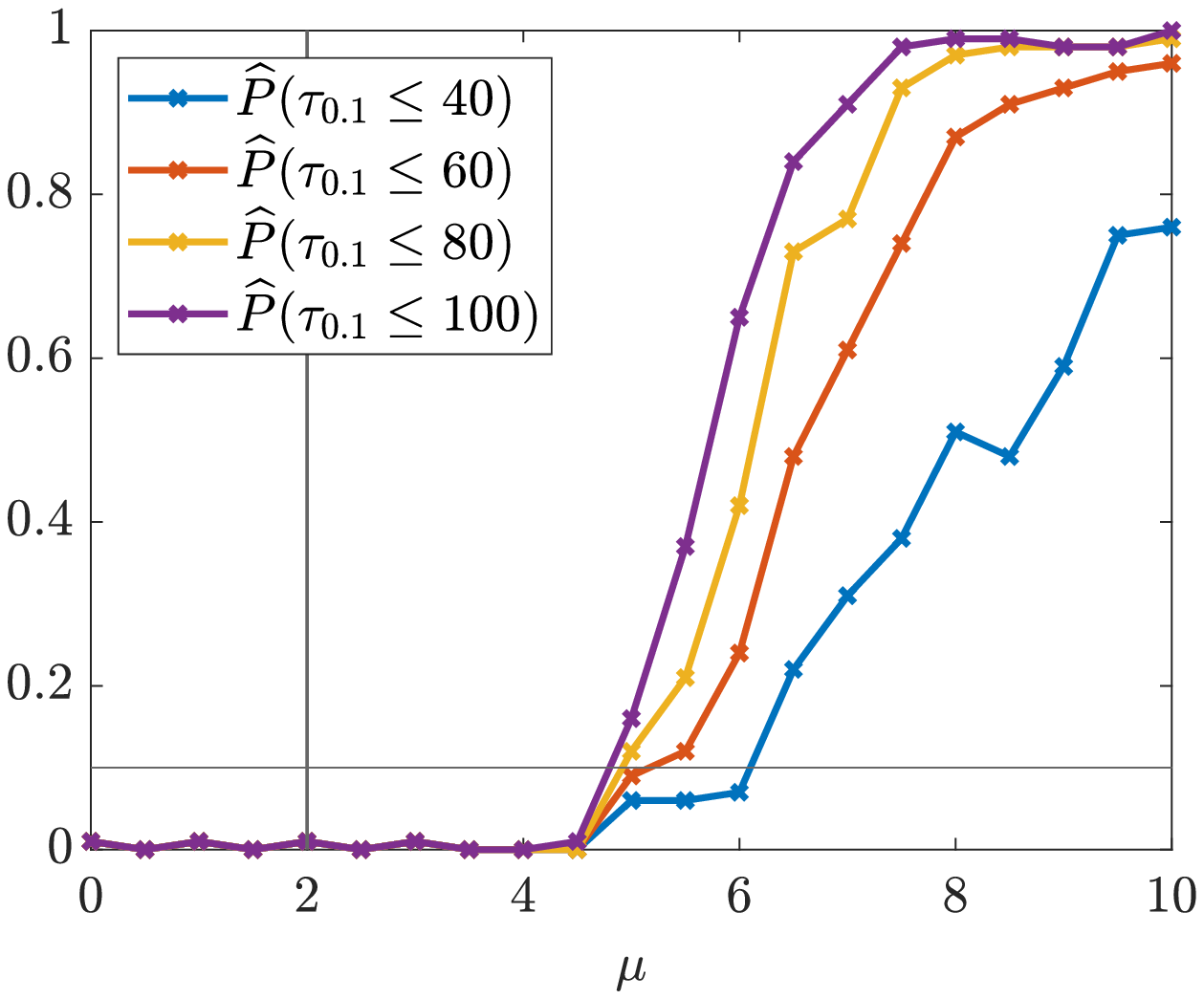}~
    \includegraphics[width =0.45\linewidth]{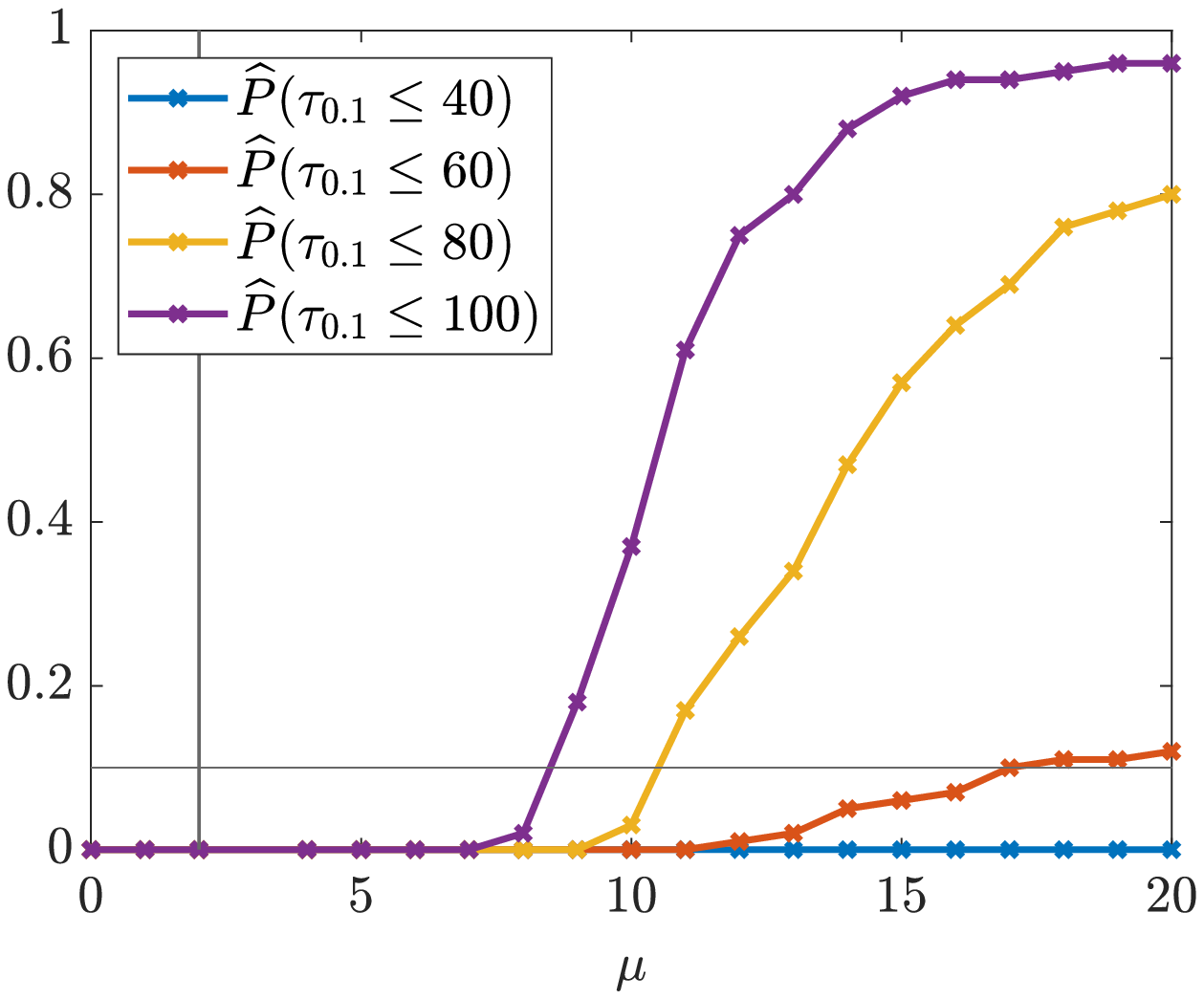} \vspace{-.25\baselineskip}\\
    \includegraphics[width = 0.45\linewidth]{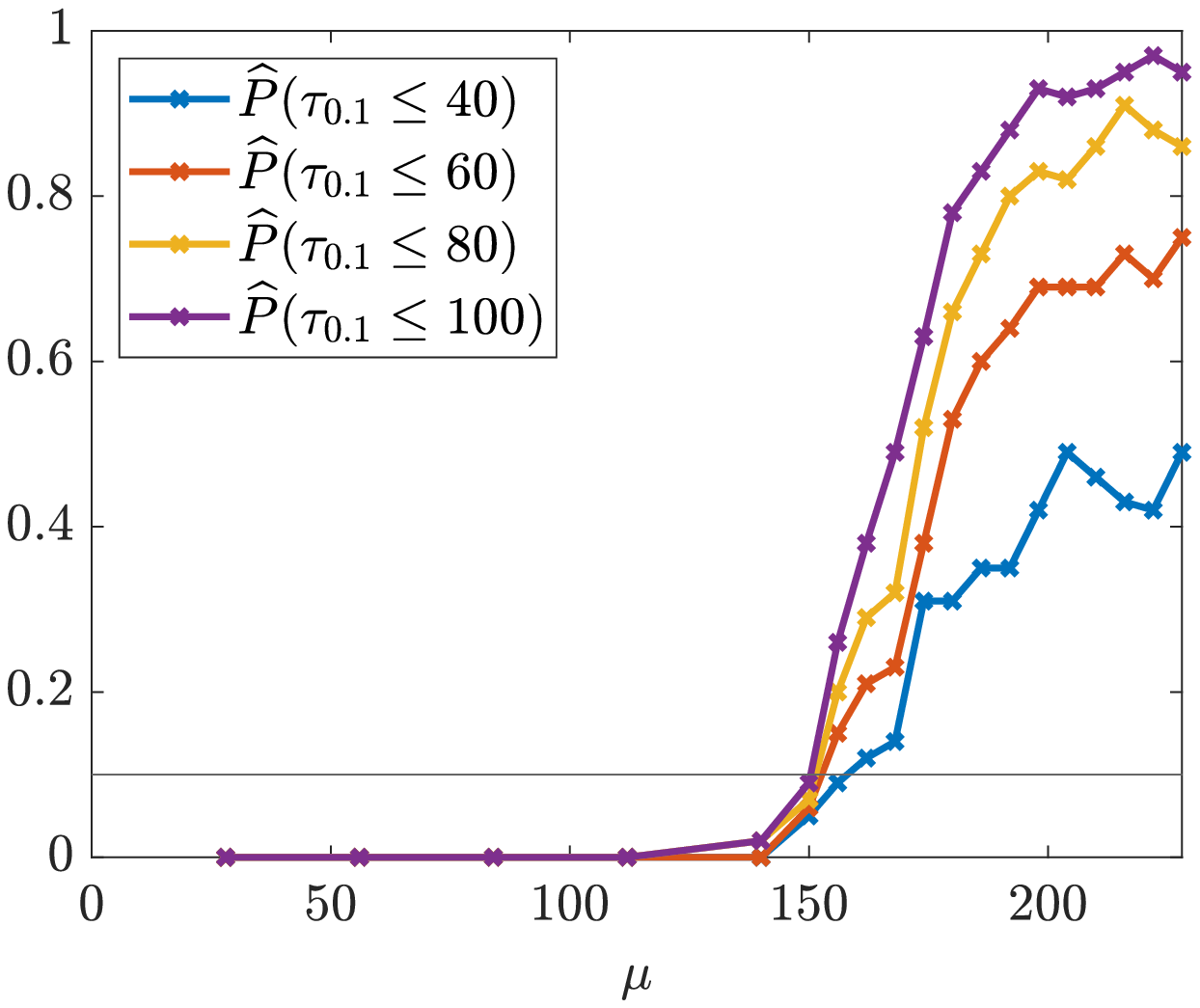}~
    \includegraphics[width =0.45\linewidth]{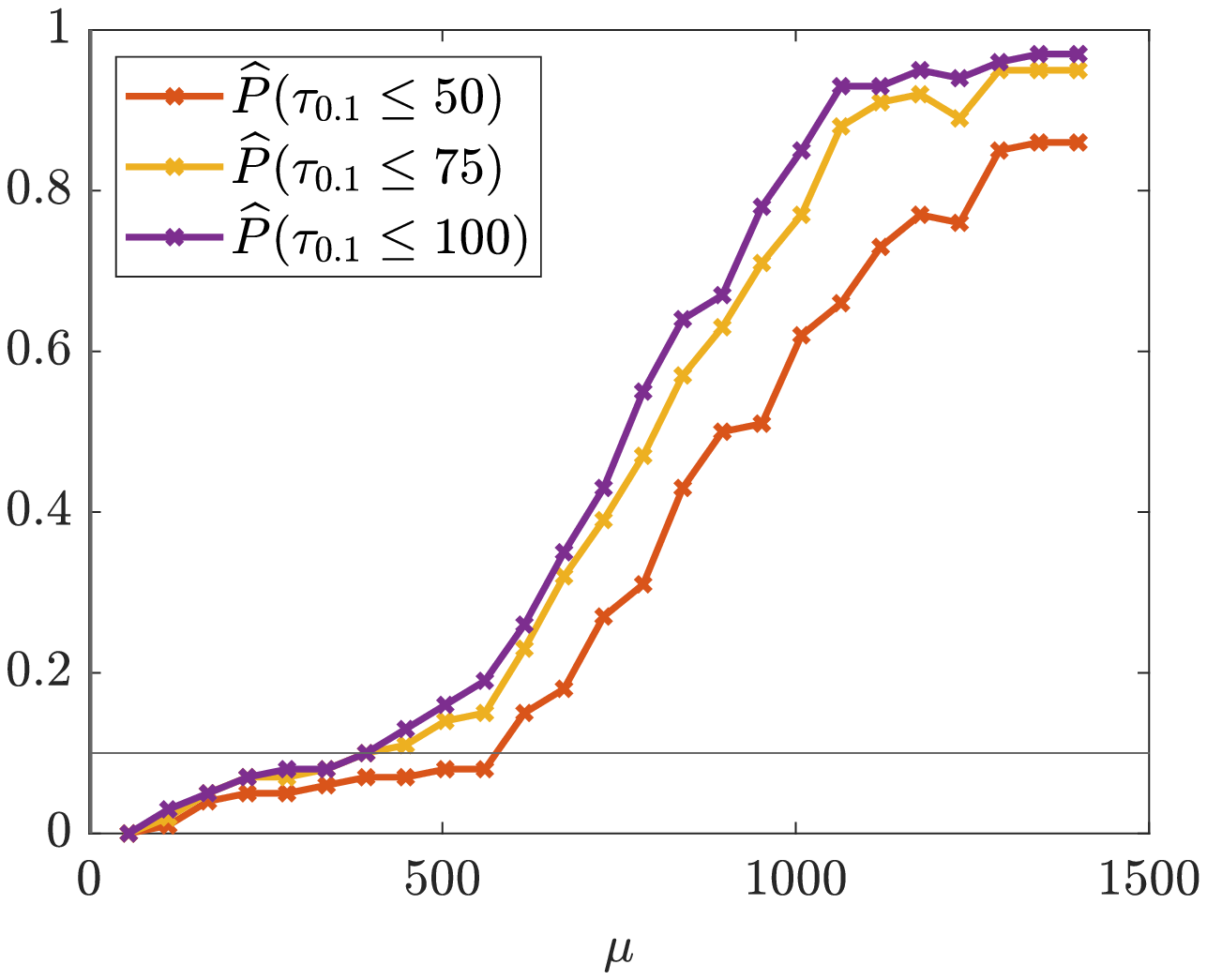} \vspace{-\baselineskip}
    \caption{\textbf{Performance of the fully nonparametric test.} Empirical rejection rates (over $100$ simulations) at $\alpha = 0.1$ versus the mean difference $\mu$ for fully nonparametric test implementations over four cases: $d=1, \batchinterval =  20$ (top left); $d = 2, \batchinterval =  20$ (top right); $d = 3, \batchinterval =  20$ (bottom left) ; $d = 4, \batchinterval =  25$ (bottom right). The thin horizontal line plots the level $\alpha = 0.1,$ and the vertical line marks $\mu = 2$ since $p(\cdot;\mu,d) \in \lc \iff \mu \le 2.$  Observe the strong validity properties in all plots, as well as the deterioration of power in higher dimensions, as signalled by the sharp increases in the scales of the X-axis.}\vspace{-1.5\baselineskip} %
    \label{fig:fully_nonparametric_tests}
\end{figure}

\begin{figure}[H]
    \centering
    \begin{minipage}{0.48\linewidth}
    \includegraphics[width = \linewidth]{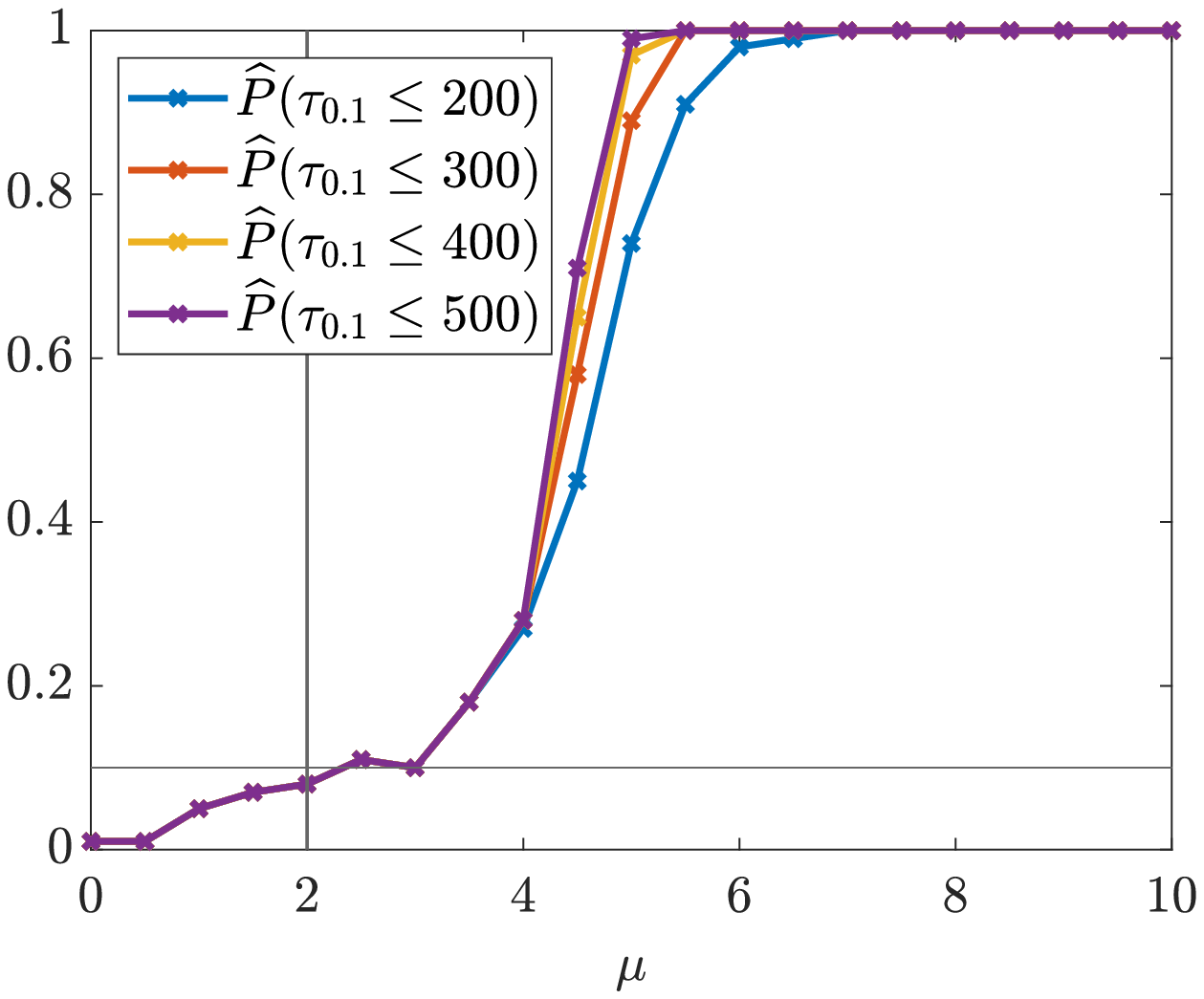}\vspace{-\baselineskip}
    \caption{\textbf{Performance of the fully nonparametric test over long horizons.} Empirical rejection rates (over $100$ simulations) in the setting of Figure~\ref{fig:fully_nonparametric_tests} for $d = 1,$ ran over a horizon of length $500$ with $\batchinterval = 50$. Observe that the validity persists over this longer horizon, and that power improves for $\mu > 2.$}%
    \label{fig:longer_runs}
    \end{minipage}~\hspace{0.019\linewidth}~
    \begin{minipage}{0.48\linewidth}
        \includegraphics[width = \linewidth]{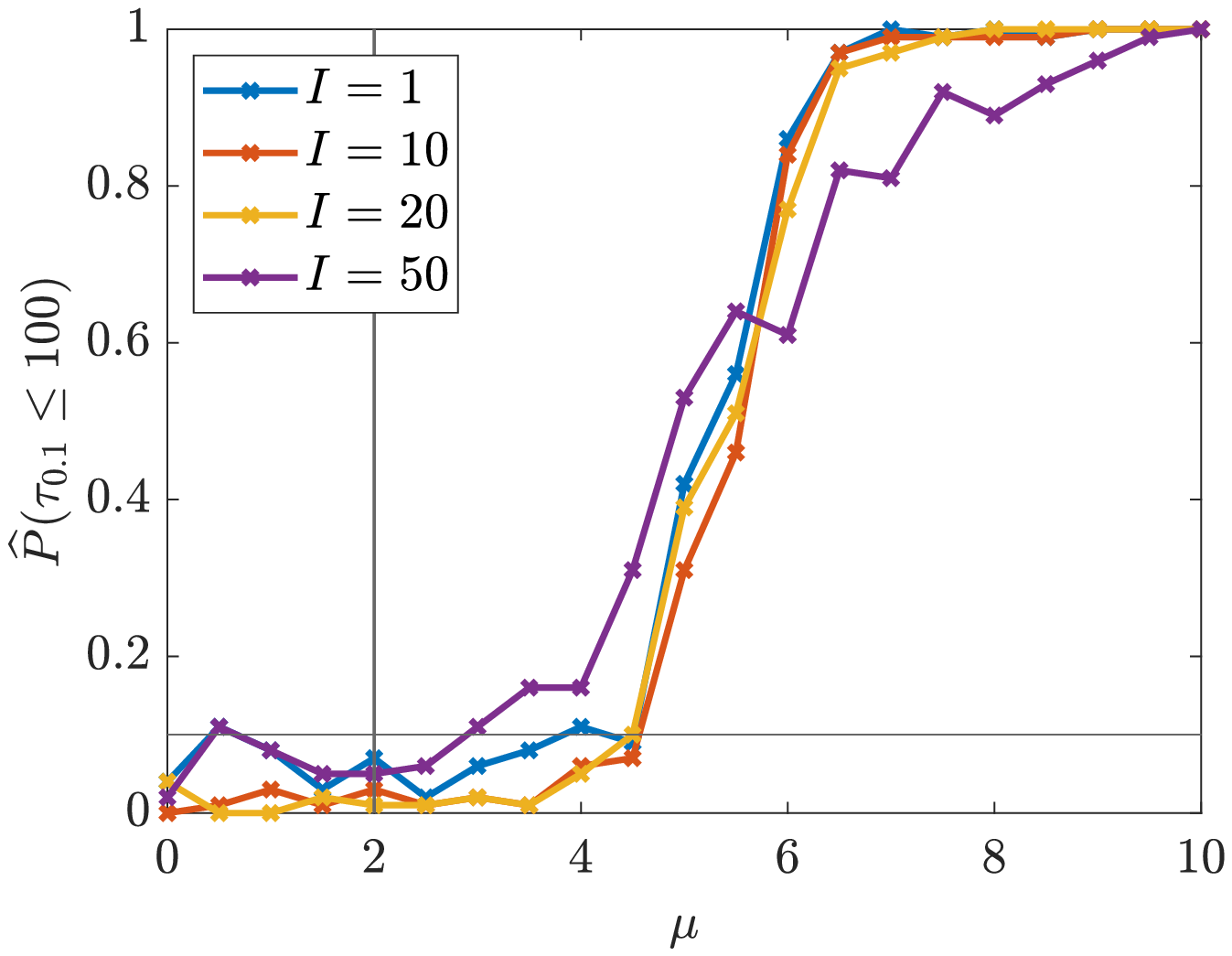}\vspace{-\baselineskip}
    \caption{\textbf{Effect of $\batchinterval$ on the fully nonparametric test. }Empirical rejection rates (over $100$ simulations) in the setting of Figure~\ref{fig:fully_nonparametric_tests} for $d = 1$ and with varying $\batchinterval \in \{1,10,20,50\}$. Observe that the rejection rates for $\batchinterval = 10,20$ are roughly the same as for $\batchinterval = 1,$ while $\batchinterval = 50$ suffers large losses.}%
    \label{fig:effect_of_batching}
    \end{minipage}
\end{figure}

\paragraph{Effect of Batching Interval.} As seen in Figure~\ref{fig:effect_of_batching},  batch sizes of $\batchinterval =  10$ and $20$ have a mild effect on the rejection rates under alternate setting ($\mu \ge 2$) when compared to the direct $\batchinterval =  1$. Interestingly, note that $\batchinterval =  20$ does somewhat better than $\batchinterval =  10$ in the setting of moderate $\mu$ (the range $4-6$), and slightly loses power for larger intermediate $\mu$ (the range $6-8$). In turn, the no batching setting, i.e., $\batchinterval =  1,$ is observed to suffer deterioration in its size ($\mu < 2$), although this remains at an acceptable level. 

The large batch size $\batchinterval = 50$ suffers the same validity issues as $\batchinterval =  1,$ but does even better than it for small but non-null values of $\mu$ ($2$-$5$). Power considerably deteriorates for larger $\mu$ ($5$-$10$). While it is unclear how much of this is an artefact of the fact that the length of the horizon is only $100$, and how much is directly due to the larger batching interval, the fact that $\batchinterval =  10,20$ perform well suggests that so long as the batching interval is a relatively small fraction of the horizon length, the loss in power is not too bad.

\subsubsection[Oracle Tests, and the Effect of the Quality of the Alternate Sequential Estimator]{Oracle Tests, and the Effect of the Quality of $\mathscr{E}$}\label{sec:oracle_tests}

\begin{figure}[tb]
    \centering
    \includegraphics[width = 0.45\linewidth]{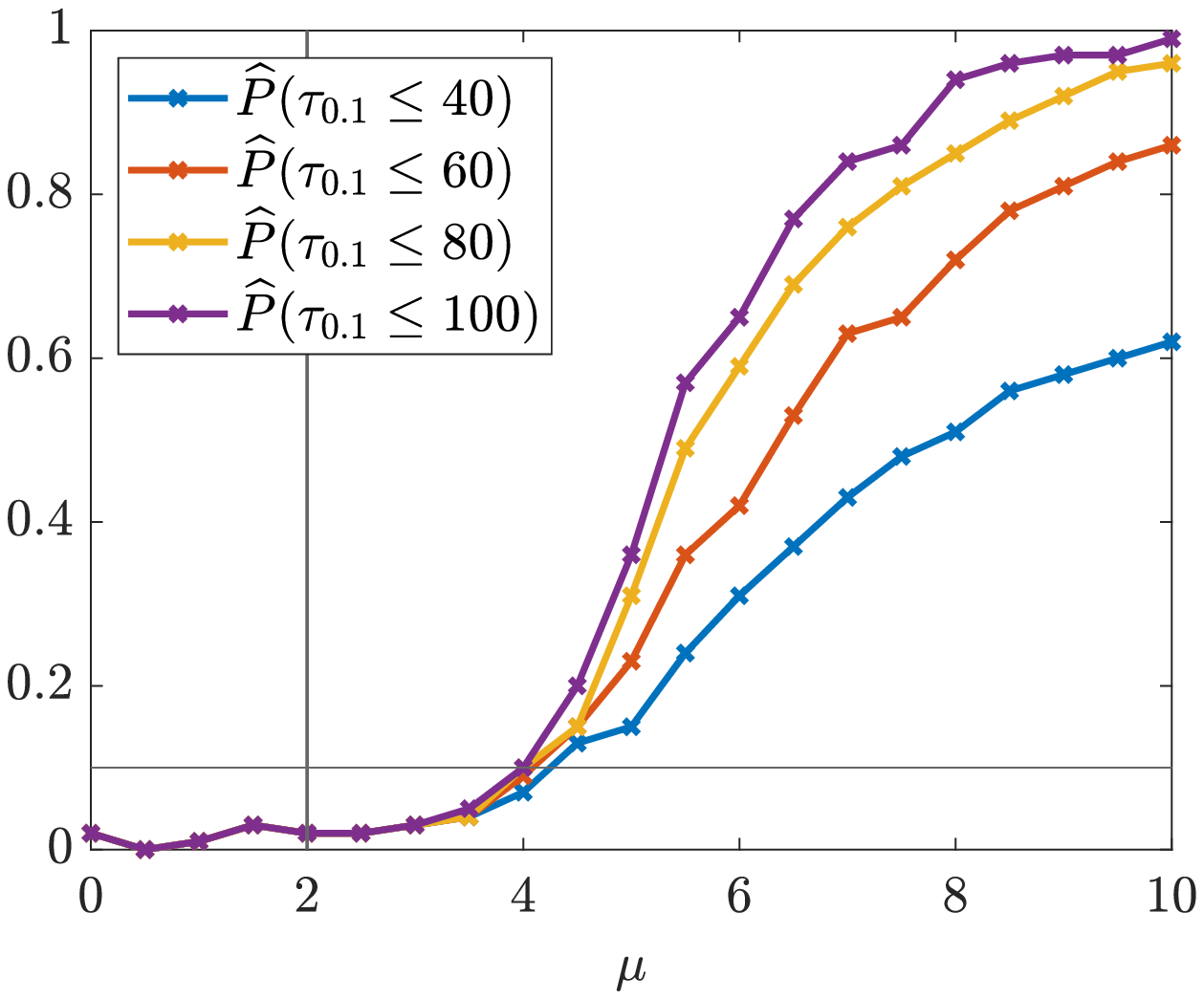}~
    \includegraphics[width = 0.45\linewidth]{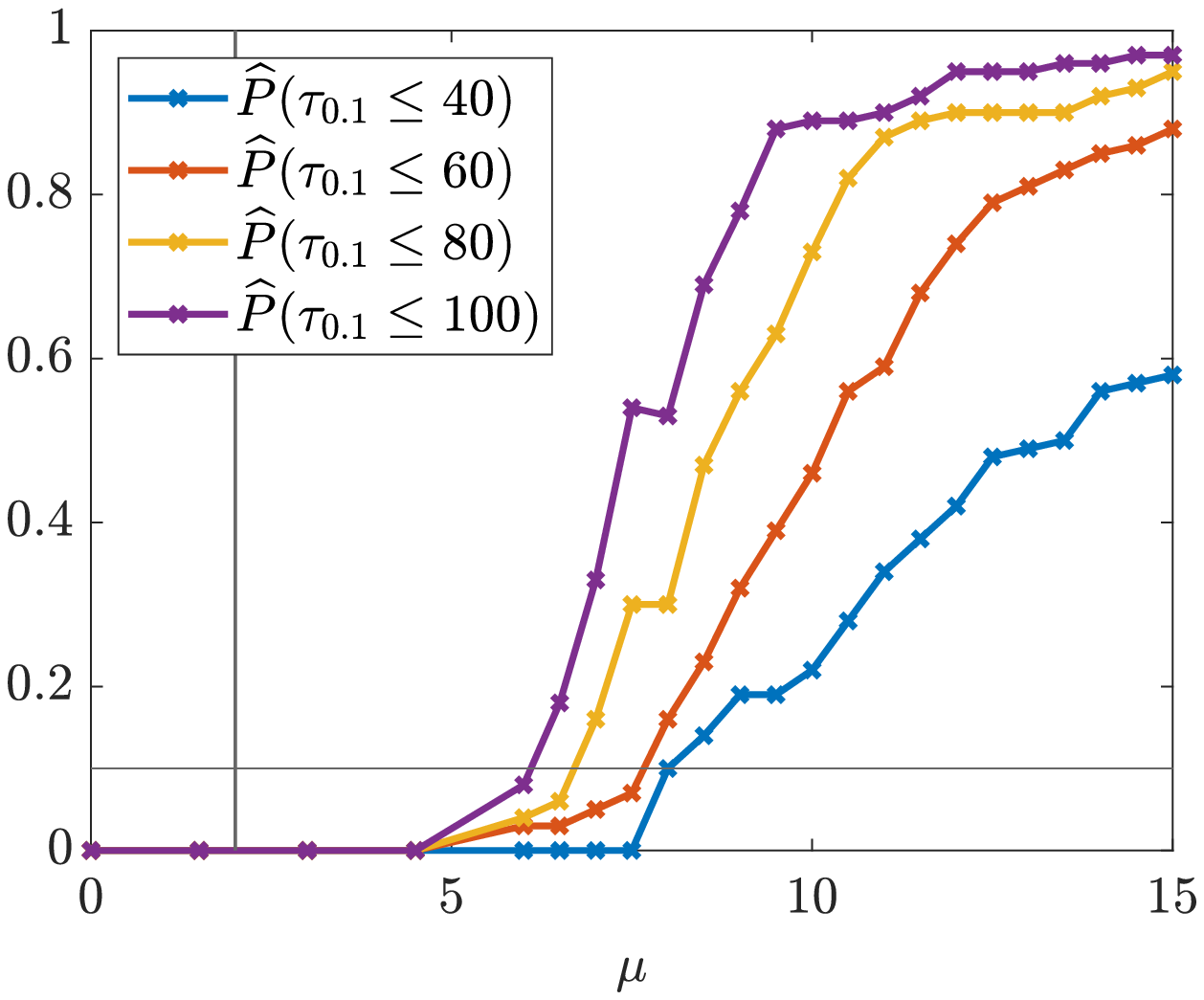} \vspace{-.5\baselineskip}\\ 
    \includegraphics[width =0.45\linewidth]{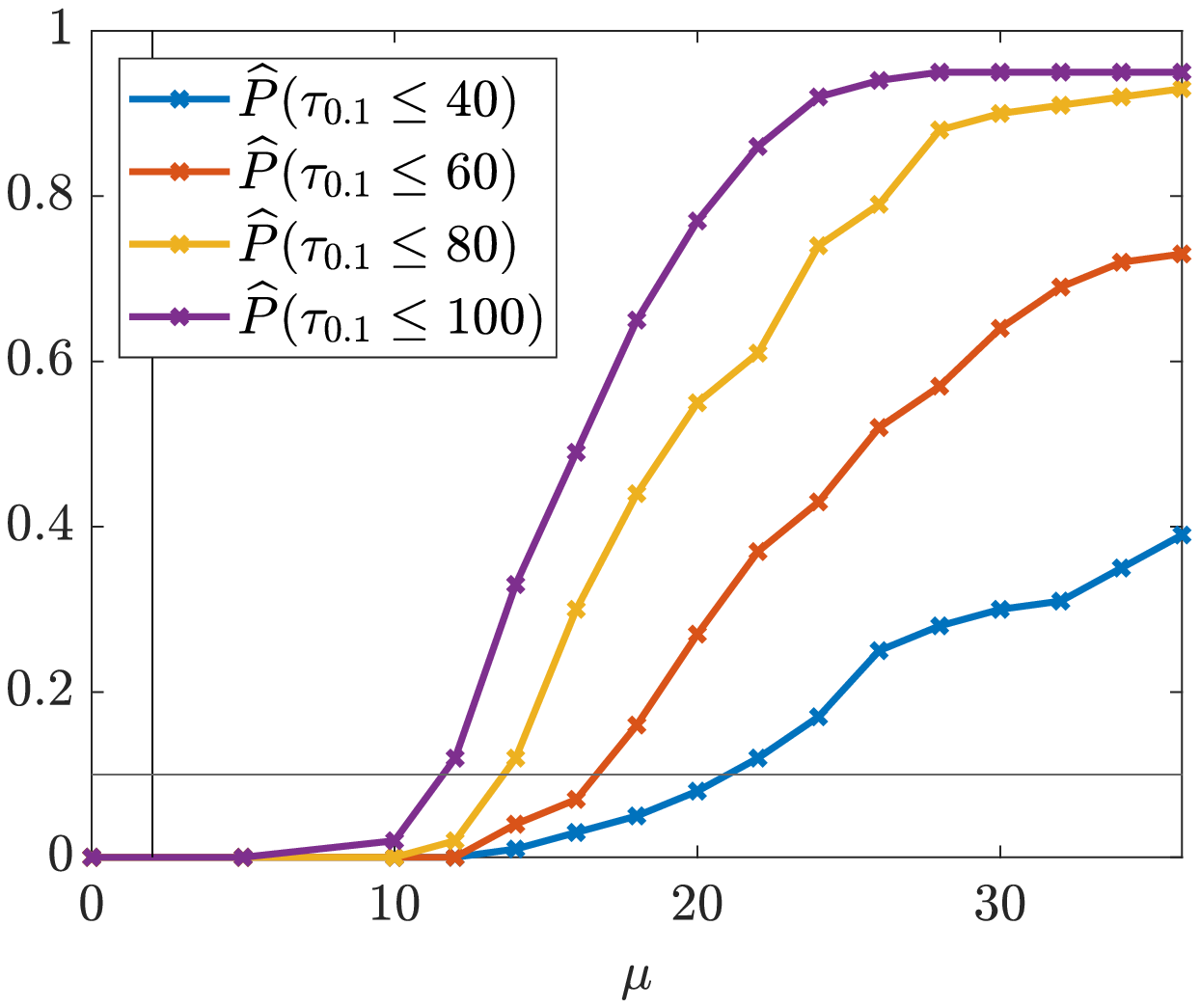}~
    \includegraphics[width =0.45\linewidth]{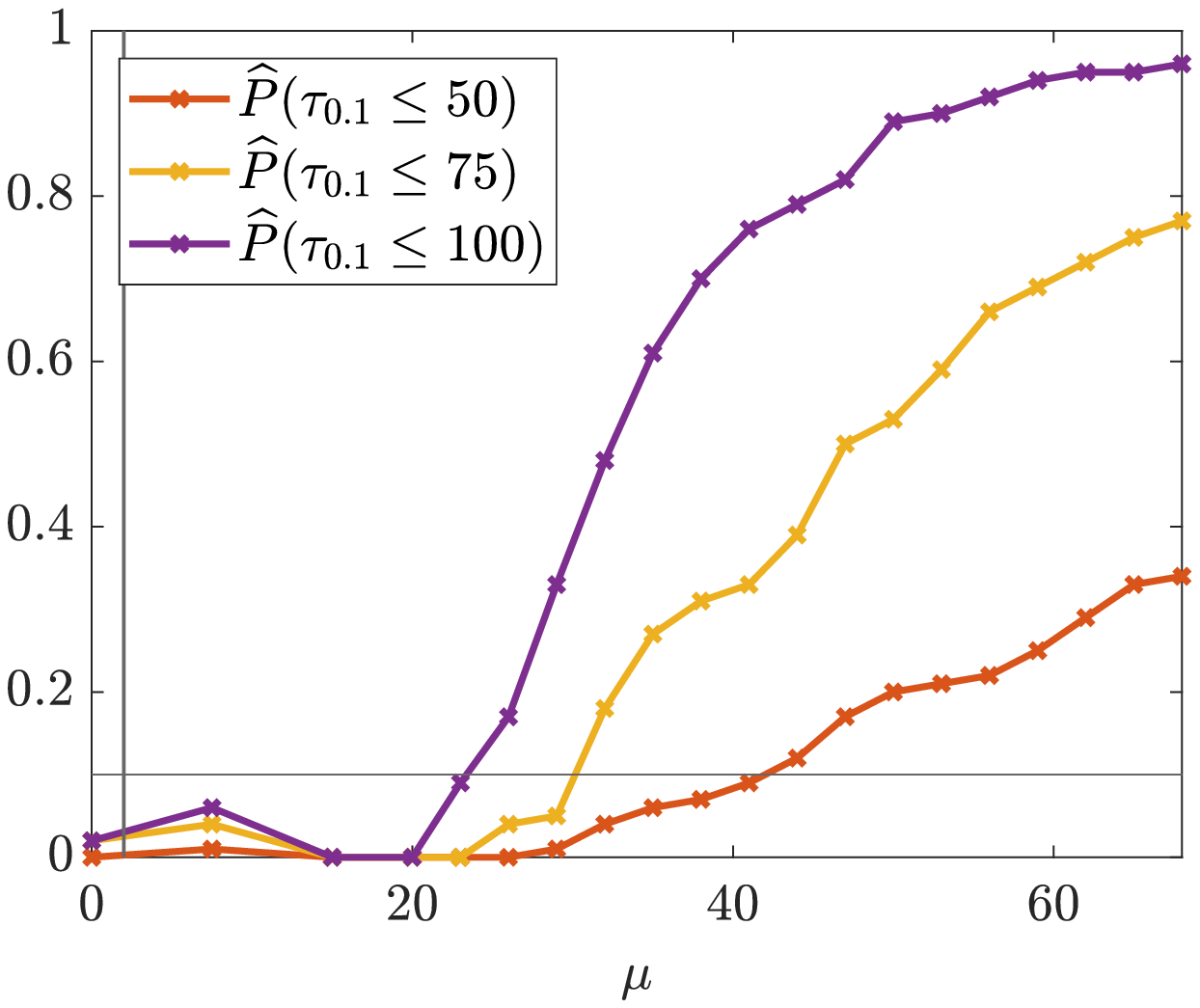}~ 
    \vspace{-\baselineskip}
    \caption{\textbf{Performance of the partial oracle test.} Empirical rejection rates (over $100$ simulations) at $\alpha = 0.1$ versus the mean difference $\mu$ for the partial oracle test implementations over four cases: $d=1, \batchinterval =  20$ (top left); $d = 2, \batchinterval =  20$ (top right); $d = 3, \batchinterval =  20$ (bottom left) ; $d = 4, \batchinterval =  25$ (bottom right). Observe that in each plot, the power improves starkly relative to the fully nonparametric test (Figure~\ref{fig:fully_nonparametric_tests}), as indicated by a strong contraction of the scale of the X-axis, especially in higher dimensions.}
    \label{fig:partial_oracle}
\end{figure}

\begin{figure}[tb]
    \centering
    \includegraphics[width = 0.45\linewidth]{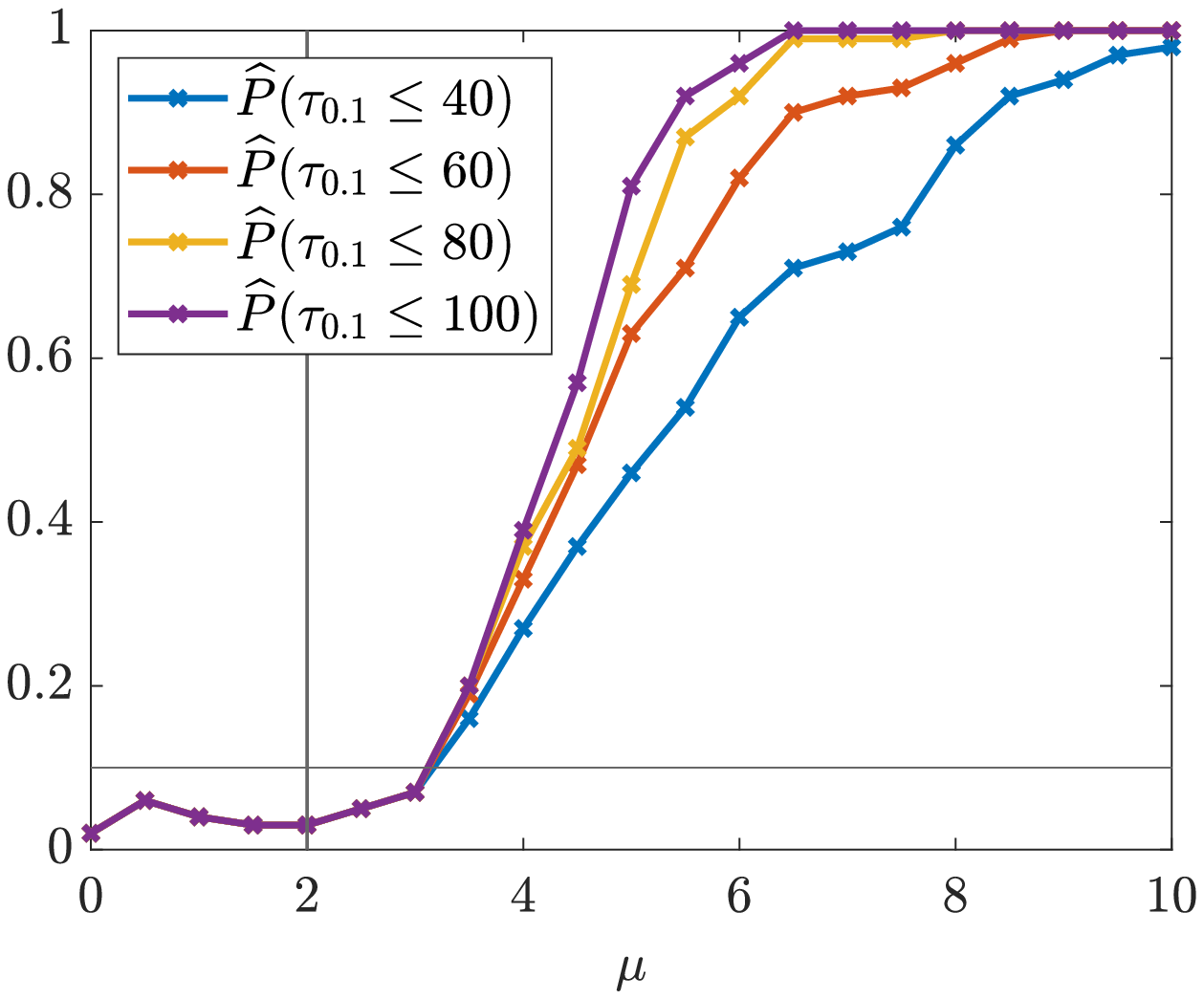}~
    \includegraphics[width =0.45\linewidth]{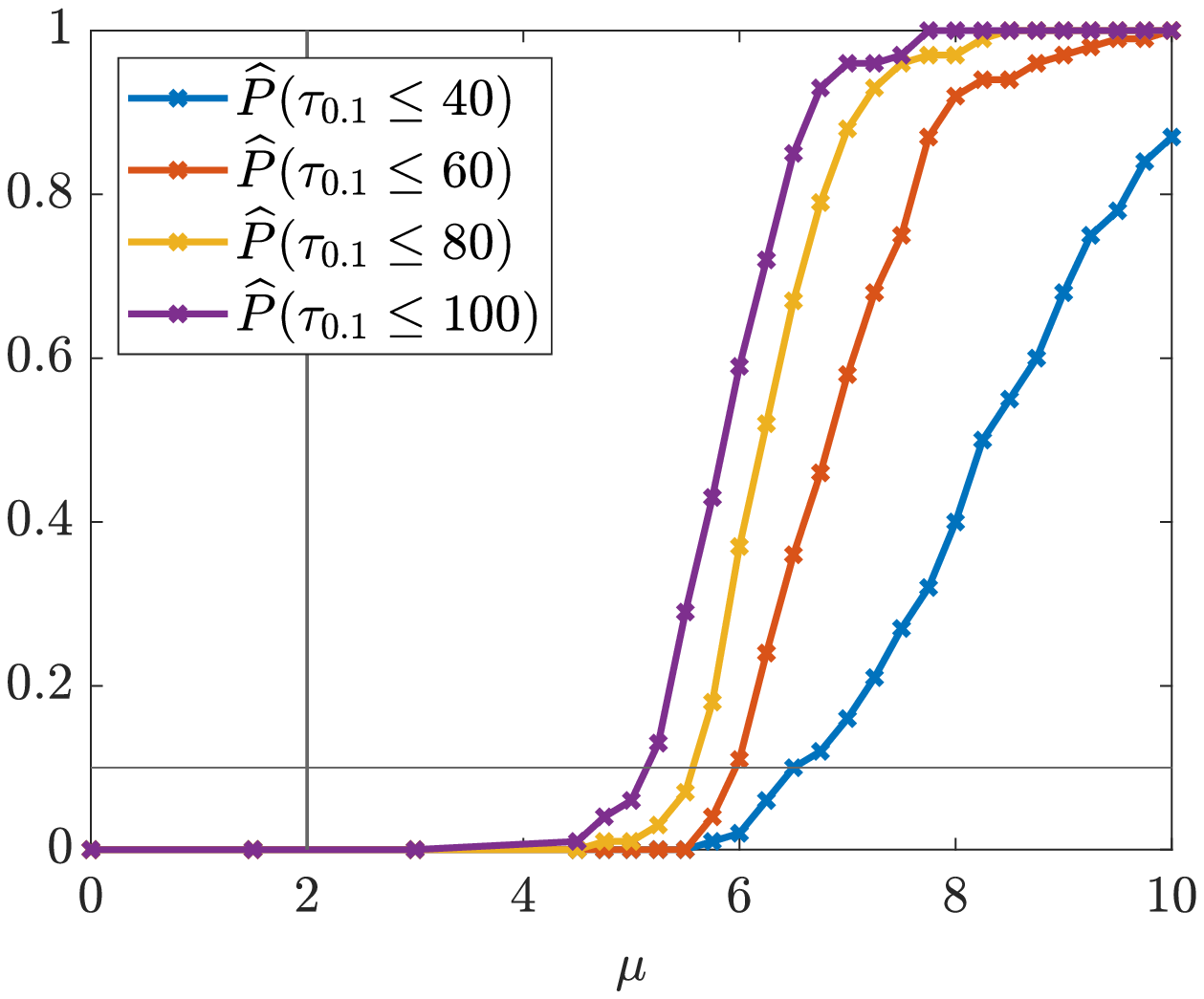} \vspace{-\baselineskip}\\
    \includegraphics[width =0.45\linewidth]{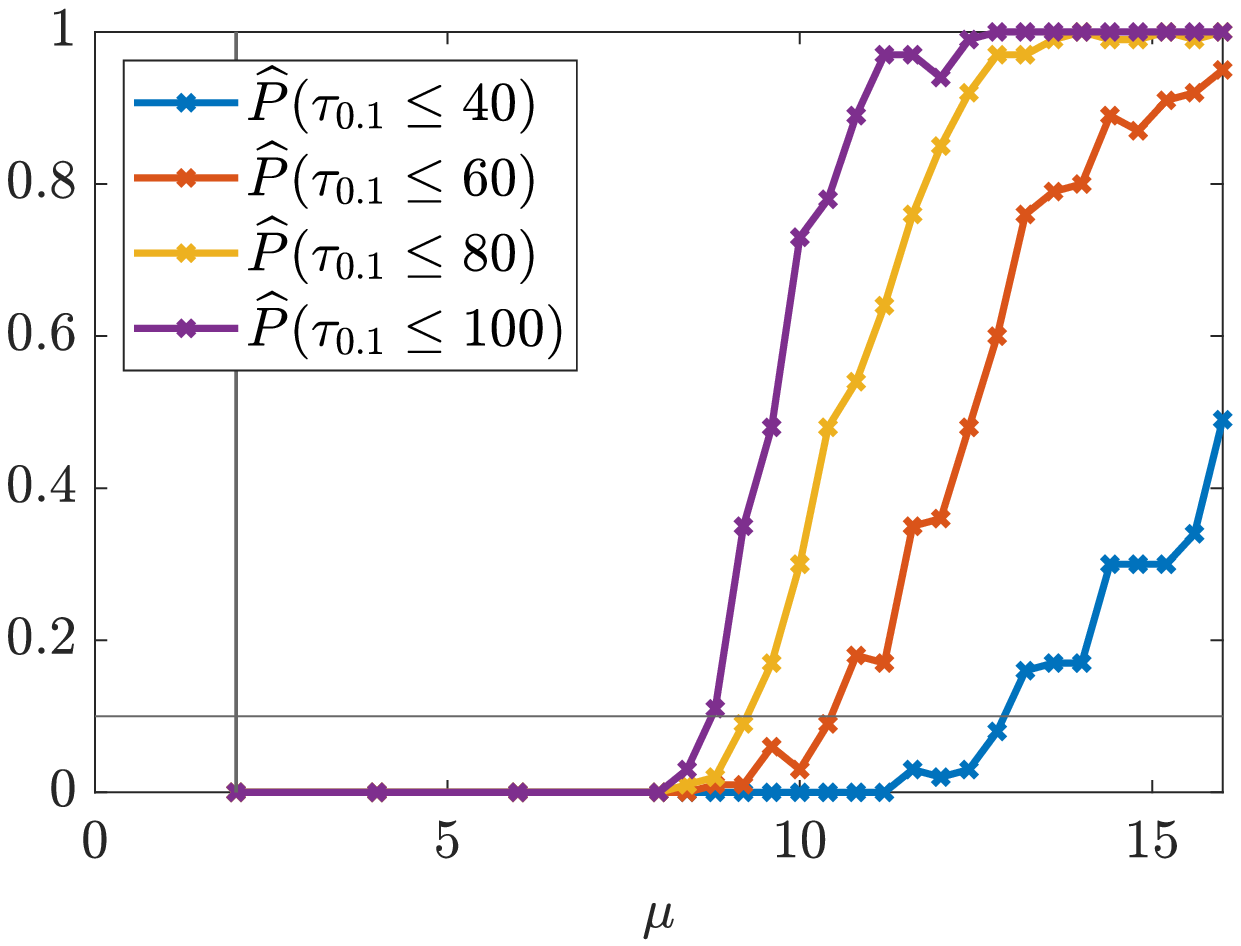}~
    \includegraphics[width =0.45\linewidth]{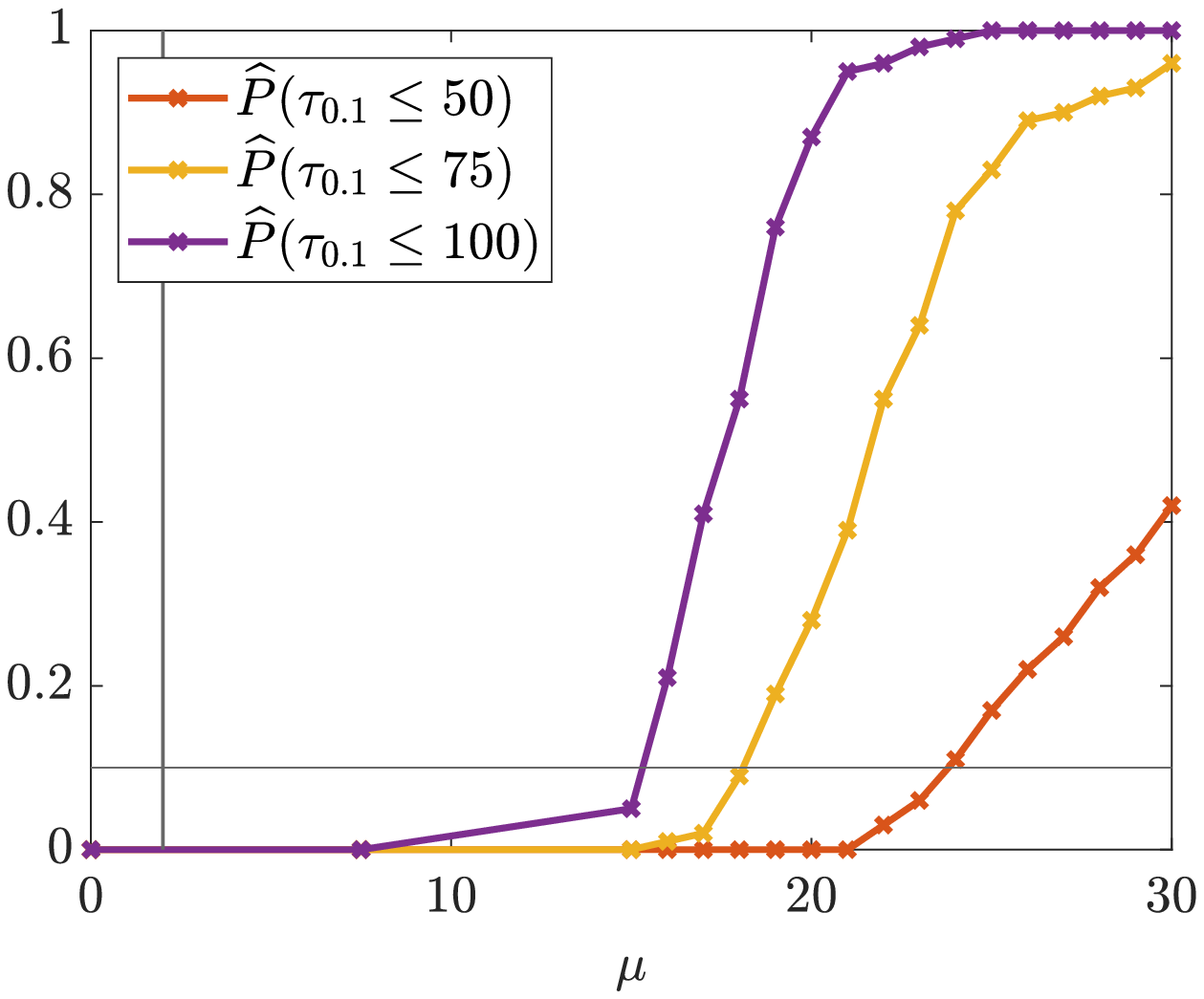}~
    \vspace{-\baselineskip}
    \caption{\textbf{Perforamance of the full oracle test.} Empirical rejection rates (over $100$ simulations) at $\alpha = 0.1$ versus the mean difference $\mu$ for the full oracle test implementations over four cases: $d=1, \batchinterval =  20$ (top left); $d = 2, \batchinterval =  20$ (top right); $d = 3, \batchinterval =  20$ (bottom left) ; $d = 4, \batchinterval =  25$ (bottom right). Observe the sharp improvement in power compared to Figure~\ref{fig:fully_nonparametric_tests}, especially in high dimensions, as indicated by a strong contraction in the scale of the X-axis. Observe also the improvement in power compared to Figure~\ref{fig:partial_oracle}, in that the curves reach high power at about half the $\mu$ that is needed for the partial oracle test.}
    \label{fig:full_oracle}
\end{figure}

\paragraph{Oracle Tests.} To probe the effect of the lossiness of the kernel density estimate on the power of the fully nonparametric test, we run `partial-oracle' and `full-oracle' oracle tests, which adjust $\mathscr{E}$ to exploit concrete information about the underlying laws $p(\cdot;\mu,d)$. In the partial-oracle, we adjust $\mathscr{E}$ to estimate a two-component Gaussian mixture model instead of a kernel density estimate, and in the full-oracle case, we directly set $\hat{q}_{t-1}(\cdot) = p(\cdot; \mu,d),$ i.e., we exactly evaluate the density. 

We expect that under data drawn from $p(\cdot; \mu,d),$ these tests are more powerful than the fully nonparametric tests discussed above, since the regret $\rho_t(\mathscr{E};p)$ would reduce in the case of the partial oracle due to a reduced complexity of the estimation class; and, of course, would reduce exactly to $0$ in the case of the full oracle. In either case, this effectively serves to increase $R_t$. These oracle tests thus let us probe the extent of the loss in power at a fixed $\mu$ (and thus a fixed distance from log-concavity) that arise purely due to the decay in rate of convergence of the log-concave MLE. In particular, the full oracle test captures exactly this effect, while the partial oracle test approaches this in a soft way. Figure~\ref{fig:partial_oracle} shows the performance of the partial oracle tests, and Figure~\ref{fig:full_oracle} shows the same for the full oracle test for $d \in\{ 1, 2, 3,4\}.$ %

Comparing Figures~\ref{fig:fully_nonparametric_tests} and~\ref{fig:partial_oracle}, we see that for using the partial oracle yields a marked increase in power, at least for $d > 1$. This is evident in $d=2$ by observing that the purple lines (overall rejection rate within $100$ times steps) rises higher and is nonzero at smaller values of $\mu$, as well as observing that the typical rejection time decreases substantially (for instance, rejection never happened below time step $60$ in the fully nonparametric case, but is quite prevalent at higher $\mu$s under the partial oracle). In $d = 3,4$ the effect is much starker - notice that the scale of the plot completely changes, from order of hundreds to tens in $d = 3$. This suggest that using the parametric mixture of Gaussians estimate offers strong improvements over the nonparametric KDE estimate due to the reduced variance scale of this estimator.

The above effect is seen even more starkly in the case of the fully oracle test, where each of the rejection rate curves is further improved (Figure \ref{fig:full_oracle}). For instance, our estimate of $\mu_*(d)$ (the smallest $\mu$ such that $\mathbb{P}_{p(\cdot;\mu,d)}(\tau_{0.1} \le 100) = 0.9)$) is about halved for the full oracle case when compared to the partial oracle (and improved manifold relative to the fully nonparametric test).

\paragraph{The Quality of $\mathscr{E}$ has a Strong Effect.} These observations from the oracle tests indicate that the quality of the estimate offered by $\mathscr{E}$ is very important in driving the overall power of the test. In these oracle examples, the quality improved by reducing the variance scale of the estimator, whilst keeping the bias at $0$ (since the law $p(\cdot;\mu,d)$ is representible by each of the estimator outputs).

Of course, in practice we cannot always hope to reduce the variance scale of our estimates whilst keeping the bias zero. Nevertheless, there is a tradeoff between the two implicit here. Indeed, as we discussed briefly in \S\ref{sec:power}, it is possible to use a biased $\mathscr{E}$ in the test, i.e. one that does not strictly estimate $p$, so long as the output of $\mathscr{E}$ does a better job of representing $p$ than the log-concave MLE. The strong dependence of the testing power on $\mathscr{E}$ indicates the critical need to investigate this design freedom, and to study how the trade-off between the variance, in terms of the convergence rates of $\hat{q}_{t},$ and the bias, i.e., the distance of $\lim \hat{q}_{t}$ from $p,$ should be balanced to optimise the testing power.

\section{Discussion}

Our work has shown that the sequential testing of log-concavity throws interesting challenges, in that the prevalent paradigm of test martingales cannot be fruitfully applied to this practically relevant setting. In the process of doing this, we developed a characterisation of the closed fork-convex hulls of independent sequential laws on a continuous space, thus contributing to the theory of this new tool that characterises the nonnegative supermartingale property. We then showed that the universal likelihood e-process instead does yield powerful tests for log-concavity. In particular, we demonstrated that these tests are consistent against large classes of nonparametric alternate laws, and further admit nontrivial rates, and made contributions to the off-the-model analysis of the convergence of log-concave MLEs, as well as the general theory of the power analysis of universal tests in order to do so. These properties are validated by running the test over a simple parametric family of laws, which further demonstrates the critical role of the sequential estimator $\mathscr{E}$ in the power of the test. Taking a broad view, the above can also be seen as a contribution to the emerging literature on e-processes, and in particular as additional evidence for the case that the study of sequential testing at large must exploit this powerful yet simple tool.

A number of directions, both theoretical and methodological remain open in this interesting subject, a few of which we discuss below.

Regarding fork-convexity, our characterisation in \S\ref{sec:triviality} and \S\ref{appx:triviality_proof} of the closed fork-convex hulls of i.i.d.~Gaussians can possibly be further enriched, and it would be very interesting to understand precisely which laws lie in this set. Additionally, notice that sequentially testing the Gaussiantiy of an i.i.d.~process itself is a basic problem that again cannot be tested using martingales (at least with respect to the natural filtration of the data). Construction and analysis of such sequential Gaussianity tests is a natural and interesting direction. Of course universal inference is again a natural approach for this class, but it may be possible to take advantage of translation and rotation invariance of the null hypothesis (all Gaussians) using methods developed in~\cite{perez2022statistics}.

Regarding the ULR e-process based test for log-concavity, \S\ref{sec:sim} shows that the power of the fully nonparametric test can be quite limited particularly as the data dimension increases. This observation was also made in the non-sequential setting by \cite{dunn2021-logConcaveTesting}, who proposed using random one-dimensional projections as an interesting method to ameliorate this. In this test, rather than computing the full $d$-dimensional kernel and log-concave estimates, one projects the data onto many one-dimensional subspaces, and averages the e-values (nonnegative test statistics with expectation at most one under the null) that result from a one-dimensional test carried on each of these projected datasets. This approach not only has computational benefits due to the speed of one-dimensional density estimation methods, but also shows statistical benefits in the scenario of \S\ref{sec:sim}, in that the decay of power is considerably limited with dimension. Such projected tests are of course possible in the sequential setting as well, and are a natural next step to investigate, both methodologically and in terms of their theoretical properties.

On a broader scale, both the theoretical bounds and the simulations illustrate the critical role that the quality of the estimator $\mathscr{E}$ plays, both specifically in the power of the test for log-concavity, but also more generally in the use of the universal likelihood ratio e-process. With this in mind, and recalling the implicit `bias-variance' tradeoff in $\mathscr{E}$ as discussed in \S\ref{sec:power} and \S\ref{sec:sim}, investigating the choice of $\mathscr{E}$ relative to the null class is an interesting question both in terms of practical methodological concerns, as well as theoretical concerns studying the power of e-process based tests.

\subsection*{Acknowledgments}
The authors thank Martin Larsson for insightful discussions on fork convexity, and Robin Dunn, for an implementation for a batched universal test for log-concavity that formed the backbone of the code underlying our simulations. A.~Rinaldo and A.G.~were supported in part by the NSF grant DMS-EPSRC 2015489.

\printbibliography

\clearpage

\appendix

\section{Proof of Triviality and Properties of Fork-Convex Hulls}\label{appx:triviality_proof}

This appendix is devoted to showing the structural lemmata regarding fork-convex hulls, and discussing technical aspects of our arguments.

\subsection[Details on the Local Sequential Closure]{ Details on the Local $L_1(\mathtt{\Gamma})$ Closure}

Let us begin by explicitly detailing the notion of convergence implicit in closed fork-convex combinations.

Recall that the $\cfhull(\mathfrak{P})$ is the closure of of $\fhull(\mathfrak{P})$ with respect to $L_1(\mathtt{\Gamma})$-convergence of likelihood ratio processes at every fixed time $t$. Let us unpack this statement in simple terms. Let $\mathtt{P}^n$ be some sequence in $\fhull(\mathfrak{P})$ of density ratio $Z^n_t := Z^{\mathtt{P}^n}_t$. We say that $\mathtt{P}^n \to \mathtt{P}$ if for every $t$, it holds that $Z_t^n \to Z_t$ in $L_1(\mathtt{\Gamma})$. Since $Z_t$ and $Z_t^n$ are $\mathscr{F}_t$ measurable objects, this convergence is simply in $L_1( \mathtt\Gamma|_t).$ Stating that the convergence needs to happen at every fixed time $t$ means that this convergence need not be uniform in $t$: it is fine for $Z_{100}^n$ to converge more slowly than $Z_1^n,$ for instance. This notion of convergence may be metrised by \[ \Delta(\mathtt{P}, \mathtt{Q}) := \sum_{t \in \mathbb{N}} 2^{-t} \| \denp{P}_t - \denp{Q}_t\|_{L_1(\mathtt{\Gamma})}. \] We note that $\Delta$ is bounded, since \[ \|Z_t^{\mathtt{P}} - Z_t^{\mathtt{Q}}\|_{L_1(\mathtt{\Gamma})} = \int \Big| \mathtt{P}|_t(\mathrm{d}x_1^t) - \mathtt{Q}|_t(\mathrm{d}x_1^t) \Big| \le \int \mathtt{P}|_t(\mathrm{d}x_1^t) + \int \mathtt{Q}|_t(\mathrm{d}x_1^t) = 2.\]

With this in hand, we first show the following auxiliary claim that is repeatedly used.
\begin{mylem}\label{lem:cfhull_okay_to_fit_prefixes}
    Let $\mathfrak{P}$ be a set of sequential laws, and let $\mathtt{R}$ be any sequential law. Suppose there exists a sequence of sequential laws $\{\mathtt{R}^T\}$ such that each $\mathtt{R}^T \in \fhull(\mathfrak{P}),$ and for all $t \le T, Z_t^{\mathtt{R}^T} = \denp{R}_t.$ Then $\mathtt{R} \in \cfhull(\mathfrak{P})$.
    \begin{proof}
        We claim that $\mathtt{R}^T \to \mathtt{R}$. Indeed, since $Z_t^{\mathtt{R}^T} = \denp{R}_t$ for all $t \le T,$ \[ \Delta(\mathtt{R}^T, \mathtt{R}) \le \sum_{t > T} 2^{-t} \cdot 2 = 2^{-(T-1)}. \]
    Thus, $\lim_{T \to \infty} \Delta(\mathtt R^T,\mathtt R) = 0$, meaning $\mathtt{R}^T \to \mathtt{R}$. Since the closed fork-convex hull of $\mathfrak P$ includes such limits by definition, the claim is proved. 
    \qedhere
    \end{proof}
\end{mylem}

The above lends significant convenience to our arguments, since it allows us to only construct processes matching some claimed member of the fork-convex hull up to finite times, which is typically easy to do in our arguments below using just finite fork-convex combinations. 

\subsection{Proofs about the Fork-Convex Hull of Independent Sequential Laws}\label{appx:triviality_proof_lemmas}

We may now proceed with the proofs of the Lemmata omitted from \S\ref{sec:triviality}.

\begin{proof}[Proof of Lemma~\ref{lem:cfhull_products}]
    As detailed in the main text, by taking repeated fork-convex combinations, it follows that $\mathtt{R}^T \in \fhull(\mathcal{P}^\infty),$ where \[ \mathtt{R}^1 := P_1^\infty, \mathtt{R}^T := \fccomb{\mathtt{R}^{T-1}}{P_T^\infty}{T-1}{0},\] where validity of the mixture weight $0$ exploits the mutual absolute continuity of laws in $\mathcal{P}$. We conclude by Lemma~\ref{lem:cfhull_okay_to_fit_prefixes}.
\end{proof}

\begin{proof}[Proof of Lemma~\ref{lem:cfhull_mixtures}]
    It suffices to show that for all finite $k, \mathcal{P}_k^\infty \subset \cfhull(\mathcal{P}^\infty),$ since $\mathcal{P}_* = \bigcup_{k} \mathcal{P}_k,$ and $\mathcal{P}_{k} \subset \mathcal{P}_{k+1}$ for all $k$. For $T \in \mathbb{N}$ and two laws $P,Q$ on $\mathbb{R}^d$, define the sequential law $\mathtt{R}_{P,Q,T}$ as the law of an independent sequence $\{X_t\}$ such that $X_t \sim P$ for $t \le T$ and $X_t \sim Q$ for $t > T$, i.e. $\mathtt{R}_{P,Q,T} = \fccomb{P^\infty}{Q^\infty}{T}{0}$. For $T \in \mathbb{N},$ define $\mathfrak{P}_{k,T}$ as the set of sequential laws of the form $\mathtt{R}_{P,Q,T}$ with $P \in \mathcal{P}_k$ and $Q \in \mathcal{P}$. 
    
    We first claim that $\mathfrak{P}_{k,T} \subset \fhull(P^\infty).$ 
    We show this inductively in $k$. Fix any $T$, and observe that trivially $\mathfrak{P}_{1,T}$ lies in this fork-convex hull. For $k \ge 2,$ we may represent each $P \in \mathcal{P}_k$ as $P = \alpha P^1 + (1-\alpha)P^2$ for some $\alpha \in [0,1], P^1 \in \mathcal{P}_{k-1}$ and $P^2 \in \mathcal{P}$. We need to show that for any such $P$, and any $Q \in \mathcal{P},$ $\mathtt{R}_{P,Q,T}$ lies in the fork-convex hull of $\mathcal{P}^\infty$. By the induction hypothesis, $\mathtt{R}_{P^1,Q,T} \in \fhull(\mathcal{P}^\infty),$ and $\mathtt{R}_{P^2, Q, T} \in \fhull(\mathcal{P}^\infty$).
        But then define the laws \[\mathtt{S}^0 := \mathtt{R}_{P^1, Q,T}, \widetilde{\mathtt{S}}^{\tau} := \fccomb{\mathtt{S}^{\tau - 1}}{\mathtt{R}_{P^2, Q, T}}{\tau-1}{\alpha}, \mathtt{S}^\tau := \fccomb{\widetilde{\mathtt{S}}^{\tau}}{\mathtt{R}_{P^1, Q,T}}{\tau}{0}.\]
        
    We note that every fork-convex combination above has valid weights since $\mathcal{P}$ is m.a.c., and so no density process is ever $0$. We claim that $\mathtt{S}^T = \mathtt{R}_{P,Q,T}$. 
    
    Indeed, let $p^1,p^2,q$ respectively denote the densities (with respect to the standard Gaussian) of $P^1, P^2$, and $Q$, and let $Z_t^1$ and $Z_t^2$ be the density processes of $\mathtt{R}_{P^1,Q,T}$ and $\mathtt{R}_{P^2,Q,T}$ respectively. These can be explicitly evaluated as \[ Z_t^i = \prod_{s \le \min(t,T)}p^i(X_s) \cdot \prod_{s = \min(t,T+1)}^t q(X_s),\] where $i \in \{1,2\},$ and we note that for $u < v, \prod_{s = v}^u \cdot = 1$. Observe that for each $i$, and any $t_1 < T,$ and $t > t_1,$ we have \[ \frac{Z_{t}^i}{Z_{t_1}^i} = \prod_{s = \min(t_1, T)+1}^{\min(t,T)} p^i(X_s) \cdot \prod_{s = \min(t,T+1)}^{t} q(X_s).\]
    
    We shall inductively claim that for each $\tau,$ the density process of $\mathtt{S}^\tau$ satisfies \[ \denp{S^\tau}_t = \prod_{s \le \min(t,\tau)} (\alpha p^1(X_s) + (1-\alpha)p^2(X_s)) \cdot \prod_{s = \min(t,\tau + 1)}^{\min(t,T)} p^1(X_s) \cdot \prod_{s = \min(t,T+1)}^t q(X_s).\] Indeed, the base claim is trivial since for $\tau = 0$ since $\mathtt{S}^0 = \mathtt{R}_{P^1,Q,T}$. Assuming the induction hypothesis for $\tau$, we observe that since $\tilde{S}^{\tau+1}$ is a fork-convex combination of $\mathtt{S}^\tau$ and $\mathtt{R}_{P^2,Q,T}$ at time $\tau,$ it shares the density process of $\mathtt{S}^\tau$ up to time $\tau$, while after that time the density is a mixture of the two density processes, giving \begin{align*} \denp{\tilde{S}^{\tau  + 1}}_t &= \prod_{s \le \min(t,\tau) } (\alpha p^1(X_s) + (1-\alpha) p^2(X_s)) \\ &\times  \left( \alpha \left\{ \prod_{s = \min(t,\tau+1)}^{\min(t,T)} p^1(X_s) \cdot \prod_{s = \min(t,T+1) }^t q(X_s)\right\} + (1-\alpha) \left\{\prod_{s = \min(t,\tau+1)}^{\min(t,T)} p^2(X_s) \cdot \prod_{s = \min(t,T+ 1)}^t q(X_s)\right\} \right),\end{align*} where we have used the behaviour of $Z^i_t/Z^i_{\tau}$ above for $t \ge \tau+1$. 

    Finally, $\mathtt{S}^{\tau+1}$ mixes the above with $\mathtt{R}_{P^1,Q,T}$ at time $\tau+1$ with a mixture weight of $0$. This means that the suffix law of $\mathtt{S}^{\tau + 1}$ beyond the time $\tau + 2$ is exactly equal to the law of $\mathtt{R}_{P^1,Q,T},$ while the prefix up to time $\tau+1$ is left alone. In other words, \begin{align*} \denp{S^\tau}_t &= \prod_{s \le \min(t,\tau)} (\alpha p^1(X_s) + (1-\alpha)p^2(X_s)) \cdot \prod_{s = \min(t,\tau+1)}^{\tau+1} (\alpha p^1(X_s) + (1-\alpha)p^2(X_s)) \\&\quad \times \prod_{s = \min(t,\tau + 2)}^{\min(t,T)} p^1(X_s) \cdot \prod_{s = \min(t,T+1)}^T q(X_s). \end{align*}
   
    The claim follows upon noticing that the first two products can be merged into $\prod_{s \le \min(t,\tau+1)} (\alpha p^1(X_s) + (1-\alpha)p^2(X_s).$

    With this in hand, the argument is straightforward. For any element  $\mathtt{P} \in \mathcal{P}_k^\infty,$ we note that there exists some member of $\mathfrak{P}_{k,T},$  say $\mathtt{P}^T$ such that the density process of $\mathtt{P}^T$ matches that of $\mathtt{P}$ up to time $T$. Applying Lemma~\ref{lem:cfhull_okay_to_fit_prefixes} immediately yields the claim.    
\end{proof}

\begin{proof}[Proof of Lemma~\ref{lem:cfhull_closures}]
    Let $\mathtt{P} = \bigotimes \{P_t\}$ for any arbitrary sequence of $P_t \in \overline{\mathcal{P}}$. We need to show that $\mathtt{P} \in \cfhull(\bigotimes \mathcal{P}).$ But, since $P_t \in \overline{\mathcal{P}}$ for each $t$, for each $t$ there further exist sequences $\{P_t^n\}_{n \in \mathbb{N}},$ with each $P_t^n \in \mathcal{P},$ such that $P_t^n \to P_t$ in $L_1(\Gamma)$. Let $\mathcal{Q} := \{P_t^n: t , n \in \mathbb{N}\}$. We note that $\bigotimes \mathcal{Q} \subset \bigotimes \mathcal{P} \implies \cfhull(\bigotimes \mathcal{Q}) \subset \cfhull(\bigotimes \mathcal{P}).$ Let $\mathfrak{Q}:= \cfhull(\bigotimes \mathcal{Q}).$ We shall argue that $\mathtt{P} \in \mathfrak{Q}.$  
    
    Let $\mathtt{P}^T$ be the sequential law with density process 
    \[Z_t^{\mathtt{P}^T} = \begin{cases} \prod_{s \le t} p_s(X_s) & \text{if } t \leq T\\
\prod_{s \le T} p_s(X_s) \cdot \prod_{T < s \le t} p_s^1(X_s) & \text{if } t > T \end{cases}.
    \]
    If we can show that for each $T$, $\mathtt{P}^T \in \mathfrak{Q},$ then the claim will follow, since $\mathtt{P}^T \to \bigotimes \{P_t\}$ as in the argument of Lemma~\ref{lem:cfhull_okay_to_fit_prefixes}, and since $\mathfrak{Q}$ is closed under the relevant notion of convergence.
        
        We shall show this inductively. Let $\mathtt{P}^{1,n}$ be a sequential law with density $Z^{1,n}_t := p_1^n(X_1) \cdot \prod_{s > \min(1,t)} p_s^1(X_s).$ Notice that $\mathtt{P}^{1,n} \in \bigotimes \mathcal{Q} \subset \mathfrak{Q}$ for every $n$. Further, \[ \Delta(\mathtt{P}^{1,n}, \mathtt{P}^1) \le \|Z_1^{1,n} - Z_1^{\mathtt{P}^1}\|_{L_1(\mathtt{\Gamma})} \to 0.\] Thus $\mathtt{P}^1 \in \mathfrak{Q}.$
        
        Now suppose that $\mathtt{P}^{T-1} \in \mathfrak{Q}$ for some $T \ge 2.$ For $T, n\in \mathbb{N},$ define $\mathtt{Q}^{n}$ as the sequential law of density ratio \[ Z_t^{\mathtt{Q}^{n}} := \begin{cases} \prod_{s < t} p_s^1(X_s)  & t < T \\ Z_{T-1}^{\mathtt{Q}^{n}} \cdot p_T^n(X_T) & t = T \\ Z_{T}^{\mathtt{Q}^{n}} \cdot \prod_{s = T+1}^t p_s^1(X_s) & t > T\end{cases}.\] It trivially follows that $\mathtt{Q}^n \in \bigotimes \mathcal{Q} \subset \mathfrak{Q}$ for all $n$. Now, define \[ \mathtt{P}^{T,n} = \fccomb{\mathtt{P}^{T-1}}{\mathtt{Q}^{n}}{T-1}{0},\] which is valid since each $P_t^n$ and $P_t$ has are mutually absolutely continuous. But $Z_t^{\mathtt{P}^{T,n}} = Z_t^{\mathtt{P}^{T-1}} = Z_t^{\mathtt{P}^T}$ for $t \le T-1,$ and for $t \ge T,$ \[ {Z}_t^{\mathtt{P}^{T,n}} - Z_t^{\mathtt{P}^T} = Z_{T-1}^{\mathtt{P}^T} \cdot ( p_T^n(X_T) - p_T(X_t)) \cdot \prod_{s = T+1}^t p_s(X_s).\]
        
        It follows that \[ \|Z_t^{\mathtt{P}^{T,n}} - Z_t^{\mathtt{P}^T}\| = \begin{cases} 0 & t < T \\ \|P_T^n - P_T\|_{L_1(\Gamma)} & t \ge T\end{cases},\] and therefore, $\Delta(\mathtt{P}^{T,n}, \mathtt{P}^T) \le \|P_T^n - P_T\|_{L_1(\Gamma)} \to 0.$ By closeness of $\mathfrak{Q},$ we conclude that $\mathtt{P}^T \in \mathfrak{Q}$.
\end{proof}

\subsection{Proof of Lemma~\ref{lem:find_nice_rectangle}}\label{appx:nice_rectangle}

\begin{proof}[Proof of Lemma~\ref{lem:find_nice_rectangle}]
    Fix an $m\in \mathbb{N}$. Since $E$ has positive mass and is measurable, there exists an open set $O \in (\mathbb{R}^d)^t$ such that $O \supset E$ and $\leb_{dt}(O) \le (1 + 1/m) \leb_{dt}(E).$ Observe here that `most' of the mass of $O$ lies within $E$.

    Since $O$ is open, there exists a sequence of disjoint open rectangles $R_i$ in $(\mathbb{R}^d)^t$ such that $\bigsqcup R_i \subset O \subset \bigcup \overline{R}_i$, and \[ \leb_{dt}\left( \bigsqcup R_i\right) = \sum \leb_{dt}(R_i) = \leb_{dt}(O). \] Further, since most of the mass of $O$ lies in $E,$ we conclude that there exists at least one $i$ such that \[ \leb_{dt}(R_i) > 0 \quad \textit{and} \quad \leb_{dt}(E \cap R_i) \ge \frac{m}{m+1} \leb_{dt}(R_i).\] Indeed, otherwise we would have \[ \leb_{dt}(E) = \leb_{dt}(E \cap O) = \sum \leb_{dt}(E \cap R_i) < \frac{m}{m+1} \sum \leb_{dt}(R_i) \leq \frac{m}{m+1} \cdot \frac{m+1}{m} \leb_{dt}(E), \] which is impossible. \qedhere %

\end{proof}

\subsection{Technical Aspects of Fork-Convex Hulls and Our Triviality Argument}

We comment on some technical aspects of the argument underlying the non-existence of nontrivial NSMs. Specifically, we discuss the necessity of our definition of nontriviality, and the m.a.c.~condition repeatedly used in the argument, how the argument can be extended to consider log-concave laws over bounded sets, and finally issues that arise when one tries to relax the definition of fork-convex combinations to handle support mismatch. %

\paragraph{Going beyond almost sure triviality.} The main text defines trivial NSMs (and NMs) as those that are $\mathtt{\Gamma}$-almost surely non-increasing (respectively, constant). Could one instead show that there are no nontrivial $\lciidseq$-NSMs in the stronger sense that such processes must be non-increasing (as opposed to only \emph{almost surely} non-increasing)?  This turns out to be impossible, as witnessed by the following process
 \[ M_t := \frac{1}{1 - \indi\{\exists (t_1, t_2, t_3, t_4) \in [1:t]^4 : X_{t_1} = X_{t_2}, X_{t_3} = X_{t_4}, X_{t_1} \neq X_{t_3}\}}.\] Since log-concave measures can have at most one atom (due to unimodality), it follows that $\{M_t\}$ is an $\lciidseq$-martingale (indeed, it is almost surely just a constant $1,$ as stated by the theorem). However, $M_t$ does diverge to $\infty,$ and this occurs almost surely against any i.i.d.~sequential law which has at least two atoms, for instance, a coin flip process. This means that while it may not be possible to reject processes with a Lebesgue density using test martingales, it \emph{is} possible to reject atomic processes. Structurally, this example has to do with the fact that one cannot approach point masses in an $L_1$ sense using measures with density. Therefore, although $\lciidseq$-NSMs must also be NSMs for independent processes with densities, this does not extend to sequences drawn from distributions with atoms. In another sense, this issue is the same as the problem discussed below regarding loss of the NSM property under extensions of fork-convex combinations of laws with support mismatch, in that two laws with distinct single atoms each have parts that are mutually singular.

\paragraph{The role of the mutual absolute continuity condition on $\mathcal{P}$.} The definition of fork-convex combinations of two laws $\mathtt{P}$ and $\mathtt{Q}$ at time $s$ involves the ratio of density processes $\denp{P}_s/\denp{Q}_s$. This ratio must indeed appear, as can be seen from the algorithmic viewpoint of \S\ref{sec:triviality} to account for the fact that if $\mathtt{R}$ is the fork-convex combination, then the prefix law $\mathtt{R}|_s = \mathtt{P}|_s.$ However, if $\denp{Q}_s = 0,$ i.e. if for $\{X_t\} \sim \mathtt{R},$ the prefix $X_1^s$ lies in a set that is almost surely impossible under $\mathtt{Q},$ then the above ratio is meaningless. This observation underlies the condition that if $\denp{Q}_s = 0,$ then the mixture weight $h$ must be exactly $1$. 

Our argument ultimately asserts that any law of the form $\bigotimes \{P_t\}$ lies in $\cfhull(\mathcal{P}^\infty)$. However, our constructions to demonstrate this fact rely on setting $h = 0$ in order to generate switches between different laws in $\mathcal{P}$. Our assumption of mutual absolute continuity is to enable precisely this flexibility without running into the issue discussed in the previous paragraph. %

\paragraph{The role of Gaussians in our argument.} Since we used the density of the Gaussians in order to show that $\lciidseq$-NSMs must also be $\bigotimes \mathcal{D}$-NSMs, it behooves us to ask how essential $\lc \supset \mathcal{G}$ is to the main point of the result.\footnote{Notwithstanding that the result is interesting in its own right for Gaussians, which tells us that there is a simple, and very natural, parametric family that cannot be tested via nonnegative supermartingales.} In the argument, Gaussians play two roles: firstly, since all Gaussians are supported on the entirety of the domain, this class is m.a.c., and we can flexibly take fork-convex combinations. Secondly, the triviality of Gaussian NSMs follows since mixtures of Gaussians are $L_1$-dense in the set of densities on the reals. Any subset of $\lc$ that satisfies these two properties will suffice for our purposes. 

\paragraph{Extending the argument to log-concave laws on subsets of $\mathbb{R}^d$.} We finally observe that our argument extends in a straightforward manner to log-concave laws on restricted subsets of the reals: for a bounded convex set $K$, define $\lc_K$ to be log-concave densities supported on $K$. Then all $\lc_K^\infty$-NSMs are trivial, in the sense that they are almost surely nonincreasing with respect to the reference measure $(\mathrm{Unif}(K))^\infty.$ This follows because truncated Gaussians are again dense and supported on the entirety of the domain $K$. 

To see this, first observe that if $\gamma := \sum \alpha_i \phi_i$ is a mixture of Gaussians, then for any $K$ of nonzero Lebesgue mass, the truncation $\gamma|_K$ is also a mixture of truncated Gaussians. Indeed, define $\theta_i = \int_K \phi_i.$ Then \[ \gamma|_K(x) = \sum \frac{\alpha_i}{\sum \alpha_i \theta_i} \phi_i(x) \cdot \mathbf{1}\{x \in K\}  = \sum \frac{\alpha_i \theta_i}{\sum \alpha_i\theta_i} \phi_i|_K(x).\]
Now, let $p$ be any density supported on $K$, and let $\gamma_n \to p$ be a sequence of mixtures of Gaussians converging so that $d_n := \int |p -\gamma_n| \to 0.$ Then, defining $\pi_n = \int_{K^c} \gamma_n,$ we have \[  \int |p - \gamma_n|_K| = \int_K \frac{| p(1-\pi_n) - \gamma_n|}{1-\pi_n} \le \int_K \frac{|p-\gamma_n|}{1-\pi_n} + \int_K \frac{\pi_n p}{1-\pi_n} \le \frac{\pi_n + \int |p - \gamma_n|}{1-\pi_n}.\] Further, since $p$ is supported on $K$, $\pi_n = \int_{K^c} \gamma_n \le \int_{K^c} \gamma_n + \int_K |p - \gamma_n| = d_n$. Therefore, \[ \mathrm{TV}(p,\gamma_n|_K) \le \frac{2d_n}{1-d_n} \to 0.\] 

But this means that we can run the entire argument of \S\ref{sec:triviality} but with Gaussians truncated over $K,$ and draw the same conclusion.

\paragraph{Can we extend nontrivial fork-convex combinations to all laws?} As we discussed above, due to the ``$\denp{Q}_T = 0 \implies h = 1$'' condition in the definition of fork-convex combinations, it is not possible to take arbitrary fork-convex combinations between sequential laws. In the extreme case of $\mathtt{P} = P^\infty$ and $\mathtt{Q} = Q^\infty$ for $P,Q$ that have disjoint support, the only possible fork-convex combinations are mixtures of the form $\alpha P^\infty + (1-\alpha)Q^\infty$. 
While this technicality did not pose a serious issue for the current paper, this situation is quite unsatisfying in general. After all, the algorithmic view of fork-convex combinations is very natural, and extends to such disjoint support situations easily. 

One can formalise this algorithmic picture by exploiting conditional densities. For a sequence of (appropriately measurable) maps $k^{\mathtt{P}}_t : (\mathbb{R}^d)^{t-1} \times \mathbb{R}^d \to \mathbb{R}_{\ge 0},$ denoted $k_t^{\mathtt{P}}(x_t|x_1^{t-1}),$ we say that $\{k_t^{\mathtt{P}}\}$ is the conditional density process of $\mathtt{P}$ if for each $x_1^{t-1}, k_t(\cdot|x_1^{t-1})$ is a density with respect to $\Gamma$, and for any $t, A \in \mathscr{F}_t,$ 
\[ 
\mathtt{P}(X_1^t \in A) = \int_A \prod_{s \le t} k_s(x_s|x_1^{s-1}) \mathtt{\Gamma}(\mathrm{d}x_1^t).
\] 
More generally, we can define a similar notion via Markov kernels. We observe that, by definition, it holds that if $\mathtt{P}$ has a conditional density process, then for any $t$ and $\mathtt{\Gamma}$-almost every $x_1^t$ that \[\denp{P}_t(x_1^t) = \prod_{s \le t} k_s(x_s|x_1^{s-1}).\] 

Using the above characterisation, we can give the following natural extended definition of fork-convex combinations: for two sequential laws $\mathtt{P}, \mathtt{Q}$ with conditional density processes $\{k_t^{\mathtt{P}}\}, \{k_t^{\mathtt{Q}}\}$ respectively, a law $\mathtt{R}$ is a fork-convex combination of $\mathtt{P}$ and $\mathtt{Q}$ at time $T$ with $\mathscr{F}_T$-measurable weight $h$ if \begin{align}\label{eq:extended_fc_comb_def} \denp{R}_t = \begin{cases} \prod_{s\le t} k_s^{\mathtt{P}}(x_s|x_1^{s-1}) & t \le T \\ \prod_{s \le T} k_s^{\mathtt{P}}(x_s|x_1^{s-1}) \cdot \left( h\prod_{s = T+1}^t k_s^{\mathtt{P}}(x_s|x_1^{s-1}) + (1-h) \prod_{s = T+1}^t k_s^{\mathtt{Q}}(x_s|x_1^{s-1}) \right) & t >T \end{cases},\end{align} the difference being that we now \emph{do not impose the restriction} that $h = 1$ if $\denp{Q}_T = 0.$ Simplistically, this is possible since we are never dividing by the potentially null $\denp{Q}_T$, and more formally, this is considering the formal ratio $\denp{Q}_t/\denp{Q}_T,$ which is interpreted in the natural way as $\prod k_s^{\mathtt{Q}}(x_s|x_1^{s-1}).$ The above extended definition genealizes our previous definition of fork-convex combinations, and we can extend the same to the fork-convex hull and its closure.

While the density process above is a perfectly sound mathematical object, such an extension is not fruitful because of a failure to preserve the NSM property under these extended fork-convex combinations in general.

To illustrate why the above extended definition fails to maintain the NSM property (unlike the restricted one used in the paper), consider the following example.
\begin{example}
$\mathtt{P} = (\mathrm{Unif}(0,1))^\infty$ and $\mathtt{Q} = ( \mathrm{Unif}(1,2))^\infty,$ and the process \[ M_t := \begin{cases}
    2 & \exists s_1 ,s_2 \le  t: X_{s_1} \in (0,1), X_{s_2} \in (1,2) \\
    1 & \textrm{ otherwise }
\end{cases}. \] This process is an NSM (indeed, a martingale) under both $\mathtt{P}, \mathtt{Q}.$ However, under any nontrivial fork-convex combination of these two laws, this process must start at $1$, and with positive probability grow to $2$ but never fall, and thus cannot be a supermartingale. 
\end{example}

Under the hood, the issue in the example above arises due to the fact that under the extended definition, for $t \ge T+1,$ $\{\denp{R}_t > 0\} = \{\denp{P}_t > 0 \} \cup \{ \denp{P}_T > 0 , \prod_{T+1}^t k_s(X_s|X_1^{s-1}) > 0\},$ but the NSM property of $\{M_t\}$ under $\mathtt{P}$ or $\mathtt{Q}$ only controls the conditional expectations of $M_t\denp{P}_t\indi\{\denp{P}_t > 0\}$ and $M_t\denp{Q}_t\indi\{\denp{Q}_t > 0\}$ under $\mathtt{\Gamma},$ which leaves the conditional behaviour of $M_t \denp{R}_t$ uncontrolled when $\mathtt{R}$ places mass on events that are null under one of these laws.

It should be noted that in the above example there \emph{is} a version of the process $\{M_t\}$, i.e., a process $\{\widetilde{M}_t\}$ such that $\mathtt{P}(\forall t, M_t = \widetilde{M}_t) = \mathtt{Q}(\forall t, M_t = \widetilde{M}_t) = 1,$ but such that $\{\widetilde{M}_t\}$ is a martingale even under extended fork-convex combinations. Concretely this process is just $\widetilde{M}_t = 1.$ One may thus wonder if this phenomenon holds true in greater in generality: is it the case that if $\{M_t\}$ is an NSM under $\mathtt{P}$ and $\mathtt{Q},$ then there is a version $\{\widetilde{M}_t\}$ of it (under $\mathtt{P}$ and $\mathtt{Q}$) such that $\{\widetilde{M}_t\}$ is an NSM against any extended fork-convex combination of $\mathtt{P}$ and $\mathtt{Q},$ without the restriction ``$\denp{Q}_T = 0 \implies h = 1$''? This turns out also to be impossible in general, as demonstrated by the following example.

\begin{example}
Let $\mathtt{P} = \mathrm{Unif}(0,1)^\infty$ and $\mathtt{Q} = \mathrm{Unif}(0,1/2)^\infty$. Define $\rho_t = \indi\{X_t \in (0,1/2)\}$ for $t \ge 1,$ and $\rho_0 = 0$. Let $\{N_t\}$ be an adapted process such that \[ N_t = \begin{cases} 1 & \rho_{t-1} = 1\\ 3/2 & \rho_{t-1} = 0, \rho_t = 1\\ 1/2 & \rho_{t-1} = 0, \rho_t = 0\end{cases}.\] Finally define $M_t = \prod_{s \le t} N_t$. It is easy to check that $M_t$ is an NM under both $\mathtt{P}$ and $\mathtt{Q}$. 

Now suppose $\mathtt{R}$ is an extended fork-convex combination of $\mathtt{P},\mathtt{Q}$ at time $T \ge 1$ with mixture weight $h < 1.$ This means that with probability $1-h,$ it holds that $X_t \in (0,1/2)$ with certainty for all $t \ge T+1$. As as result, we can explicitly compute that \begin{align*} \mathbb{E}[N_{T+1}|\mathscr{F}_T] = \rho_T + (1-\rho_T) \left( (1-h) \cdot \nicefrac{3}{2} + h ( \nicefrac{1}{2} \cdot \nicefrac{3}{2} + \nicefrac{1}2 \cdot \nicefrac{3}{2} ) \right) = \begin{cases} 1 & \rho_T = 1 \\ 1 + (1-h)/2 & \rho_T = 0\end{cases} ,\end{align*} and so as long as $h < 1,$ $\mathbb{E}[N_{T+1}|\mathscr{F}_T] > 1$ if $\rho_T = 0,$ and therefore $\{M_t\}$ violates the NSM property under $\mathtt{R}$ at the time $T+1$. Note here that it is hard to construct any nontrivially different version of the above process since the law of $\mathtt{P}$ dominates that of $\mathtt{Q}$. 
\end{example}

In light of the above discussion, generalised definitions of fork-convex combinations are at loggerheads with maintaining the NSM property these combinations. Of course, since our purpose in using fork-convexity is to assert the triviality of NSMs over large classes of sequential laws, this latter property is essential to maintain for such statistical applications. At the same time, while the restricted original definition does maintain the NSM property, the included restriction is unsatisfying, and in conflict with the algorithmic intuition underlying the idea of these combinations. Finding an appropriate generalised definition of fork-convex combinations that abstains from imposing these support conditions, but nevertheless retains NSM closure under the NSM property is an interesting, and challenging, question left for future work.

\section{Proofs of Consistency and Power Analysis}

Recall the notation $\sigma_t(P) := \sum_{s \le t} \log p(X_s) - \log \hat{p}_t(X_s)$. The main arguments of this section control the behavious of $\sigma_t(P),$ in particular arguing that if the Hellinger distance of $P$ from log-concavity is large, then $\sigma_t(P)$ must eventually grow linearly. We show this in asymptotic and nonasymptotic regimes in \S\ref{appx:consistency_proof} and \S\ref{appx:power_analysis_proofs} respectively.

Corollary~\ref{prop:consistency_box} and Corollary~\ref{thm:rate_for_box} each relies on further control on the behaviour of $\rho_t(\mathscr{E};p) = \sum_{s \le t} \log p(X_s) - \log\hat{q}_{s-1}(X_s)$ when $p$ is a bounded Lipschitz law on the unit box. This argument is left to \S\ref{appx:bounded_unit_box_regret}.

\subsection{Proof of Consistency}\label{appx:consistency_proof}

Our arguments rely on the following bracketing tail estimate, developed by Wong and Shen to analyse the behaviour of sieve-based maximum likelihood estimates \cite[Thm.~1]{wong1995probability}. The estimate involves the bracketing entropy of a class of laws $\mathcal{Q}$ under the Hellinger metric. We refer the reader to the text of Van der Vaart and Wellner \cite{vaart1996weak} for a thorough introduction, and give a brief account.

A bracket $[u,v]$ is defined by two functions $u(x) , v(x)$ such that $u(x) \le v(x)$ for all $x$, and consists of the set of all functions $f$ such that $u(x) \le f(x) \le v(x)$ for all $x$. Since we shall only be interested in functions that are densities, we may restrict attention to nonnegative functions. The Hellinger size of such a bracket $[u,v]$ is defined as $|[u,v]| = \|\sqrt{u} - \sqrt{v}\|_2/2.$ 
We say that a class of distributions $\mathcal{Q}$ is bracketed by $\{[u_i, v_i]\}_{i = 1}^N$ if for all $Q \in \mathcal{Q}$, there exists an $i$ such that $q \in [u_i, v_i],$ where recall that for a distribution $Q$, we denote its density by $q$. Note that this bracketing is typically ``improper'', i.e., $u_i, v_i$ will generally not lie in $\mathcal{Q}$ (because $q$ integrates to one, and so its lower bracket $u$ will integrate to less than one, and its upper bracket $v$ will integrate to more than one). The Hellinger bracketing entropy of $\mathcal{Q}$ at scale $\zeta$ is the logarithm of the most parsimonious bracketing of $\mathcal{Q}$ by brackets of size at most $\zeta,$ i.e. 
\[
\mathcal{H}_{[]}(\mathcal{Q}, \zeta) := \inf \{ \log N : \text{$\mathcal Q$ has an $N$-sized Hellinger bracketing at scale $\zeta$} \}
\]
Note, of course, that bracketing entropies are nonincreasing in $\zeta$. 

\begin{mylem}\label{lem:wong_shen}
    \emph{(Simplification of \cite[Thm.~1]{wong1995probability})} For a class of distributions $\mathcal{Q}$ and a natural number $t$, define $\varepsilon_t$ as the smaller number $\varepsilon$ such that \[ \int_{\epsilon^2/2^8}^{\sqrt{2}\epsilon} \sqrt{\mathcal{H}_{[]}(\mathcal{Q}, \zeta/10)}\mathrm{d}\zeta \le 2^{-11} \sqrt{t} \epsilon^2.\]
        For every $t$ and  $\epsilon \ge \epsilon_t,$ it holds that for any law $P$ such that $d_H(P,\mathcal{Q}) \ge \epsilon,$ we have \[ P^{\otimes t} \left( \inf_{q \in \mathcal{Q}} \sum_{s \le t} \log p(X_s) - \log q(X_s) \le t\epsilon^2/24 \right) \le 4 \exp{ - C t\epsilon^2},\] where $C > 2^{-14}$ is a constant.    
\end{mylem}
Informally, if $P$ is far enough from $\mathcal Q$ in the Hellinger metric (where far enough is determined by the bracketing entropy of $\mathcal Q$), then it is exponentially unlikely (in the sample size) for the maximum log-likelihood under $\mathcal Q$ to be linearly close to the log-likelihood under $P$. 
Exploiting this observation in our context requires us to argue that eventually, the log-concave MLE $\hat{p}_t$ must lie in a set with small entropy. To this end, we appeal to the following result due to Dunn et al., which extends the convergence analysis of Cule \& Samworth \cite{cule2010theoretical}.
\begin{mylem}\label{lem:small_brack}
    \emph{\cite[Lem.~1]{dunn2021-logConcaveTesting}} Consider any distribution $P \in \mathcal{D}_1$, not necessarily log-concave. For any $\eta > 0,$ there exists a bracket $[u_\eta, v_\eta]$ of size at most $\eta$ that contains the log-concave projection $\lcproj_P$, and eventually also contains the log-concave MLE $\hat{p}_t$: $P^\infty(\exists t_0: \forall t \ge t_0, \hat{p}_t \in [u_\eta, v_\eta]) = 1.$    
\end{mylem}

In words, the lemma states that for large enough $t$, the log-concave MLE $\hat{p}_t$ is certain to lie in a very small bracket around the log-concave projection $\lcproj_P$ of the true distribution $P$. With this in hand, we are in a position to show Lemma~\ref{lem:growth_of_sigma_t}, the main statement underlying the proof of Theorem~\ref{thm:consistency}.

\begin{proof}[Proof of Lemma~\ref{lem:growth_of_sigma_t}]
    Let $\epsilon := d_H(P, \lcproj_P) > 0.$ Define $\eta_\epsilon = \epsilon^2/2^{11}.$ Using Lemma~\ref{lem:small_brack}, we know that there exists a bracket $[u^*, v^*]$ such that $|[u^*, v^*]| \le \eta_\epsilon$ and, almost surely, $\hat{p}_t \in [u^*, v^*]$ for all large enough $t$. But observe that $\mathcal{H}_{[]}([u^*, v^*], \epsilon^2/2^{11}) = 0,$ since the size of $[u^*, v^*]$ is already $\eta_\epsilon$. Further, since $\lcproj_P \in [u^*, v^*], $ by the triangle inequality, \[ d_H(P, [u^*, v^*]) = \inf_{Q \in [u_\eta, v_\eta]} d_H(P,Q) \ge d_H(P, \lcproj_P) - |[u^*, v^*]| \ge \epsilon (1 - 2^{-11}) \geq \epsilon \cdot \sqrt{24/25}.\]

    Let us define 
    \[ \widetilde\sigma_t(P) := \inf_{q \in [u^*, v^*]} \sum_{s \le t} \log p(X_s) - \log q(X_s). \] 
    By exploiting the above observations, Lemma~\ref{lem:wong_shen} yields that for every $t$,\[ P^\infty \left( \widetilde\sigma_t(P) \le t \varepsilon^2/25 \right) \le 4 \exp{-Ct\epsilon^2}.\]
    
    Note further that if $\hat{p}_t \in [u^*, v^*],$ then since $\hat{p}_t$ is a maximum likelhood estimate, it must hold that $\sigma_t(P) = \widetilde\sigma_t(P)$. Let $\mathsf{E}_{s}:= \{ \forall t \ge s, \hat{p}_t \in [u^*, v^*]\}$ be the event that $\hat{p}_t$ lies in the small bracket after time $s$, and $\mathsf{A}_t := \{ \sigma_t(P)/t\varepsilon^2 \ge 1/25\}$ be the event that $\sigma_t(P)$ is larger than $t\varepsilon^2/25$.
    
    By Lemma~\ref{lem:wong_shen}, for every fixed time $s$ and $t \ge s,$ \[ P^\infty(\mathsf{A}_t^c \cap \mathsf{E}_s) \le 4\exp{-Ct\varepsilon^2},\] and since this upper bound is summable, by the Borel-Cantelli Lemma \[
0 = P^\infty \left( \limsup_t (\mathsf{A}_t^c \cap \mathsf{E}_s ) \right) = P^\infty \left( (\limsup_t \mathsf{A}_t^c) \cap \mathsf{E}_s  \right),
    \]
    and so for any time $s$, \[ P^\infty(\limsup_t \mathsf{A}_t^c) \le P^\infty( \limsup_t \mathsf{A}_t^c \cap \mathsf{E}_s) + P^\infty(\mathsf{E}_s^c) = P^\infty(\mathsf{E}_s^c).\] By Lemma~\ref{lem:small_brack}, $\hat{p}_t$ must eventually almost surely fall in $[u^*,v^*],$ $\lim_{s \to \infty} P^\infty(\mathsf{E}_s^c) \to 0.$ Further notice that  \[ \limsup_t \mathsf{A}_t^c = \{ \sigma_t(P)/t\varepsilon^2 < 1/25 \textrm{ infinitely often}\} = \{\liminf \sigma_t(P)/t\varepsilon^2 < 1/25\}.\] Putting the observations together, we conclude upon sending $s \to \infty$ that \[ P^\infty( \liminf \sigma_t(P)/t\varepsilon^2 < 1/25) \le \lim_{s \to \infty} P^\infty(\mathsf{E}_s^c) = 0. \qedhere \] 
\end{proof}

\subsection{Proofs Underlying the Power Analysis}\label{appx:power_analysis_proofs}

We shall begin by stating the key lemmata underlying our argument, which exploit our bracketing entropy control from Lemma~\ref{lem:lc_entropy_bound} along with results in the literature that bound the maximum value attained by a log-concave density in order to make the same effective. We then prove the main result, and conclude by proving Lemma~\ref{lem:lc_entropy_bound}.

\subsubsection{Controlling the Maximum Value Attained by the Log-Concave MLE}

The rate analysis quantitatively exploits Lemma~\ref{lem:wong_shen}. To do so, we first need bracketing entropy bounds for log-concave laws, which is precisely the subject of Lemma~\ref{lem:lc_entropy_bound}. We recall that this controls the bracketing entropy of the class $\lc_{d,\boundparam}$ of log-concave laws with densities supported on $[-1,1]^d$ that are bounded from above by $\boundparam$, showing that \[ \mathcal{H}_{[]}(\lc_{d,\boundparam}, \zeta) = \widetilde{O}( (\boundparam/\zeta)^{\max(d/2, (d-1)})).\]

The role of $\boundparam$ in the above is quantitatively unimportant as long as this constant does not scale with relevant parameters. This fact is assured for log-concave laws with near identity covariance. Intuitively, since the covariance is lower bounded in all directions, the laws cannot concentrate too much, and thus the value of the density at the mode cannot be too large. This observation is encapsulated in the following result, which follows trivially from the work of Kim \& Samworth.

\begin{mylem}\label{lem:cov_large_means_law_small}
    \emph{\cite[Cor.~6]{kim2016global}} Let $\lc_d^{\gamma}$ denote the set of log-concave laws distributed on $[-1,1]^d$ with covariances lower bounded in the positive semidefinite order by $\gamma I.$ Then there exists a dimension dependent constant $C_d$ such that for any $f \in \lc_d^\gamma,$  \[ \max_{x \in [-1,1]^d} f(x) \le \gamma^{-d/2}C_d.\]
\end{mylem}

Of course, our bounds in Theorem~\ref{thm:rate} depend on $\Delta_P,$ which roughly speaking only controls that the covariance of the underlying law $P$. The relevance of this quantity arises from the following observation, due to Barber and Samworth.

\begin{mylem}\label{lem:mle_is_controlled_guts}
    \emph{\cite[Cor.~8]{barber2021local}} Let $P\in \mathcal{D}_1$ be a law supported on $[-1,1]^d$ such that \[ \Delta_P := \min_{v : \|v\| = 1} \mathbb{E}_p[ |\langle v, X - \mathbb{E}_p[X]\rangle|] > 0.\] Then there exists a dimension dependent constant $c_d$ such that \( \mathrm{Cov}(\lcproj_P) \succeq c_d \Delta_P^2 I.\) Further, there exists a dimension-independent constant $C$ such that for any $t \ge 2Cd^3/\Delta_P^2,$ it holds with probability at least $1 - 2\exp{-Ct \Delta_P^2/d^2}$ that \( \mathrm{Cov}(\hat{p}_t) \succeq \frac{c_d\Delta_P^2}{4}I\) for the log-concave MLE $\hat p_t$. 
\end{mylem}
\begin{proof}[Proof of Lemma~\ref{lem:mle_is_controlled_guts}]
    The first observation is a direct restatement of Lemma 7 of Barber and Samworth. The second statement follows from the fact that over $v: \|v\| = 1,$ $v \mapsto \langle v, X - \mathbb{E}_P[X]\rangle$ is bounded by $2\sqrt{d},$ and is clearly continuous in $v$. Thus exploiting standard subGaussian concentration results over the unit ball, it follows that with probability at least $1 - 2\exp{- C t \Delta_P^2/d^2},$ it holds that for the empirical law $p_t = \frac{1}{t} \sum_{s \le t} \delta_{X_s},$ \[ \min_{v : \|v\| = 1} \mathbb{E}_{p_t}[|\langle v , X - \mathbb{E}_{p_t}[X]\rangle|] \ge \Delta_P/2.\] But notice that $\hat{p}_t = \lcproj_{p_t},$ from which the claim follows by the first part.  \qedhere
\end{proof}

Merging Lemmas~\ref{lem:cov_large_means_law_small} and~\ref{lem:mle_is_controlled_guts} immediately yields the following observation, which serves as a concrete bound for the scale of $\boundparam$ we need to employ in Lemma~\ref{lem:lc_entropy_bound}.

\begin{mylem}\label{lem:mle_is_controlled}
    There exists a constant $C_d$ depending only on $d$ such that for any $t \ge 2C_d d^3/\Delta_P^2,$ it holds with probability at least $1 - 2 \exp{-C_dt\Delta_P^2/d^2}$ that \[ \max_{x \in [-1,1]^d} \hat{p}_t(x) \le C_d  \Delta_P^{-d}.\]
\begin{proof}
    Employing Lemma~\ref{lem:cov_large_means_law_small}, we observe that  $\{ \max \hat{p}_t \le (c_d \Delta_P^2/4)^{-d/2}\} \subset \{ \mathrm{Cov}(\hat{p}_t) \preceq c_d\Delta_P^2/4 I\},$ and the latter has probability at least $1 - 2\exp{-t (C\Delta_P^2/d^2)}$ for $t \ge 2Cd^3/\Delta_P^2$. Take $C_d = \max(C, c_d^{-d/2}).$
\end{proof}
\end{mylem}

\subsubsection{Proof of Bounds on Rejection Times}\label{appx:rate_bound_proof}

With the above in hand, we may proceed with the main argument.

\begin{proof}[Proof of Theorem~\ref{thm:rate}]
    Recall the definition $\sigma_t := \sum_{s \le t} \log p(X_s) - \log \hat{p}_t(X_s).$ We shall first lower bound $\sigma_t$ with high probability.
    
    Let $\boundparam$ be a quantity that we will choose later. Let $\varepsilon_t$ denote the solution to the fixed point equation from Lemma~\ref{lem:wong_shen}, instantiated with the bracketing entropy of $\lc_{d,\boundparam}$. Further, let define the event \[ \mathsf{E}_t := \{ \hat{p}_t \in \lc_{d,\boundparam} \}. \] 
        For any $t,$ provided that such that $ \varepsilon_t \le d_H(p, \lc)$ and $\hat{p}_t \in \lc_{d,\boundparam},$ Lemma \ref{lem:wong_shen} yields that  \begin{equation}\label{eq:lower_bound_on_sigma_under_entropy_scale} \sigma_t \ge \inf_{q \in \lc_{d, \boundparam}: \\ d_H(p,q) \ge d_H(p, \lc)} \log \prod_{s \le t} \frac{p(X_s)}{q(X_s)} \ge \frac{t d_H^2(p, \lc)}{24}\end{equation} 
    with probability at least $ 1 -  \exp{-Ct d_H^2(p, \lc)} - P^\infty(\mathsf{E}_t^c).$
    
    In the rest of the proof, we will determine the range of $t$ that leads to a small enough value for $\varepsilon_t$ to ensure that the condition  $\epsilon_t \leq d_H(p,\lc)$ is met and, at the same time, control $P^\infty(\mathsf{E}_t^c)$. To this end, we deploy Lemma~\ref{lem:lc_entropy_bound}. First, observe that for $d \ge 3$ and for any positive constants $c$ and $C$ \[ \int_{c\varepsilon^2}^{C\varepsilon} \sqrt{\tilde{O}(\boundparam^{d-1}\zeta^{-(d-1)})} \mathrm{d}\zeta = \widetilde{O}( \boundparam^{(d-1)/2} \varepsilon^{-(d-3)}). \] Note that polylogarithmic terms do not affect the main growth of the integral.\footnote{This can be seen by iterating the relation $\int x^n \log^m x = \frac{x^{n+1} \log^mx}{n+1} - \frac{m}{n+1} \int x^n \log^{m-1}(x).$} Therefore, solving the fixed point equation \[ \widetilde{O}(\boundparam^{(d-1)/2} \varepsilon^{-(d-3)}) = \varepsilon^2 t^{1/2},\] we obtain that for $d \ge 3$ \[ \varepsilon_t(\boundparam) = \widetilde{O}(\boundparam^{1/2} t^{-1/2(d-1)}),\] where we highlight the dependence on the as yet undetermined quantity $\boundparam$.
    
    A similar argument using the entropy bound $ \zeta^{-d/2}$ yields $\varepsilon_t(\boundparam) = \widetilde{O}( \boundparam^{d/(d+4)}t^{-2/(d+4)})$ for $d \in \{1,2\}$. Now define \[ T_1(\boundparam) = \inf\{t : \varepsilon_t(\boundparam) \le d_H(p,\lc)\}\]  and observe that \[ T_1(\boundparam) = \begin{cases} \widetilde{O}( \boundparam^{(d-1)}(d_H(p, \lc))^{-2(d-1)}) & d \ge 3 \\ \widetilde{O}( \boundparam^{d/2} (d_H(p,\lc)^{-(4+d)/2})) & d \in \{1,2\}\end{cases}. \]   
    
    Finally, by Lemma~\ref{lem:mle_is_controlled}, for $\boundparam \ge C_d\Delta_P^{-d}$ and $t \ge T_2 := C \Delta_P^2/d^2$ the probability of the event $\mathsf{E}_t$ is at least $1 - 2\exp{-t C\Delta_P^2/d^2}.$ Let us set $\boundparam_* = C_d\Delta_P^{-d}$ and let $T_0 := \max(T_1(\boundparam_*), T_2).$ We obtain that the lower bound \[ \sigma_t \ge \frac{t d_H^2(p, \lc)}{24}\] holds with probability at least $ 1 - C \exp{- t c d_H^2(p, \lc)} - C \exp{-t c\Delta_P^2/d^2}$ for $t \ge T_0$.    
    Now, observe that at any time $t \ge \max(T_0 , \frac{600\log(1/\alpha)}{d_H^2(p,\lc)}),$ it holds with probability at least $1 - \pi_t  - C \exp{ - t d_H^2(p,\lc)} - C \exp{-C t \Delta_P^2/d^2}$ that \[\log R_t =  \sigma_t - \rho_t  \ge \frac{t d_H^2(p, \lc)}{600} \ge \log(1/\alpha), \] and thus the probability that the rejection time $\tau_\alpha := \inf\{t : R_t \ge 1/\alpha\}$ exceeds the above bound is bounded by $\pi_t + C \exp{-t d_H^2(p, \lc)} + C \exp{-Ct\Delta_P^2/d^2}$.
\end{proof}

\subsubsection{Proof of Bracketing Entropy Bound on Log-Concave Laws}\label{appx:lc_entropy_bound_proof}

We proceed to show Lemma~\ref{lem:lc_entropy_bound}. We note that the upper bound for $d \le 3$ was shown by Kim and Samworth \cite{kim2016global}. Below we focus on $d \ge 4$. We shall exploit two existing results in the literature regarding convex sets and functions. The first is essentially due to Bronshtein (also see \cite[Lem.~3]{kur2019optimality}).
    
    \begin{mylem}\label{lem:bronshtein1976varepsilon}
         \emph{\cite{bronshtein1976varepsilon}} Let $\mathcal{K}_d$ denote the collection of convex sets in $[-1,1]^d$. For any $\zeta > 0,$ there exists a collection of pairs of convex sets $\mathcal{K}_{d,\zeta} \subset \mathcal{K}_d \times \mathcal{K}_d$ with $\log |\mathcal{K}_{d,\zeta}| = O(\zeta^{-(d-1)/2})$ such that \begin{itemize}
            \item Every $(\underline{K}, \overline{K}) \in \mathcal{K}_{d,\zeta}$ satisfies $\mathrm{Leb}_d( \overline{K}\setminus\underline{K}) \le \zeta.$
            \item For every $K \in \mathcal{K}_d,$ exists $(\underline{K}, \overline{K})\in \mathcal{K}_{d,\zeta}$ satisfying $\underline{K} \subset K \subset \overline{K}.$
        \end{itemize}
    \end{mylem}    
    
    In other words, the bracketing entropy of convex sets under the set difference metric is controlled at rate $(d-1)/2$. Importantly, the bracketing demonstrated above is proper. This result may be extended to the following bracketing entropy bound on convex functions as by Gao and Wellner.
    
    \begin{mylem}\label{lem:Thm-5.1-gao2017entropy}
        \emph{\cite[Thm.~1.5]{gao2017entropy}} Let $K$ be a convex set in $[-1,1]^d,$ and let $\mathcal{C}_{K,\boundparam}$ be the set of convex functions upper bounded by $\boundparam$ over $K$. Then the $L_2(K)$ bracketing entropy of $\mathcal{C}_{K,\boundparam}$ at scale $\zeta$ is bounded as $O((\boundparam /\zeta)^{(d-1)}).$
    \end{mylem}
    Above, the $L_2(K)$ metric is the usual $L_2$ distance $\|f-g\|_{L_2(K)} = (\int_K (f-g)^2 \mathrm{d}x)^{1/2},$ and the $L_2(K)$ bracketing entropy is the bracketing entropy when the size of a bracket $[u,v]$ is  $|[u,v]| = \|u-v\|_{L_2(K)}$.
    
    With the above in hand, we may proceed with the proof.

\begin{proof}[Proof of Lemma~\ref{lem:lc_entropy_bound}]
    For any log-concave law $f$, let $S:= \{ x \in \mathbb{R}^d \colon f (x)\ge \zeta^3\} = \{x \in \mathbb{R}^d \colon \log f(x) \ge 3\log\zeta\}$. Since $f$ is log-concave, the set $S$ is convex. As a result, by Lemma~\ref{lem:bronshtein1976varepsilon}, there exists some convex set $\tilde{S} \in \mathcal{K}_{d,\zeta^2/\boundparam}$ such that \( \mathrm{Leb}(S \setminus \tilde{S}) \le \zeta^2/\boundparam \) and $\tilde{S} \subset S$. Let $\tilde{\mathcal{C}}_{\tilde{S}, \zeta, \boundparam}$ denote a $\zeta$-bracketing of convex functions bounded by $\boundparam$ on $\tilde{S}$. Since, on $\tilde{S}$, the function $-\log f$ is convex  and is upper bounded by $-\log \boundparam$, by Lemma~\ref{lem:Thm-5.1-gao2017entropy} there exists a bracket $[-u, -l] \in \tilde{\mathcal{C}}_{\tilde{S}, \zeta/\boundparam, \log \boundparam/\zeta^3 }$ such that, on $\tilde{S},$ \( l \le \log f \le u\), and $ \int_{\tilde{S}} (u(x)-l(x))^2 \mathrm{d}x \le \zeta^2/\boundparam^2$. Note that, on $\tilde{S}$, $f$ is lower bounded by $-3\log \zeta$  and that,  without loss of generality,  we may  assume that $\sup_{x \in \tilde{S}}u(x) \le \log \boundparam,$ since this is already a pointwise upper bound on $f.$ 
    
    Next, we construct the functions \begin{align*}
       x \in [-1,1]^d \mapsto  U(x) &:= \begin{cases} e^{u(x)} & x \in \tilde{S} \\  \boundparam & x \in S\setminus \tilde{S} \\ \zeta^3 & x \in [-1,1]^d \setminus S \end{cases},\\
      x \in [-1,1]^d \mapsto   L(x) &:= \begin{cases} e^{l(x)} & x \in \tilde{S} \\ \zeta^3 & x \in S \setminus \tilde{S} \\ 0 & x \in [-1,1]^d\setminus S  \end{cases}. 
    \end{align*}
    
Observe that $U \ge f \ge L$ on $[-1,1]^d$. Furthermore, for $\zeta < 2^{-d},$ \begin{align*}
        \int (\sqrt{U} - \sqrt{L})^2 \mathrm{d}x &= \int_{\tilde{S}} (e^{u(x)/2} - e^{l(x)/2})^2 \mathrm{d}x + \int_{S \setminus \tilde{S}} \boundparam \mathrm{d}x + \int_{[-1,1]^d \setminus S} \zeta^{3} \mathrm{d}x \\
        &\le \int_{\tilde{S}} e^{u(x)}(1 - e^{u(x) - l(x)/2})^2 \mathrm{d}x + \boundparam \cdot \zeta^2/ \boundparam + 2^d \zeta^3 \\
        &\le \int_{\tilde{S}} \boundparam^2 (u(x) - l(x))^2/4 \, \mathrm{d}x + 2\zeta^2 \\
        &\le \boundparam^2 \cdot \zeta^2/\boundparam^2 + 2\zeta^2 = 3\zeta^2,
    \end{align*} 
    where we have exploited the fact that $z \mapsto e^{z/2}$ is Lipschitz on $[-\infty, \log C],$ with derivative bounded by $e^{(\log \boundparam)/2}/2 = \sqrt{\boundparam}/2$ to argue that $ e^{(u(x) - l(x))/2} - e^{0} \le \sqrt{\boundparam} |u(x) - l(x) - 0|/2$. 
        
    Since this construction can be carried out for any $f$, we conclude that we can construct a bracketing cover of $\lc_{d,\boundparam}$ at scale $O(\zeta)$ as the union of the bracketing covers of convex functions on each of the smaller sets in $\mathcal{K}_{d, \zeta^2/\boundparam}$. By Lemmas~\ref{lem:bronshtein1976varepsilon} and~\ref{lem:Thm-5.1-gao2017entropy}, the size of this cover is  
    \[
    \exp{ O((\boundparam/\zeta)^{d-1})} \cdot \exp{O( (\log \boundparam/\zeta^3)(\boundparam)/\zeta)^{d-1}} = \exp{ \widetilde{O}((\boundparam/\zeta)^{d-1})},
    \]
    and the claim now follows. Let us again observe that the resulting cover is improper, in that the maps $U(x)$ and $L(x)$ are not log-concave. 
\end{proof}

\subsection{Regret Control for Bounded Lipschitz Laws on the Unit Box}\label{appx:bounded_unit_box_regret}

As this subsection demonstrates, both Corollaries~\ref{thm:rate_for_box} and~\ref{prop:consistency_box} rely on arguing that laws in $\mathcal{D}_{\mathrm{Box, Lip, \boundparam}}$ can be estimated in a low-regret manner online. We argue this by exploiting the following result, which follows as a simplification of the results of Wong and Shen on sieve estimators.

\begin{mylem}\label{lem:kl_bound}
    \emph{(Adaptation of \cite[Cor.~1 \& Thm.~6]{wong1995probability})}For every $P \in \mathcal{D}_{\mathrm{Box,Lip,\boundparam}}$ and $t \geq 1,$ there exists a sieve MLE $\hat{q}(\cdot) = \hat{q}(\cdot ;X_1^t)$ and a constant $A > 1$ depending only on $\boundparam$ such that for every $\zeta \ge \zeta_t,$ \[ P^\infty\left(\mathrm{KL}(p\|\hat{q}) > \frac{1}{A} \zeta^2\log(1/\zeta) \right) \le A\exp{ -t\frac{\zeta^2}{A\log(1/\zeta)}}, \] where $\zeta_t = \widetilde{O}(t^{-1/2(d+2)}).$

\begin{proof}
    The cited results of Wong and Shen apply because densities of laws in $\mathcal{D}_{\mathrm{Box,Lip,\boundparam}}$ are uniformly upper bounded. This directly yields the entirety of the statement, barring the scale bound on $\zeta_t$. This scale is determined by the same entropy integral fixed point equation that appears in Lemma~\ref{lem:wong_shen}, and for this instance, the bound can be derived by using the standard fact that the Hellinger bracketing entropy of Lipschitz functions on a box at scale $\eta$ are controlled as $O( \eta^{-(d+1)})$ \cite{van1994bracketing}.
\end{proof}
\end{mylem}

The sieve estimators in this result can be taken with a fair bit of lassitude. In particular, one explicit choice is to construct for each $\zeta > 0$ a bracket of the class $\mathcal{D}_{\mathrm{Box, Lip, \boundparam}}$ at scale $\zeta,$ and choose a representative density within each bracket of the class. The sieve MLE then involves choosing a $\zeta$ at each time, and estimating the law as the maximiser of likelihood amongst the aforementioned representative densities. Importantly for us, the lower brackets in these bracketings can be taken to be uniformly larger than $1/\boundparam,$ and the upper brackets smaller than $\boundparam,$ since $p \in [1/\boundparam, \boundparam],$ and as a result the sieve estimates are uniformly bounded between $1/\boundparam$ and $\boundparam$.

Below we first show Corollary~\ref{thm:rate_for_box} using the above results, and then show Corollary~\ref{prop:consistency_box} follows as a simple consequence of this argument.

\begin{proof}[Proof of Corollary~\ref{thm:rate_for_box}]
    As argued in the main text, the expected rejection time is bounded as $\mathbb{E}[\tau] \le \sum \pi_t + O(T_0),$ where $T_0 = o(d_H(p,\lc)^{-2(d+3)}).$ We thus only need to show that a sequence $\pi_t$ exists such that $\sum \pi_t$ is appropriately small, and that for any $t,$ \[ P^\infty\left(  \frac{\rho_t(\mathscr{E};p)}{td_H^2(p,\lc)} \ge \frac{1}{25}\right) \le \pi_t,\] where $p \in \mathcal{D}_{\mathrm{Box,Lip,\boundparam}}$ and $\mathscr{E}$ are sieve estimators. We proceed to do so below.
    
    For succinctness, we shall define $\varepsilon = d_H(p,\lc)$. Let $A$ be the constant from Lemma~\ref{lem:kl_bound}, and set \[T_1 := \min\{t : \zeta_t^2\log(1/\zeta_t)/A < \varepsilon^2/200, \zeta_t < 1/\sqrt{e}\}.\] Further let \[ \zeta(\varepsilon) := \max\{ \zeta \in [0,1/\sqrt{e}] : \zeta^2 \log(1/\zeta) \le A\varepsilon^2/200\}.\] In the subsequent proof, we shall use Lemma~\ref{lem:kl_bound} with $\zeta = \zeta(\varepsilon) \ge \zeta_{T_1}$. To this end, we note that if  $\zeta(\varepsilon) < 1/\sqrt{e} \iff A\varepsilon^2/200 < 1/2e,$ and in this case the equality $\zeta(\varepsilon)^2 \log(1/\zeta(\varepsilon)) = A\varepsilon^2/200$ holds. From this, we may derive\footnote{This equation is equivalent to $x \log x = y$ for $x = \zeta(\varepsilon)^2, y = A\varepsilon^2/100$ in the range $0<x < 1/e$. The claim follows by noting that the map $x \mapsto x \log(1/x)$ is monotonically increasing on [0,1/e], and verifying that for $y \in [0,1/2e], \frac{y}{2\log(1/y)} \cdot  \log(2\log(1/y)/y) < y.$ Indeed, this inequality is equivalent to arguing that $\log(2 \log(1/y)) < \log(1/y) \iff y\log(1/y) < 1/2,$ which holds since the maximum value of $y \mapsto y \log(1/y)$ is $1/e < 1/2$.} that $\zeta(\varepsilon)^2 > a \varepsilon^2/\log(1/\varepsilon)$ for some small enough constant $a$, and so that the exponent of the upper bound of Lemma~\ref{lem:kl_bound} is  
    \[\frac{\zeta(\varepsilon)^2}{A\log(1/\zeta(\varepsilon))} = \frac{400 \zeta^4}{A^2 \varepsilon^2} \ge  \frac{\varepsilon^2}{A' \log(1/\varepsilon)}\] for some large enough constant $A'.$ We shall also assume that $A' \ge \max(1,A).$ 
    
    Let $\mathscr{E}$ be a choice of sieve estimators such that for every $t, x \in [-1,1]^d, \hat{q}_{t-1}(x) \in [1/\boundparam,\boundparam],$ which can be ensured due to the discussion above. Notice, by the independence of the data $\{X_t\},$ that for any $t$, \[ \mathbb{E}[\log p(X_t)/\hat{q}_{t-1}(X_t) | \mathscr{F}_{t-1}] = \mathrm{KL}(p \|\hat{q}_{t-1}).\]

    Let $\theta \in (0,1)$ and $M \ge 0$ be two parameters of argument that we shall set later. Let us consider the case of $t = T_1 + \tau$ for some $\tau \ge M T_1$. 
    
    Since for each $\tau > 0, \zeta_{T_1 + \theta \tau} \le \zeta_{T_1} \le \zeta(\varepsilon),$ the bound of Lemma~\ref{lem:kl_bound} is effective at each time $s \in [T_1 + \theta\tau : T_1 + \tau]$ with $\zeta = \zeta(\varepsilon).$ As a result, applying Lemma~\ref{lem:kl_bound} to each $s$ in this range, and exploiting the behaviour of $\zeta(\varepsilon)^2$ established above, \[ P^\infty( \mathrm{KL}(p\|\hat{q}_{s-1}) > \varepsilon^2/200) \le A' \exp{-  s \frac{\varepsilon^2 }{A' \log(1/\varepsilon)}}. \] 

    Next, by applying the union bound over $s \in [T_1 + \theta\tau: T_1 + \tau]$ in the above result, we conclude that \begin{align*}
        &P^\infty\left(\exists s \in [T_1+\theta \tau : T_1 + \tau] : \mathrm{KL}(p\|\hat{q}_{s-1}) > \frac{\varepsilon^2}{200}\right) \\ &\qquad\le \sum_{s = T_1 + \theta\tau}^{T_1 + \tau} A' \exp{- s \frac{\varepsilon^2 }{A' \log(1/\varepsilon)} } \\ 
                     &\qquad= A' \exp{-(T_1 + \theta\tau) \frac{\varepsilon^2 }{A' \log(1/\varepsilon)}} \cdot \frac{1}{1-\exp{-\varepsilon^2/(A'\log(1/\varepsilon)}},\\
                     &\qquad\le \frac{A'^2 \log(1/\varepsilon)}{\varepsilon^2} \exp{- \theta \tau \frac{\varepsilon^2}{A' \log(1/\varepsilon)}}.
    \end{align*}
    where the equality sums over the geometric series, and the final inequality uses that $T_1 \ge 0$ and that for $u < 1, 1/(1-e^{-u}) \le \frac{2}{u}$.

    Next, observe that since for any $x\in [-1,1]^d,$  $\frac{1}{\boundparam^2} \le \frac{p(x)}{\hat{q}_{t-1}(x)} \le\boundparam^2,$ we have the bound $|\log(p(X_t)/\hat{q}_{t-1}(X_t))| \le 2\log \boundparam.$ Therefore, the Azuma-Hoeffding inequality is applicable, and yields that for every $\tau \ge 1, \delta > 0$ \[ P^\infty\left( \sum_{s = T_1 + \theta \tau }^{T_1 + \tau} \log \frac{p(X_s)}{\hat{q}_{s-1}(X_s)} > \sum_{s = T_1+\theta\tau }^{T_1 + \tau} \mathrm{KL}(p\|\hat{q}_{s-1}) + (\tau - \theta\tau) \delta \right) \le \exp{- (\tau - \theta \tau)\delta^2/8\log^2\boundparam}.\]

    We proceed by setting $\delta = \varepsilon^2/200$ in the above, and applying the union bound, to conclude that there exists a constant $C$ such that \begin{equation}\label{ineq:regret_tail_for_large_t}
        P^\infty\left( \sum_{s = T_1+\theta \tau}^{T_1 + \tau} \log \frac{p(X_s)}{\hat{q}_{s-1}(X_s)} > \frac{(1-\theta) \tau \varepsilon^2}{100}\right) \le \exp{- \frac{(1-\theta)\tau \varepsilon^4}{C \log^2\boundparam}} + \frac{C\log(1/\varepsilon)}{\varepsilon^2}\exp{-\theta\tau \frac{\varepsilon^2}{C \log(1/\varepsilon)}}.\end{equation} 

    Let us call the right hand side of (\ref{ineq:regret_tail_for_large_t}) $\pi(\tau, \theta)$. By the definition of $\rho_t,$ and the boundedness of $\log \frac{p(x)}{\hat{q}_{s-1}(x)}$ for every $s$, it follows that with probability at least $1 - \pi(\tau, \theta),$ \begin{align*}
        \rho_{T_1 + \tau}(\mathscr{E};p) &= \sum_{s = 1}^{T_1 + \tau} \log \frac{p(X_s)}{\hat{q}_{s-1}(X_s)} \\
                                         &\le 2(T_1 + \theta \tau) \log \boundparam + \frac{(1-\theta)\tau\varepsilon^2}{100}.
    \end{align*} 

    So long as we can choose $\theta, M$ such that the upper bound above is smaller than $(\tau + T_1) \varepsilon^2/25,$ the inequality (\ref{ineq:regret_tail_for_large_t}) will limit the probability that $\rho_{T_1 + \tau} > \varepsilon^2(T_1 + \tau)/25,$ which is precisely our goal. But observe that this indeed occurs if $(\theta + 1/M) \le 3 \varepsilon^2/(200 \log \boundparam),$ since in such a case \begin{align*}
        2(T_1 + \theta\tau) \log \boundparam + \frac{(1-\theta) \tau \varepsilon^2}{100} &\le \tau \left( 2(1/M + \theta)\log\boundparam + \frac{\varepsilon^2}{100}\right) \\
        &\le \tau \left( \frac{3 \varepsilon^2}{200 \log \boundparam} \cdot 2 \log\boundparam + \frac{\varepsilon^2}{100} \right)\\ 
        &= \frac{\tau \varepsilon^2}{25} \le \frac{(T_1 + \tau)\varepsilon^2}{25}.
    \end{align*}

    So, we may set $\theta = \min(1/2, \varepsilon^2/(100\log\boundparam))$ and $M = \max(1,(200 \log \boundparam)/\varepsilon^2),$ and conclude that for any $\tau \ge MT_1$, it holds that \[ P^\infty(\rho_{T_1 + \tau}(\mathscr{E};p)/t\varepsilon^2 > 1/25) \le \pi(\tau),\] where, for a constant $C',$ \[\pi(\tau) = \exp{- \frac{\tau \varepsilon^4}{C'\log^2\boundparam}} + \frac{C'\log(1/\varepsilon)}{\varepsilon^2} \exp{-\frac{\tau \varepsilon^2}{C' \log(1/\varepsilon)\cdot \log\boundparam}}. \]

    Note that in the terminology of Theorem~\ref{thm:rate}, $\pi_t = \pi(t - T_1)$ for $t \ge (M+1)T_1$. Of course we can always provide the trivial bound $\pi_t \le 1$ for $t < (M+1)T_1.$ It remains to compute the resulting bound on expected rejection time. To this end, observe by summing the appropriate geometric series that \begin{align*} \sum_{t \ge 1} \pi_t &\le (M+1)T_1 + \sum_{\tau = MT_1}^\infty \pi(\tau) \\
    &\le (M+1)T_1 + \frac{1}{1 - \exp{- \varepsilon^4/C'\log^2\boundparam}} + \frac{C'\log(1/\varepsilon)}{\varepsilon^2(1 - \exp{-\varepsilon^2/C'\log(1/\varepsilon) \cdot \log\boundparam}} \\
    &\le O\left(\frac{1}{\varepsilon^2}\right) T_1 + \widetilde{O}\left( \frac{1}{\varepsilon^4} \right),\end{align*}
    where the $O$ bounds are as $\varepsilon \to 0$, and we have hidden the dependence on $\boundparam$ and $\log(1/\varepsilon)$. But, since in Lemma~\ref{lem:kl_bound}, $\zeta_t = \widetilde{O}(t^{-1/2(d+2)}),$ and since $T_1$ is the first time that $\zeta_t^2\log^2(1/\zeta_t) \le A \varepsilon^2/200,$ we may conclude that $T_1 = \widetilde{O}(\varepsilon^{-2(d+2)}).$ The claim follows upon noticing that $O(\varepsilon^{-2}) \cdot T_1 = \widetilde{O}(\varepsilon^{-2(d+3)}),$ and recalling that $\varepsilon = d_H(p,\lc)$.
\end{proof}

We conclude with a brief proof of Corollary~\ref{prop:consistency_box} that exploits the bounds developed in the argument above.
\begin{proof}[Proof of Corollary~\ref{prop:consistency_box}]
    It suffices to argue that using the estimators in the proof of Corollary~\ref{thm:rate_for_box}, for any $P \in \mathcal{D}_{\mathrm{Box,Lip,\boundparam}},$ \[ P^\infty\left(\limsup \frac{\rho_t(\mathscr{E};p)}{td_H^2(p,\lc)} \le \frac{1}{25}\right) = 1. \] This follows since for each $t$, \[ P^\infty\left( \frac{\rho_t(\mathscr{E};p)}{td_H^2(p,\lc)} > \frac{1}{25} \right) \le \pi_t,\] and $\sum \pi_t < \infty,$ which yields precisely the above relation by the Borel-Cantelli Lemma. \qedhere
\end{proof}

\end{document}